%% file: neumann.tex
\documentclass[reqno,a4paper]{amsart}

\usepackage[utf8]{inputenc}
\usepackage{latexsym,amssymb,amsmath,soul}
\usepackage[T1]{fontenc}
\usepackage{times,mathptmx,bbm}
\usepackage{ifthen}
\usepackage{graphicx,tikz,color}
\usepackage{cite}
\usepackage{hyperref}
\usepackage{enumerate}
\usetikzlibrary{external}
\makeatletter

\providecommand{\tikzsetnextfilename}[1]{}

\AtBeginDocument{
  \@ifpackageloaded{color}{%
    \definecolor{umons-red}{cmyk}{0., 1., 0.6, 0.2}
    \definecolor{umons-gray}{cmyk}{0, 0, 0, 0.5}
    \definecolor{umons-turquoise}{cmyk}{0.9, 0., 0.2, 0.}

    \definecolor{n2}{RGB}{220,76,76}
    \definecolor{n3}{RGB}{220,157,76}
    \definecolor{n4}{RGB}{94,172,59}
    \definecolor{n5}{RGB}{76,186,220}
    \definecolor{n6}{RGB}{64,85,255}
    \definecolor{n7}{RGB}{116,76,220}

    \definecolor{b2+}{named}{umons-red}%
    \definecolor{b2-}{named}{b2+}
    \definecolor{b3+}{RGB}{50, 90, 235}%
    \definecolor{b3-}{named}{b3+}
    \definecolor{b4+}{RGB}{10, 170, 15}%
    \definecolor{b4-}{named}{b4+}
    \definecolor{b5+}{RGB}{215, 155, 70}%
    \definecolor{b5-}{named}{b5+}
    \definecolor{b6+}{RGB}{195, 30, 230}%
    \definecolor{b6-}{named}{b6+}
    \definecolor{b7+}{RGB}{98, 137, 156}%
    \definecolor{b7-}{named}{b7+}
    \definecolor{b8+}{RGB}{189, 212, 147}%
    \definecolor{b8-}{named}{b8+}
    \definecolor{b9+}{RGB}{170, 140, 130}%
    \definecolor{b9-}{named}{b9+}
  }{}%
}

\@ifpackageloaded{hyperref}{%
  \hypersetup{%
    pdfauthor={Denis Bonheure, Christopher Grumiau, Christophe Troestler},
    pdfsubject={Lane Emden problem with Neumann boundary
      conditions: symmetry breaking, bifurcations and multiplicity
      of positive solutions},
    pdfpagemode=UseNone,
    pdfstartview=FitH,
    colorlinks=true,
    urlcolor=umons-turquoise,
    citecolor=umons-red,
    linkcolor=umons-red
  }}{}

\providecommand{\texorpdfstring}[2]{#1}

\IfFileExists{graphs/c.tex}{%
  \newcommand{\graphpath}{graphs/}
}{%
  \newcommand{\graphpath}{}}

\theoremstyle{plain}
\newtheorem{Prop}{Proposition}[section]
\newtheorem{Thm}[Prop]{Theorem}
\newtheorem{Lem}[Prop]{Lemma}
\newtheorem{Cor}[Prop]{Corollary}
\newtheorem{Conj}[Prop]{Conjecture}
\theoremstyle{definition}
\newtheorem{Def}[Prop]{Definition}
\theoremstyle{remark}
\newtheorem{Rem}[Prop]{Remark}

\numberwithin{equation}{section}

\newcommand{\IN}{\mathbb{N}}

\newcommand{\IR}{\mathbb{R}}
\newcommand{\IS}{\mathbb{S}}
\newcommand{\C}{\mathcal{C}}

\newcommand{\N}{{{\mathcal N}}}

\newcommand{\E}{{{\mathcal E}}}

\newcommand{\abs}[1]{\mathopen\vert#1\mathclose\vert}
\newcommand{\bigabs}[1]{\bigl| #1\bigr|}

\newcommand{\norm}[1]{\mathopen\Vert#1\mathclose\Vert}
\newcommand{\intervalcc}[1]{\mathopen[#1\mathclose]}
\newcommand{\intervaloc}[1]{\mathopen]#1\mathclose]}
\newcommand{\intervalco}[1]{\mathopen[#1\mathclose[}
\newcommand{\intervaloo}[1]{\mathopen]#1\mathclose[}
\newcommand{\bigintervaloo}[1]{\bigl]#1\bigr[}

\newcommand{\intd}{\,{\mathrm d}}
\renewcommand{\phi}{\varphi}
\let\epsilon=\varepsilon
\@ifundefined{leqslant}{}{%
  \renewcommand{\le}{\leqslant}%
  \renewcommand{\leq}{\leqslant}}
\@ifundefined{geqslant}{}{%
  \renewcommand{\ge}{\geqslant}%
  \renewcommand{\geq}{\geqslant}}
\@ifundefined{varnothing}{}{}
\renewcommand{\subset}{\subseteq}

\newcommand{\spanned}[1]{\mathop{\operator@font span}\{#1\}}
\newcommand{\closure}[1]{\overline{#1}}
\newcommand{\limplies}{\Rightarrow}
\newcommand{\liff}{\Leftrightarrow}
\DeclareMathOperator{\id}{\mathbbm{1}}

\DeclareMathOperator{\Ran}{Im}
\DeclareMathOperator{\codim}{codim}
\DeclareMathOperator{\dist}{dist}
\DeclareMathOperator{\e}{e}
\DeclareMathOperator{\sign}{sign}
\DeclareMathOperator{\arcsinh}{arcsinh}

\DeclareMathOperator{\order}{o}
\newcommand{\MIrad}{\operatorname{MI_{rad}}}
\newcommand{\Underline}[1]{\underline{\smash{#1}}}

\newcommand{\lambdarad}{\lambda^{\text{\upshape rad}}}

\renewcommand{\paragraph}{%
  \@startsection{paragraph}{4}{\z@}{1ex}{-\fontdimen 2\font }%
  {\normalfont \itshape \bfseries}}

\makeatother

\begin{document}

\title[Semilinear elliptic problems with Neumann boundary conditions]{%
Multiple radial positive solutions of semilinear elliptic problems
with Neumann boundary conditions}
\author[D.~Bonheure, C.~Grumiau, C.~Troestler]{Denis Bonheure \and
  Christopher Grumiau \and Christophe Troestler}

\address{
  D\'epartement de Math{\'e}matique\\
  Universit\'e libre de Bruxelles, CP 214\\
  Boulevard du Triomphe, B-1050 Bruxelles, Belgium}
\email{denis.bonheure@ulb.ac.be}

\address{
  Institut Complexys\\
  D\'epartement  de Math\'ematique\\
  Universit\'e de Mons\\
  20, Place du Parc, B-7000 Mons, Belgium}
\email{Christopher.Grumiau@umons.ac.be}
\email{Christophe.Troestler@umons.ac.be}

\begin{abstract}
Assuming $B_{R}$ is a ball in $\mathbb R^{N}$, 
we analyze the positive solutions of the problem
\begin{equation*}
\left\{
    \begin{aligned}
      -\Delta u+u&= \abs{ u}^{p-2}u,&&\text{ in } B_{R}, \\
      \partial_{\nu}u&=0,&&\text{ on } \partial B_{R},
    \end{aligned}
  \right.
\end{equation*}
that branch out from the constant solution $u=1$ as $p$ grows from $2$
to $+\infty$.
The non-zero constant positive solution is the unique positive
solution for $p$ close to $2$.  We show that there exist arbitrarily
many positive solutions as $p\to\infty$ (in particular, for
supercritical exponents) or as $R \to \infty$ for any fixed value of
$p>2$, answering partially a conjecture in \cite{BonheureNorisWeth}. We give the explicit lower bounds for $p$ and $R$ so that
a given number of solutions exist.
The geometrical properties of those solutions
are studied and illustrated numerically.  Our simulations
motivate additional conjectures.  The structure of the least energy
solutions (among all or only among radial solutions) and other related
problems are also discussed.
\end{abstract}

\keywords{Neumann boundary conditions, bifurcation, subcritical and
supercritical exponent, Lane Emden problem, boundary  value problems, Nehari
manifold, symmetries, clustered layer solutions}

\thanks{Work performed within the programs
  ``Qualitative study of solutions of variational elliptic partial
differerential equations. Symmetries, bifurcations, singularities,
  multiplicity and numerics'', project 2.4.550.10.F,
  and ``Existence and asymptotic behavior of solutions to systems of semilinear
  elliptic partial differential equations'',
  project T.1110.14,
  of the FNRS,
  Fonds de la Recherche Fondamentale Collective.
  The first author is supported by INRIA~-- Team MEPHYSTO,
MIS F.4508.14 (FNRS)  \& ARC AUWB-2012-12/17-ULB1-IAPAS.
The last two authors are also partially supported by a grant from the national
bank of Belgium.
}

\maketitle
\tableofcontents

\section{Introduction} \label{Section-Intro}

In this paper, we consider the
semilinear elliptic problem 
\begin{equation*}
\tag{\protect{$\mathcal{P}_{\lambda,p}$}} \label{pblP} \left\{
\begin{aligned}
-\Delta u+\lambda u&= \abs{ u}^{p-2}u,&&\text{ in } \Omega, \\
u&>0, &&\text{ in } \Omega, \\
\partial_{\nu}u&=0,&&\text{ on } \partial \Omega,
\end{aligned}
\right.
\end{equation*}
where $\Omega$ is a smooth bounded domain in $\IR^N$, $N\geq 3$, $\lambda>0$,
$p>2$ and $\partial_\nu$ denotes the outward normal
derivative. 
This problem, sometimes referred to as the Lane-Emden equation with
Neumann boundary conditions, arises for instance in mathematical
models which aim to study pattern formation, and more specifically in
those governed by diffusion and cross-diffusion
systems~\cite{ni98}. The problem is also related to the stationary
Keller-Segel system in
chemotaxis~\cite{keller-segel,Gierer-Meinhardt,keller-segel2,Lin-Ni-Takagi88}.

As \eqref{pblP} admits a constant solution, the
solvability of \eqref{pblP} differs from
the case of positive solutions of the Lane-Emden equation with
Dirichlet boundary conditions
\begin{equation}\label{Diri}
  \left\{
    \begin{aligned}
      -\Delta u &= \abs{u}^{p-2}u,&&\text{ in } \Omega, \\
      u&>0, &&\text{ in } \Omega, \\
      u&=0,&&\text{ on } \partial \Omega,
    \end{aligned}
  \right.
\end{equation}
for which it is well known that, if $\Omega$ is starshaped and $N\ge
3$, existence is restricted to the subcritical range
\begin{equation}
  \label{eq:6}
  p<2^{*}:=\frac{2N}{N-2}
\end{equation}
as a consequence of Pohozaev's identity
(see \cite{Pohozaev}). 
In the sequel of the paper we set $2^{*}=+\infty$ if $N=2$.

The subcriticality assumption (\ref{eq:6}) allows to tackle the
problem \eqref{pblP} with variational methods, i.e.,
the equation arises as the Euler-Lagrange equation of the energy functional
\begin{equation*}
  \mathcal{E}_{\lambda,p}:H^1(\Omega)\to\IR:u\mapsto\frac{1}{2}\int_{\Omega}
  \abs{\nabla u}^2+\lambda u^2-\frac{1}{p}\int_{\Omega}\abs{u}^p.
\end{equation*}
Moreover, due to the compact embedding
$H^1(\Omega) \hookrightarrow L^{p}(\Omega)$, the existence of a
solution to~\eqref{pblP}
follows by standard arguments. 
Indeed, it is enough to minimize $\E_{\lambda,p}$ on the Nehari manifold
\begin{equation*}
  \mathcal{N}_{\lambda,p}
  := \bigl\{u\in H^1\setminus\{0\} :  \E_{\lambda,p}'(u)[u]=0 \bigr\}
\end{equation*}
and to observe that the minimizer is nonnegative whereas the strong
maximum principle implies its positivity. The minimizers are called
least energy or ground state solutions. Looking at the quadratic form
$ \E_{\lambda,p}''(u_{0})[u,u]$, it is easily seen that any minimizer
$u_{0}$ is non constant if\footnote{In this paper $\lambda_{i}(\Omega)$
  ($i \ge 1$) stands for the $i$\thinspace th eigenvalue of $-\Delta$
  with Neumann
  boundary conditions on~$\partial\Omega$.}
$\lambda(p-2) >\lambda_{2}(\Omega)$.  On the other hand, if $\lambda$
is small, the only minimizer is the constant solution as Lin, Ni and
Takagi~\cite{Lin-Ni-Takagi88} proved that uniqueness holds for
\eqref{pblP} for $\lambda$ small.

In contrast with the nonexistence result for \eqref{Diri},
the energy functional for the critical exponent, $\E_{\lambda,2^{*}}$,
achieves its minimum on $\N_{\lambda,2^{*}}$. Moreover, Wang
\cite{wang91} proved that when $\lambda$ is sufficiently large, the
constant solution cannot be a minimizer.

For $\lambda$ small and $p=2^{*}$, Lin and Ni \cite{Lin-Ni}
conjectured that the constant
solution must be the unique solution. The conjecture was studied by
Adimurthi and Yadava \cite{Adimurthi-Yadava91, Adimurthi-Yadava97}
and Budd, Knapp and Peletier~\cite{Budd-Knapp-Peletier}
in the case of radial solutions when $\Omega$ is a ball. It happens
that in this case, the conjecture is true in dimension $N=3$ or
$N\ge 7$, while it is false in dimension $N=4, 5, 6$.
The conjecture was further extended to convex domains in dimension
$N=3$ and has lead to many developments in the recent years.
We refer to \cite{wangweiyan} and to the references therein for
further details.




In the supercritical range, namely when $p>2^{*}$, most of the
previous works on the existence of solutions of \eqref{pblP} are
devoted to perturbative cases where either $\lambda\to +\infty$ or a
slightly supercritical exponent $2^{*}+\varepsilon$ is considered, see
e.g. \cite{delpinomussopistoia,reywei1,reywei2}. By scaling, it is
easily seen that the case $\lambda\to +\infty$ amounts to consider a
small diffusion coefficient $\varepsilon$
(in front of $-\Delta$), see below. In this setting, it is physically
relevant to study the existence of solutions which concentrates
around a single or multiple points or even around some curve or a higher
dimensional manifold as $\varepsilon\to 0$, see for
example~\cite[Chapter~9 and~10]{Ambrosetti-Malchiodi}, \cite{nitakagi1,nitakagi2,MR2056434,Malchiodi-Ni-Wei05,MR3262455,MR2296306,MR2221246,MR2072213} and the references therein.

In this paper, we deal with \eqref{pblP} in a ``non-perturbative way'' and
therefore our contribution is more closely related to the recent works
\cite{Barutello-Secchi-Serra,BonheureSerra,GrossiNoris,SerraTilli,BonheureNorisWeth}. It
was observed in \cite{SerraTilli} that when $\Omega=B_{R}$ is a ball
of radius $R>0$, compactness can be recovered in the supercritical
case by considering the subspace of radially
\emph{increasing} functions of $H^1_{\text{rad}}(B_{R})$,
where $H^1_{\text{rad}}(B_{R})$ is the space of
functions of $H^1(B_{R})$ invariant under the action of the group
$O(N)$. This fact was used in \cite{BonheureNorisWeth} to prove the
existence of a non-constant radially increasing solution of
\eqref{pblP} in the supercritical regime, i.e., without assuming
\eqref{eq:6}, under the assumption that\footnote{In this paper,
  $\lambdarad_{i}(B_{R})$
  stands for the $i$\thinspace th eigenvalue of the operator $-\Delta$
  restricted to radial functions on $B_{R}$, with Neumann
  boundary conditions on~$\partial B_{R}$.}
$$\lambda(p-2)>\lambdarad_{2}(B_{R}).$$
In the critical case, the existence of such a radially
increasing solution has been proved using a shooting argument and the
Emden-Fowler transformation in~\cite{Adimurthi-Yadava91} under the
same assumption
\begin{equation*}
  \lambda(2^{*}-2) > \lambdarad_{2}(B_{R}).
\end{equation*}
This condition is satisfied if $R$ is large enough.
%
%


In our study of \eqref{pblP}, one of our main motivation is to understand to
what extent the precise value of $p$ plays a role in the existence and
qualitative properties of solutions.
Our main results are multiplicity of solutions with
respect to the value of the power $p$, without assuming
subcriticality.
It has been shown in \cite{BonheureNorisWeth,GrossiNoris,SerraTilli}
that for the Neumann problem \eqref{pblP} in a ball, no growth
restriction is needed to prove the existence of at least one non
constant solution. 
Since we deal with a simpler model than in the quoted references, we
are able to perform
here a refined analysis. Namely, we obtain non trivial solutions that
branch out from the constant solution, see Section~\ref{bifu}.
Combined with a priori estimates, this leads to the following
multiplicity result.
\begin{Thm}
  \label{intro1}
  Assume $\Omega=B_{R}$ , $N \ge 2$,
  $n\in \IN_0$, $p \in \intervaloo{2, +\infty}$ and
  $\lambda > 0$.
  \begin{enumerate}[(i)]
  \item   If $\lambda(p-2) >
    \lambdarad_{n+1}(B_R)$, then Problem~\eqref{pblP}
    has at least $n$ distinct non constant radial solutions.
\item  If $\lambda(p-2) >
    \lambdarad_{n+1}(B_R)$ and $p< 2^*$, then~\eqref{pblP} possesses at least
    $2n$ distinct non constant radial solutions.
  \item\label{mult:before:crit} If $\lambda(2^*-2) >
    \lambdarad_{n+1}(B_R)$ and $N \ge 3$,
    there exists $\varepsilon_{n,R}>0$ such that
    if
    $$ \lambdarad_{n+1}(B_R) - \varepsilon_{n,R} <  \lambda(p-2) < \lambdarad_{n+1}(B_R),$$
    then~\eqref{pblP} has at least
    $2n$ distinct non constant radial solutions.
  \end{enumerate}
\end{Thm}

This theorem implies the existence of arbitrarily many solutions for either large $p$ or large~$R$. We anticipate that the $n$ solutions
$u_{i}$ are distinguished by the number of nodal regions of $u_{i}-1$
(and also by the number of critical points). Indeed, the bifurcation
analysis shows that, given a positive integer $n$, \eqref{pblP} has at
least one radial solution $u$ such that $u-1$ has  $n$ nodal regions
provided $p > 2 + \lambdarad_{n+1}(B_R)/\lambda$
(see Section~\ref{bifu} for more details). We also anticipate that
the validity of~\eqref{mult:before:crit}
relies on a key estimate of Bessel's function (see
Lemma~\ref{lemma:integ>0}).  Numerical evidence shows it to be valid
in dimension $N=2$ but this is not formally proved.

\begin{figure}[ht]
  \centering
  \tikzpicturedependsonfile{\graphpath bifurcation-0.dat}
  \tikzsetnextfilename{bifurcation-diag}
  \newcommand{\critII}{2.62425}
  \newcommand{\critIII}{3.67692}
  \begin{tikzpicture}[x=10ex, y=2pt]
    \draw[->] (2,0) -- (2,25) node[left]{$u(0)$};
    \draw[dashed] (2,1) node[left]{$1$} -- (5,1);
    \draw[dashed] (4, 0) -- (4, 27) node[right]{$p = 2^*$};
    \draw[fill] (\critII,1) circle(1.7pt);
    \draw[fill] (\critIII,1) circle(1.7pt);
    \begin{scope}
      \clip (1.8, -10) rectangle (5.6,26);
      \draw[thick, smooth] plot file{\graphpath bifurcation-1.dat};
      \draw[thick, smooth] plot file{\graphpath bifurcation-3.dat};
    \end{scope}
    \begin{scope}[y=6pt, yshift=-4pt]
      \draw (2,0) -- (2,1);
      \draw[->] (2,0) -- (5.5, 0) node[below]{$p$};
      \draw (\critII, 0) ++ (0,2pt) -- ++(0,-4pt)
      node[below]{\small $2 + \lambdarad_2/\lambda$};
      \draw (\critIII, 0) ++ (0,2pt) -- ++(0,-4pt)
      node[below]{\small $2 + \lambdarad_3/\lambda$};
      \draw[line width=1pt, smooth] plot file{\graphpath bifurcation-0.dat};
      \draw[line width=1pt, smooth] plot file{\graphpath bifurcation-2.dat};
    \end{scope}
  \end{tikzpicture}

  \vspace*{-2ex}
  \caption{Radial bifurcation branches from $2 + \lambdarad_i/\lambda$,
    $i=2,3$ when
    $2 + \lambdarad_1/\lambda < 2 + \lambdarad_2/\lambda < 2^{*}$.}
  \label{fig:bifuc-simple}
\end{figure}
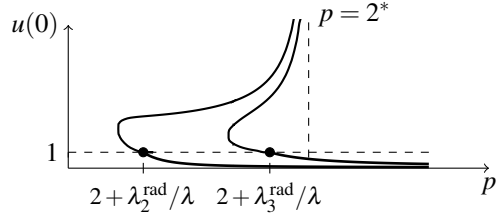

\medbreak

Theorem \ref{intro1} contrasts with the classical uniqueness result \cite{gnn}
of either radial or non radial solutions of \eqref{Diri} in a
ball. Even in the case of an annulus where uniqueness may fail ---
coexistence of radial and non radial solutions was first observed
in~\cite{BrezisNirenberg} --- uniqueness in the class of radial solutions
was proved in \cite{NiNussbaum}. Multiplicity for \eqref{pblP} was
observed in~\cite{BonheureSerra} where the existence of at least three
non constant
solutions is proved but the results therein are perturbative, assuming
$p\to \infty$, and only concern the case of an annulus.

\smallbreak

Theorem \ref{intro1} is consistent with the analysis of~\cite{wang91} in
the critical case $p=2^{*}$. One of the added value of our result is
that it holds
for any $p>2$ and gives further and precise informations on the
multiplicity of solutions and not solely on the existence.

\smallbreak

The structure of the bifurcations (see Fig.~\ref{fig:bifuc-simple})
also allows to identify degenerate radial solutions along some of the
branches (see Theorem~\ref{deg}). This leads to another striking
difference between the problems \eqref{pblP} and
\eqref{Diri} as it is known that the positive solution to~\eqref{Diri}
is non-degenerate when $\Omega$ is
a ball~\cite{dgp99,Ramaswamy-Srikanth87,Korman98}
or a ``large'' annulus~\cite{Bartsch-Clapp-Grossi-Pacella12}.

\smallbreak

The bifurcation analysis can also be performed without assuming radial
symmetry of the domain $\Omega$, see also \cite{nt}. However, in this
case, it seems necessary to restrict ourselves to a nonlinearity with
subcritical growth. Also we do not have such a precise picture of the
bifurcations since non simple eigenvalues may arise and the study of
the behavior of solutions along a branch is much more delicate. In
particular, we cannot expect any a priori bounds in the supercritical
regime and we expect bifurcations from infinity.

\smallbreak

Even in the case of a radially symmetric domain, non-radial
bifurcations appear, for instance at the first bifurcation
point. Indeed, this first bifurcation occurs at a non-radial
eigenvalue of the elliptic operator (see
Section~\ref{bifu}). Nevertheless, the corresponding eigenfunctions
are axially symmetric.  As we can perform our bifurcation analysis in
the space of axially symmetric functions and since it provides axially
symmetric functions along the branches, it is natural to conjecture
that the first bifurcation is responsible of the symmetry breaking of
the least energy solution when $2+\lambda_2(B_{R})/\lambda< 2^*$ (as
it is also expected that $u = \lambda^{1/(p-2)}$
is the unique positive least energy solution for
$p \le 2+\lambda_2(B_{R})/\lambda$). Moreover, we conjecture that this
first bifurcation is unbounded in $p$ leading to the existence of a
non-radial solution for large $p$ on large balls or for large
$\lambda$ (see Section~\ref{sec:first-bifurc}).

\begin{Conj}
  \label{conj:nonradial-supercritical}
  Let $N \ge 3$, $\lambda>0$, $\Omega = B_R$ and
  $2 + \lambda_2(B_{R})/\lambda< 2^*$. For every $p \in
  \intervaloo{2+\lambda_2(B_{R})/\lambda,\, +\infty}$,
  there exists a positive non radial
  solution of \eqref{pblP} which is axially symmetric.
\end{Conj}

\smallbreak

Concerning the qualitative properties of the least energy solutions to
Problem~\eqref{pblP}, Lopes showed that such a solution is even with
respect to a family of hyperplanes, see~\cite{Lopes}, and in fact it
is furthermore cap symmetric, see~\cite{Weth}. Moreover, Lopes showed
that either the least energy solution is constant or it is
non-radially symmetric. We provide in Section~\ref{symbre} an
alternative and shorter proof of this fact.  This in turn provides the
upper bound $2+\lambda_{2}(B_{R})/\lambda$ on the exponent $p$ at
which the radial symmetry of least energy solutions is lost. For $p$
close to $2$, as a consequence of \cite{bbg}, any least energy
solution of~\eqref{pblP} is invariant under the action of the symmetry
group of the domain~$\Omega$. For instance, in radial domains, these
are radial functions. According to the previous discussion, this
suggests, at least for the ball, that for $p$ close to $2$, any least
energy solution is in fact constant. This is actually true for every
domain, see Theorem~\ref{1=unique sol}. Numerical experiments, based
on the mountain pass algorithm, suggest that
$2+\lambda_{2}(B_{R})/\lambda$ is the exact threshold for the
existence of non constant least energy solutions, see
Section~\ref{expIl}.

\begin{Conj}
  Let $N \ge 2$, $\lambda>0$, $p \in \intervaloo{2, 2^*}$ and $\Omega
  = B_R$. The positive constant solution is the least energy solution
  to~\eqref{pblP} if and only if $p< 2+\lambda_2(B_{R})/\lambda$ and
  there is no other positive solution in this range of the
  parameters. For $p> 2+\lambda_2(B_{R})/\lambda$, the least energy
  solutions are not radially symmetric and belong to the branch
  bifurcating from $(p, u) = \bigl(2+\lambda_2(B_{R})/\lambda,
  \lambda^{1/(p-2)} \bigr)$.
\end{Conj}

Our numerical simulations also complements the
papers~\cite{nitakagi1,nitakagi2} where it was shown that, on a smooth
domain~$\Omega$, the least energy solutions $u_\varepsilon$ of
\begin{equation*}
  \tag{\protect{$\mathcal{P}_\varepsilon$}} \label{pblE}
  \left\{
    \begin{aligned}
      -{\varepsilon}\, \Delta u+u&= f(u),&&\text{ in } B_{R}, \\
      u&>0, &&\text{ in } B_{R}, \\
      \partial_{\nu}u&=0,&&\text{ on } \partial B_{R}
    \end{aligned}
  \right.
\end{equation*}
with $f(u) = \abs{u}^{p-2}u$ with $2<p < 2^*$ concentrate, as
$\varepsilon\to 0$, around a single point of the boundary
$\partial\Omega$. If $\Omega=B_{R}$, this obviously means that when
$R$ is large, the radial symmetry of least energy solutions breaks
down at any fixed subcritical exponent. Our analytical results and the
numerical simulations indicate that ``large'' likely means
$1+\lambda_2(B_{R}) < 2^*$.

In Section~\ref{sec:small diffusion}, we apply our radial bifurcation
analysis on the problem \eqref{pblE}. The parameter $\varepsilon >0$
aims here to model a small diffusion. By a simple scaling argument,
it is easily seen that Problem~\eqref{pblE}
with $f(u) = \abs{u}^{p-2}u$
is equivalent to \eqref{pblP} with $\lambda=1$ in the ball
$B_{R/\sqrt{\varepsilon}}$.
We require few assumptions on $f$. Namely,
$f$ is of class $\C^{1}$ and satisfies, for some $u_{0}>0$,
\begin{gather}
  \tag{\protect{$F_0$}} \label{F0}
  f(0)=f'(0)=0;
  \\[1\jot]
  \tag{\protect{$F_{1}$}} \label{F1}
  f(u_{0})=u_{0}
  \quad\text{and}\quad
  f'(u_{0})>1;
  \\[1\jot]
  \tag{\protect{$F_{2}$}} \label{F2}
  F(s) - \frac{s^2}{2} < \varliminf_{s\to +\infty} \left(F(s)
    -\frac{s^2}{2}\right) \text{ for } 0\le s\le u_{0},\
\end{gather}
where $F(s) := \int_{0}^s f(t)\intd t$.   Assumption~\eqref{F1} implies
in particular that $u_{0}$ is a solution. The third assumption
provides a priori bounds for a large family of solutions which
bifurcate from the constant solution $u_{0}$.
\begin{Thm}\label{intro2}
  Assume $f\in \C^{1}$ satisfies \eqref{F0}, \eqref{F1}, \eqref{F2}, and
  $N\ge 2$.  Then for any $n\in \IN_0$ and any
  $\varepsilon >0$ such that $\varepsilon
  < ({f'(u_{0})-1})/{\lambdarad_{n+1}(B_{R})}$, Problem \eqref{pblE}
  has at least $n$ distinct non-constant radial solutions.
\end{Thm}
If we assume further that $f$ has a subcritical growth, then we can
prove the existence of more solutions, at least $2n$ actually, as in
the case of a pure power. Theorem \ref{intro2} should be compared with
\cite{nt}. Since we deal with a radially symmetric domain, we are able
to go much deeper into the bifurcation analysis. Except from the
restrictive assumption on the
domain, our assumptions on the
nonlinearity $f$ are quite general. In particular, $f$ can have a 
fast growth at infinity.  Notice also that Theorem~\ref{intro2} is not
of perturbative nature since we precisely characterize the values of
$\varepsilon$ at which new solutions arise. The conclusion of
Theorem~\ref{intro2} can be made more precise when $f(s)=s^p$, see
\cite[Theorem B]{Miyamoto} which will be discussed in
Section~\ref{sec:small diffusion}.

We also emphasize that our solutions do not display interior
concentrations as $\varepsilon\to 0$ in opposition to
e.g.~\cite{Aprile-Pistoia10,grossipistoia}.
Actually, our families of solutions correspond to boundary clustered
layer solutions, that is solutions with many local maxima accumulating on
the boundary when $\varepsilon\to 0$.  In particular, the bifurcation
analysis provides an easy approach to find the boundary clustered
layer solutions of~\cite[Corollary~1.3]{MR2056434}
and~\cite[Theorem~1.1]{Malchiodi-Ni-Wei05}. In fact, the bifurcation analysis
gives the complete picture of radial clustering solutions completing
those obtained in \cite{MR2056434,Malchiodi-Ni-Wei05}.  For results in
that direction in a non symmetric setting, we refer to
\cite{MR3262455,MR2296306,MR2221246,MR2072213}.


\medbreak

The paper is organized as follows. Section~\ref{aprioriesti} deals
with a priori bounds, both with or without assuming radial symmetry,
which are crucial in the bifurcation analysis of \eqref{pblP}. In
Section~\ref{bifu}, we first give a general insight on the bifurcation
analysis and then a refined analysis of the radial bifurcations when
$\Omega$ is a ball leads to Theorem~\ref{intro1}. In
Section~\ref{sec:small diffusion}, we prove
Theorem~\ref{intro2}. Section~\ref{symbre} deals with the qualitative
properties of the least energy solutions in a ball. Finally,
Section~\ref{expIl} contains numerical simulations and further
conjectures.


\section{A priori estimates}
\label{aprioriesti}
In this Section, we derive a priori estimates on positive
solutions. These are helpful to control the norm of the solution along
the branches bifurcating from the constant solution. Of course, the
dependence on the bifurcation parameter is important and will be
emphasized. We start with a uniform $L^{1}$ bound.

\begin{Lem}\label{borneL1}
  Any nonnegative solution $u$ of \eqref{pblP} satisfies
  \begin{equation*}
    \int_{\Omega}u^{p-1} = \lambda \int_{\Omega}u
    \le \lambda^{(p-1)/(p-2)} \abs{\Omega}.
  \end{equation*}
\end{Lem}

\begin{proof}
  Integrating the equation leads to
  $\lambda \int_{\Omega} u = \int_{\Omega}u^{p-1}$. H\"older inequality implies
  \begin{equation*}
    \int_{\Omega}u^{p-1} = \lambda \int_{\Omega}u
    \le \lambda \abs{\Omega}^{1-\frac{1}{p-1}} \norm{u}_{p-1},
  \end{equation*}
  so that the claim follows.
\end{proof}

This $L^{1}$ bound can
be improved through a bootstrap argument.

\begin{Prop}
  \label{apriori-bound}
  Assume $2<\bar p< 2^*$.
  There exists $C_{\bar p}>0$ such that any nonnegative solution
  to~\eqref{pblP} with $\lambda=1$ and $2 < p \le \bar p$ satisfies
  \begin{equation}
    \label{eq:bdd-pos}
    \max\bigl\{ \|u\|_{H^1},\|u\|_{L^\infty} \bigr\} \le C_{\bar p}.
  \end{equation}
\end{Prop}

\begin{proof}
  Assume first $2 < \bar{p} < (2N-2)/(N-2)$ and consider a family
  $(u_p)_{p \in \intervaloc{2, \bar{p}}}$ of positive solutions.  We
  argue as in Ni and Takagi~\cite{nt}. From the $L^{1}$ bound
  on $u_p^{p-1}$, by an elliptic regularity result of
  Brezis-Strauss~\cite{Brezis-Strauss}, we deduce a bound for
  $(u_p)$ in $W^{1,q}$ with $1 \le q < N/(N-1)$.  Sobolev embeddings
  give a bound in $L^{r_0}$ for $1 < r_0 < N/(N-2)$ and therefore, by
  the standard elliptic regularity theory, in $W^{2, r_0}$.
  We then bootstrap to increase the regularity.  If $(u_p)$ is bounded in
  $W^{2,r_n}$ with $1 < r_n < {N}/{2}$ then $(u_p)$ is bounded in $W^{2,
    r_{n+1}}$ with
  \begin{equation*}
    r_{n+1} = 
    \frac{1}{\bar{p} - 1} \frac{N}{N-2} r_n.
  \end{equation*}
  As $\bar{p} < \frac{2N-2}{N-2}$, one has $\frac{1}{\bar{p} - 1}
  \frac{N}{N-2} > 1$.  Taking $n$ large enough and choosing $r_0$
  adequately, one deduces that $(u_p)$ is bounded in $W^{2, r_n}$ with
  $r_n > N/2$ and therefore in the desired spaces.

  Let now $2 < \Underline{p} < \bar p < 2^*$.  It remains to prove
  that a family $(u_p)_{p \in \intervalcc{\Underline{p}, \bar p}}$ of
  positive solutions satisfies~\eqref{eq:bdd-pos}.  We follow the
  classical blow-up approach of Gidas-Spruck~\cite{Gidas-Spruck81}, so
  we will only sketch the argument.  Let us argue by contradiction and
  suppose on the contrary that there exists a sequence of exponents
  $(p_n) \subset \intervalcc{\Underline{p},\bar p}$ and a sequence of
  positive solutions $(u_{p_n})$ such that $\norm{u_{p_n}}_\infty \to
  +\infty$.  One can assume that $p_n \to p^* \ge \Underline p >
  2$.  Let $x_n$ be a point where $u_{p_n}$ achieves its maximum. Define
  \begin{equation*}
    v_{n}(y) := \mu_n u_{p_n}\bigl(\mu_n^{(p_n-2)/2} y + x_{n}\bigr),
    \quad
    \text{where}\quad \mu_n := 1 / \norm{u_{p_n}}_{L^\infty} \to 0.
  \end{equation*}
  Note that $v_n(0) = \norm{v_{n}}_{L^\infty} = 1$.
  The function $v_n$ satisfies
    \begin{equation*}
    -\Delta v_{n} + \mu_n^{p_n-2} v_{n} = v_{n}^{p_n - 1}
    \quad
    \text{on } \Omega_n := (\Omega - x_{n})/\mu_n^{(p_n-2)/2},
  \end{equation*}
  with Neumann boundary conditions.
  By elliptic regularity, $(v_n)$ is bounded in $W^{2,r}$ and
  $\C^{1,\alpha}$, $0 < \alpha < 1$ on any compact set.
  Thus, up to a subsequence and a rotation of the domain, one
  concludes that
  \begin{equation*}
    v_n \to v^* \quad \text{in } W^{2,r} \text{ and } \C^{1,\alpha}
    \text{ on compact sets of }
    \Omega^* = \IR^N \text{ or } \IR^{N-1}\times \IR_{> a^*},
  \end{equation*}
  where the choice between the two possibilities for $\Omega^*$
  depends on the limit of the ratio
  $\dist(x_n, \partial\Omega) / \mu_n^{(p_n-2)/2}$ .
  Clearly, one has $v^* \ge 0$, $v^*(0) = 1 = \norm{v}_{L^\infty}$ and
  $v^*$ satisfies
  \begin{equation*}
    -\Delta v^* = (v^*)^{p^*-1} \quad\text{in } \IR^N
    \qquad\text{ or }\qquad
    \begin{cases}
      -\Delta v^* = (v^*)^{p^*-1} &\text{in } \IR^{N-1}\times \IR_{>a^*},\\
      \partial_N v^* = 0& \text{when } x_N = a^*.
    \end{cases}
  \end{equation*}
  Liouville theorems~\cite{Gidas-Spruck81, YanYan-Li03} imply $v^* =
  0$ which contradicts $v^*(0) = 1$.
\end{proof}

This a priori estimate allows to conclude that for $p$ close to $2$,
the constant $u_{0}=1$ is the unique solution of \eqref{pblP} with
$\lambda=1$. In fact, it will be clear that even if $u$ is nonnegative
and solves the equation with Neumann boundary conditions, it has to be
the constant solution. The argument is again inspired from Ni and
Takagi~\cite{nt}.

\begin{Thm}
  \label{gsc2}\label{1=unique sol}
  Let $\Omega$ be a smooth bounded domain in $\IR^{N}$. There
  exists $\tilde{p} = \tilde p(\Omega) >2$ such that,
  if $2<p\le\tilde{p}$, the
  sole nonnegative solutions to Problem~\eqref{pblP} with $\lambda=1$
  are the constant functions $0$ and~$1$.
\end{Thm}
\begin{proof}
  Let $u \ne 0$ be a nonnegative solution to~\eqref{pblP}
  and write $u=\bar u
  +\tilde u$ where $\bar u$ denotes the average of $u$ on $\Omega$ so
  that $\tilde u$ has zero mean. Multiplying the equation by $\tilde
  u$ and integrating gives
  \begin{equation*}
    \int_{\Omega}|\nabla\tilde u|^2+|\tilde u|^2
    = \int_{\Omega} (\bar u + \tilde u)^{p-1}\tilde u
    = \int_{\Omega} \Bigl(\int_{0}^1(p-1)
    (\bar u + s\tilde u)^{p-2}\tilde u^2\intd s \Bigr).
  \end{equation*}
  As $\tilde u$ has zero mean, the left-hand side satisfies
  \begin{equation*}
    \int_{\Omega}|\nabla\tilde u|^2+|\tilde u|^2
    \ge (\lambda_{2} + 1) \int_{\Omega}|\tilde u|^2.
  \end{equation*}
  On the other hand, for any fixed $2< \tilde p <2^*$,
  it follows from Proposition~\ref{apriori-bound}
  that $\bar u + s \tilde u$ is uniformly bounded
  where the bound depends neither on $p\in \intervaloc{2, \tilde p}$
  nor on $s\in \intervalcc{0,1}$.
  Taking $\tilde p$ smaller if necessary, we may assume that
  for every $s\in \intervalcc{0,1}$ and any $2<p\le \tilde p$,
  \begin{equation*}
    (p-1)(\bar u + s \tilde u)^{p-2} < \lambda_{2} + 1.
  \end{equation*}
  We thus deduce that, for $p\le \tilde p$, $\tilde u=0$.
\end{proof}

Next we consider radial solutions of
\begin{equation*}
  \tag{\protect{$\mathcal{P}_{R,p}$}} \label{pblPr}
  \left\{
    \begin{aligned}
      -\Delta u+u&= \abs{ u}^{p-2}u,&&\text{ in } B_{R}, \\
      u&>0, &&\text{ in } B_{R}, \\
      \partial_{\nu}u&=0,&&\text{ on } \partial B_{R}.
    \end{aligned}
  \right.
\end{equation*}
It is observed in \cite{BonheureNorisWeth} that radially increasing
solutions are a priori bounded in $L^{\infty}$. We now show that an a
priori estimate holds true as soon as $u(0)<1$.

\begin{Thm}
  \label{thm21}
  If $u$ is a classical radial solution to Problem~\eqref{pblPr} such
  that $u(0)<1$, then
  \begin{equation*}
    \norm{u}_{\infty}\le \exp(1/2)
    \ \text{ and }\
    \norm{\partial_r u}_{\infty} \le 1.
  \end{equation*}
\end{Thm}

\begin{proof}
  In radial coordinates,  where $'$ denotes $\partial_r$, the
  equation~\eqref{pblPr} writes
  \begin{equation}
    \label{eq:radial}
    -u'' -\frac{N-1}{r}u' + u = \abs{u}^{p-2} u
  \end{equation}
  with $u > 0$ and $u'(0)= u'(R)=0$.
  Multiplying by $u'$, we get that, for all $r > 0$,
  \begin{equation}
    \label{eq:h}
    \frac{\mathrm{d}}{\mathrm{d}r}h(r) = -\frac{N-1}{r}|u'(r)|^{2}\le 0,
    \end{equation}
  where
  \begin{equation}\label{eq:Hamiltonian}
    h(r):= \frac{|u'(r)|^{2}}{2} +\frac{\abs{u(r)}^{p}}{p}-\frac{u^{2}(r)}{2}.
  \end{equation}
  In particular, this means that $h(r)\le h(0)$ for any $r$.
  As we assume $u(0)< 1$ and given that $u'(0) = 0$, we have
  \begin{equation*}
    h(0)
    = \frac{\abs{u(0)}^{p}}{p}-\frac{u^{2}(0)}{2}
    = u^{2}(0) \biggl(\frac{\abs{u(0)}^{p-2}}{p}-\frac{1}{2}\biggr)
    \le 0.
  \end{equation*}
  As a consequence, we deduce (see Fig.~\ref{fig:h} where the thick curve
  corresponds to $h = 0$ and the dashed curves to $h < 0$) that
  \begin{equation*}
    u(r)\le \max_{r\in[0,R]} u(r)
    \le \Bigl(\frac{p}{2}\Bigr)^{1/({p-2})}
    \le \exp(1/2)
  \end{equation*}
  and
  \begin{equation*}
    |u'(r)|^{2}\le \frac{p-2}{p}
    \le 1.
    \qedhere
  \end{equation*}
\end{proof}
\begin{figure}[ht]
  \centering
  \begin{tikzpicture}[x=10ex, y=12ex]
    \draw[->] (-2,0) -- (2,0) node[below]{$u$};
    \draw[->] (0,-0.8) -- (0,0.8) node[left]{$\dot u$};
    \input{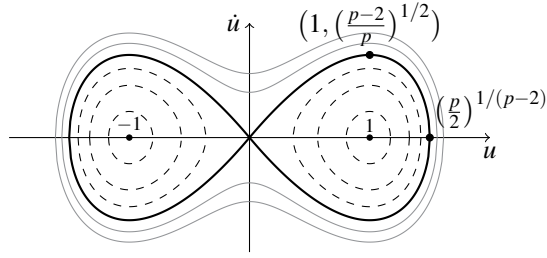}
    \draw[fill] (1,0) circle(1pt)
    node[above, yshift=-0.3ex]{$\scriptstyle 1$};
    \draw[fill] (-1,0) circle(1pt)
    node[above, yshift=-0.3ex]{$\scriptstyle -1$};
    \draw[fill] (1.5,0) circle(1.3pt)
    node[above right, xshift=-0.5ex]{$\bigl( \frac{p}{2} \bigr)^{1/(p-2)}$};
    \draw[fill] (1, 0.57735) circle(1.3pt)
    node[above] {$\bigl(1, \bigl(\frac{p-2}{p}\bigr)^{1/2} \bigr)$};
  \end{tikzpicture}
  \caption{Curves $\tfrac{1}{2}|\dot u|^2 + \frac{1}{p} \abs{u}^p
    - \frac{1}{2} u^2 = \text{const}$.}
  \label{fig:h}
\end{figure}

As soon as an $L^{\infty}$ estimate holds true, we essentially have a
bound in any topology by help of a standard bootstrap argument. The
main feature is that the bound explicitly depends on ${p}$ and does
only blow up as $p\to\infty$.
\begin{Cor}
  \label{cor21}
  For every $k\ge 0$ and $q\ge 1$, there exists $C>0$ such that if $u$
  is a classical radial solution of Problem~\eqref{pblPr} satisfying
  $u(0)<1$, then
  $$\norm{u}_{W^{k,q}}\le C^{p-1}.$$
\end{Cor}
\begin{proof}
  Since $u(r)\le C$ for any $r\ge 0$, we infer that
  \begin{equation*}
    \Bigl(\int_{\Omega} \bigabs{u^{p-1}}^{q}\Bigr)^{1/q}
    \le \abs{\Omega}^{1/q} C^{p-1}.
  \end{equation*}
  It follows from elliptic regularity (see~\cite[Lemma 2.2]{nt}) that
  for every $q>1$, there
  exists $K>0$ such that $\norm{u}_{W^{2,q}}\le K^{p-1}$ and the proof
  follows by induction.
\end{proof}

To handle Problem~\eqref{pblE}, we need to extend the previous
bounds to this case.

\begin{Prop}
  \label{apriori:bound:small:diffus}
  Assume $f$ is of class $\C^{k}$, $k\ge 0$, and \eqref{F2} holds.  For
  any $q\ge 1$ and any $\varepsilon_0>0$, there exists $C>0$ such
  that if $u$ is a classical radial solution of Problem~\eqref{pblE}
  with $u(0)\le u_0$ and $\varepsilon\le\varepsilon_0$, then
  \begin{equation*}
    \norm{u}_{W^{k+2,q}}\le \varepsilon^{-1} C.
  \end{equation*}
\end{Prop}

\begin{proof}
  The equation writes
  \begin{equation*}
    -\varepsilon \Bigl(u'' +\frac{N-1}{r}u' \Bigr) + u = f(u).
  \end{equation*}
  Arguing as in the proof of
  Theorem~\ref{thm21} and using assumption \eqref{F2}, one deduces that
  \begin{equation*}
    \varepsilon\frac{|u'(r)|^{2}}{2 } + F(u(r)) - \frac{u^2(r)}{2}
    \le F(u(0)) - \frac{u^2(0)}{2}
    \le M,
  \end{equation*}
  where $M:=\max_{s\in[0,u_0]}\bigl(F(s) -\frac{s^2}{2}\bigr)$.
  Since
  \begin{equation*}
    \varliminf_{s\to +\infty} \Bigl(F(s) -\frac{s^2}{2}\Bigr) > M,
  \end{equation*}
  it follows that
%
 there exists a constant $L>0$ such that
  \begin{equation*}
    u(r)\le \max_{r\in\mathopen[0,R]} u(r)
    \le L.
  \end{equation*}
  As $f$ is continuous, there exists $K>0$, depending on
  $\varepsilon_0$, such that
  \begin{equation*}
    \biggl(\int_{\Omega} \Bigl|\frac{f(u)-
    (1-\varepsilon)  u}{\varepsilon} \Bigr|^{q} \biggr)^{1/q}
    \le \varepsilon^{-1} |\Omega|^{1/q} K.
  \end{equation*}
 Since $\varepsilon(- \Delta u + u) = f(u) - (1-\varepsilon) u
 $, it follows from elliptic regularity (see~\cite[Lemma 2.2]{nt})
 that for every $q>1$, there exists
  $C>0$ independent of $\varepsilon$ such that
  \begin{equation*}
    \norm{u}_{W^{2,q}} \le \varepsilon^{-1} C
  \end{equation*}
  and the proof follows by induction.
\end{proof}

The next lemma is in the spirit of \cite[Lemma 3.5]{Lin-Ni}. It will
be useful in the bifurcation analysis.

\begin{Lem}
  \label{Lem:uniqeps}
  Assume $f$ is continuous, \eqref{F0}, and \eqref{F2} holds.  Then there exists
  $\overline{\varepsilon} > 0$ such that if $u$ is a
  non constant nonnegative classical
  radial solution of Problem~\eqref{pblE} with
  $\varepsilon \ge \overline{\varepsilon}$, then $u(0)>u_0$.
\end{Lem}

\begin{proof}
  Assume by contradiction that $u(0)\le u_0$ and $u$ is not constant. Proposition \ref{apriori:bound:small:diffus} implies that $u$ is uniformly
  bounded for $\varepsilon\ge 1$. Writting $u=\bar u +\tilde
  u$ and multiplying the equation by $\tilde u$, we get
  \begin{equation*}
    \varepsilon\int_{B_R}|\nabla\tilde u|^2+\int_{B_R}|\tilde u|^2
    = \int_{B_R} f(u)\tilde u.
  \end{equation*}
  Since $u$ is a priori bounded, we infer that
  there exists $C>0$ such that
  $$ \varepsilon\int_{B_R}|\nabla\tilde u|^2 \le C\int_{B_R}|\tilde u|^2,$$
  which obviously implies that $\tilde u=0$ when $\varepsilon >
  \lambda_2 /C$, whence a contradiction.
\end{proof}

\section{Bifurcation analysis}
\label{bifu}

Since $u$ is a solution of~\eqref{pblP} on $\Omega$ iff
$x \mapsto \lambda^{-1/(p-2)} u(x/\sqrt{\lambda})$
solves~\eqref{pblP1} on $\Omega_{\lambda}:=\{\sqrt{\lambda} x \mid
x\in\Omega\}$,
we can fix $\lambda=1$ so that $u=1$ is a solution
of~\eqref{pblP}.
We consider the solvability of
\begin{equation*}
  \tag{\protect{$\mathcal{P}_{1,p}$}} \label{pblP1} \left\{
    \begin{aligned}
      -\Delta u+u&= \abs{ u}^{p-2}u,&&\text{ in } \Omega, \\
      \partial_{\nu}u&=0,&&\text{ on } \partial \Omega,
    \end{aligned}
  \right.
\end{equation*}
and we will check the positivity of the solutions a posteriori.
The solutions to Pro\-blem \eqref{pblP1} with $2<p<2^*$ can be seen
as the zeros of the Fr\'echet differential of the functional
\begin{equation*}
  \mathcal{E}_{p}:H^1(\Omega)\to\IR:u\mapsto\frac{1}{2}\int_{\Omega}
  \abs{\nabla u}^2+u^2-\frac{1}{p}\int_{\Omega}\abs{u}^p,
\end{equation*}
i.e., the zeros of the map
$$H^1(\Omega) \to (H^1(\Omega))' : u \mapsto \E_{p}'(u),$$
where the linear map
$\E_{p}'(u):H^1(\Omega)\to \IR : h \mapsto \E_{p}'(u)[h]$ is given by
\begin{equation*}
  \E_{p}'(u)[h]
  = \int_{\Omega}\nabla u\nabla h +  u h - \int_{\Omega}\abs{u}^{p-2} uh.
\end{equation*}
We consider $p$ as a unknown in the problem and we investigate the
bifurcation points along the trivial solution curve
$\{ (p,1) : p > 2 \} \subset \mathbb R^{+}\times H^1(\Omega)$
to~\eqref{pblP1}. We recall that a
point $(p^{*},1)$ is called a \emph{bifurcation point} if every
punctured neighborhood of
$(p^{*},1)$ contains a solution of~\eqref{pblP1}.
The Implicit Function Theorem implies that if $(p^{*},1)$ is a
bifurcation point, then the map
\begin{equation*}
  H^1(\Omega) \to (H^1(\Omega))' :
  \varphi \mapsto \E_{p^{*}}''(1)[\varphi, \_],
\end{equation*}
where
$\E_{p^{*}}''(1)[\varphi, \_] : H^1(\Omega)\to \IR : \psi \mapsto
\E_{p^{*}}''(1)[\varphi, \psi]$
is given by
\begin{equation*}
  \E_{p^{*}}''(1)[\varphi, \psi]
  = \int_{\Omega}
  \nabla \varphi\nabla \psi + (2- p^{*})\int_{\Omega}\varphi\psi,
\end{equation*}
is not an isomorphism.  This is the case if and only if
\begin{equation}
  \label{eq:p-bifurc}
  p^{*} = 2+\lambda_i(\Omega),\qquad \text{for some } i>1,
\end{equation}
where $0 = \lambda_1(\Omega) < \lambda_2(\Omega) < \cdots$
are the eigenvalues of
the operator $-\Delta$ with Neumann boundary conditions in $\Omega$.

Thanks to the fact that the problem has a variational structure, the
converse is also true~\cite{Bohme72, Krasnoselskii64, Marino73,
Rabinowitz77}, see also~\cite{brown,mw}. Namely, if $p^{*}$
satisfies~\eqref{eq:p-bifurc}, then $(p^{*},1)$ is a bifurcation
point. Moreover, standard arguments in degree theory imply that there
is actually a continuum of nontrivial solutions when the
dimension of the
eigenspace for $\lambda_i(\Omega)$ is odd.
A continuum $B$ of nontrivial
solutions which cannot be extended (i.e., a connected
component) is called a \emph{branch}.
If $\closure{B} \ni (2+\lambda_i, 1)$, we say that
$B$ bifurcates from $(2+\lambda_i, 1)$.
In this case, Rabinowitz's principle~\cite{rabibif} applies:
a branch $B$ bifurcating from
$(2+\lambda_i, 1)$ is unbounded in $\IR \times H^1$
or there exists an eigenvalue $\lambda_j \ne \lambda_i$
such that $(2 + \lambda_j, 1) \in \closure{B}$
(in which case we say that the branch is \emph{linked by pair}).

In order to be able to establish more properties of the bifurcating
branches, we now restrict ourselves to the case where $\Omega$ is a
ball $B_R$ of radius $R$.  Then, one has a precise knowledge of the
eigenspaces of $-\Delta$ which makes the analysis much simpler.

We already know that the first eigenvalue\footnote{Since we have now
  fixed the domain, we drop the dependance of the eigenvalues on the
  domain.} $\lambda_1$ equals $1$ and any associated eigenfunction is
constant.  Let $r = \abs{x}$ and $\theta = \frac{x}{\abs{x}} \in
\IS^{N-1}$.  By the
method of separation of variables, one concludes that all
eigenfunctions of $-\Delta$ with Neumann boundary conditions have the
form
\begin{gather}
  \label{eigg}
  u(x) = r^{-\frac{N-2}{2}} J_{\nu}(\sqrt{\lambda_i}r)
  P_k\Bigl(\frac{x}{\abs{x}}\Bigr),
  \qquad\text{where }
  \nu = k+\frac{N-2}{2},
\end{gather}
$J_\nu$ is the Bessel function of the first kind of order~$\nu$,
and $P_k:\IR^N\to\IR$ is an harmonic homogenous polynomial of degree
$k$ for some $k \in \IN$.  To satisfy the boundary conditions, the
corresponding eigenvalue $\lambda_{i} \ge 0$ of $-\Delta$ must
be such that $\sqrt{\lambda_{i}}R$ is a root of the map
\begin{equation*}
  z \mapsto (k - \nu) J_\nu(z) + z \partial J_\nu(z)
  = k J_\nu(z) - z J_{\nu+1}(z).
\end{equation*}
In other words,
each of the infinitely many real roots $z_{k,\ell}$, $\ell \ge 1$,
of this function gives rise to the
eigenvalue $z_{k,\ell}^2 / R^2$ of $-\Delta$.
Radial eigenfunctions correspond to $k=0$. For each eigenspace $E_i$,
the dimension of its intersection with radial functions is $0$
or~$1$. Moreover, for functions in $E_i$ with $k=0$, one notices that
the zeros of those functions are simple.

The remaining of this section is devoted to the study of radial
bifurcations. We say that $(p^{*},1)$ is a \emph{radial bifurcation
point}
if every punctured neighborhood of $(p^{*},1)$ contains radial solutions.
In the sequel, we denote by $0 = \lambdarad_1 < \lambdarad_2 < \lambdarad_3
< \cdots$ the eigenvalues of $-\Delta$ whose eigenspaces contain radial
eigenfunctions.

\subsection{Radial bifurcations in
  \texorpdfstring{$\C^{2,\alpha}$}{C²ᵅ}}
\label{Section:radialbif}

For $k\ge 0$, we denote by
$\C^{k,\alpha}_{\textup{rad}}(\Bar B_R)$ the space
$\C^{k,\alpha}(\Bar B_R)$ restricted to radially invariant functions and
\begin{equation*}
  \Tilde{\C}^{2,\alpha}_{\textup{rad}}(\Bar B_R)
  := \bigl\{u\in \C^{2,\alpha}_{\textup{rad}}(\Bar B_R) \bigm|
  \partial_{r}u(R) = 0 \bigr\}.
\end{equation*}
If we define $-\Delta$ on the space
$\Tilde{\C}^{2,\alpha}_{\textup{rad}}(\Bar B_R)$, then the spectrum is
made of the increasing sequence
$0 = \lambdarad_1 < \lambdarad_2 < \lambdarad_3
< \cdots$ of simple eigenvalues.

The function $u\in \Tilde{\C}^{2,\alpha}_{\text{rad}}$ is a classical
solution of Problem~\eqref{pblP1} with $\Omega=B_{R}$ if and only if
the couple $(p,u)$ is a zero of the function
\begin{equation}\label{mapF}
  F: \IR \times \Tilde{\C}^{2,\alpha}_{\text{rad}} \to
  {\C}^{0,\alpha}_{\text{rad}} : (p, u) \mapsto
  (-\Delta + \id) u - \abs{u}^{p-2} u.
\end{equation}
It is easily seen that
\begin{equation}\label{duF}
  \partial_uF(2+\lambda_i, 1)[v] = (-\Delta + \id) v -  \lambda_i v,
\end{equation}
so that  classical bifurcation theory implies
that if $E_i\cap\Tilde{\C}^{2,\alpha}_{\text{rad}} =\{0\}$, then
$(2+\lambda_i,  1)$ is not a  radial bifurcation point in
$\Tilde{\C}^{2,\alpha}_{\textup{rad}}$ whereas
if $E_i\cap \Tilde{\C}^{2,\alpha}_{\text{rad}}  \ne \{0\}$, then
$(2+\lambda_i, 1)$ is a bifurcation point in
$\Tilde{\C}^{2,\alpha}_{\textup{rad}}$.
We will improve this first vague result by studying the local behavior
of the bifurcations branches from $(2+\lambdarad_i, 1)$
in $\Tilde{\C}^{2,\alpha}$.
We will use the celebrated Crandall-Rabinowitz
theorem~\cite[Theorem~1.7 and 1.18]{Crandall-Rabinowitz71}
(see also~\cite{ambroprodi}) in a form that we recall first.

\begin{Prop}[Crandall-Rabinowitz]\label{bifurc-simple}
  Let $X$ and $Y$ be two Banach spaces, $p^* \in \IR$ and $u^* \in
  X$. Assume $F : \IR \times X\to Y: (p, u)\mapsto F(p,u)$ is such
  that
  \begin{enumerate}[(i)]
  \item $F(p, u^*)=0$ for any $p$ in a neighborhood of $p^*$;
  \item the partial derivatives $\partial_pF$, $\partial_uF$ and
    $\partial_{pu}F$ exist and are continuous in a neighborhood of
    $(p^*, u^*)$;
  \item $\ker\bigl(\partial_u F(p^*,u^*)\bigr)$ is one-dimensional
    and is thus spanned by some $\phi^* \in X\setminus\{0\}$;
  \item $\Ran\bigl(\partial_u F(p^*,u^*)\bigr)$ has codimension $1$
    and is thus the kernel $\{ y \in Y : \langle \psi, y \rangle = 0
    \}$ of some continuous linear functional $\psi : Y \to \IR$.
  \end{enumerate}
  Then the following assertions hold.
  \begin{enumerate}[(1)]
  \item If
    $$a := \langle\psi, \partial_{pu}F(p^*,u^*)[\phi^*] \rangle \ne 0,$$
    then $(p^*,u^*)$ is a bifurcation point for $F$.  In addition, the
    set of nontrivial solutions of $F=0$ in a neighborhood of
    $(p^*,u^*)$ is given by a unique continuous curve $s \mapsto
    (\bar p(s), \bar u(s))$ defined for $s$ close to~$0$.
    More precisely $(\bar p(0), \bar u(0)) = (p^*, u^*)$,
    $\bar u$ is of class $\C^1$,
    $\partial_s \bar u(0) = \phi^*$,
    and, for all $(p,u)$ in a neiborhood of $(p^*, u^*)$,
    $$\bigl(F(p,u) = 0 \land u \ne u^* \bigr) \liff \exists\, s \ne 0,\
    \linebreak[2]
    (p,u) = (\bar p(s), \bar u(s)).$$
    If $\partial_{u}^2F$ exists and is continuous, the curve is of
    class~$\C^1$.
  \item Assuming $a\ne0$, if $\partial^2_uF$ is continuous and
    $$b := - \frac{1}{2a}
    \bigl\langle\psi, \partial^2_uF(p^*,u^*)[\phi^*,\phi^*]
    \bigr\rangle \ne 0,$$ then the bifurcation point is
    \emph{transcritical} and the nontrivial solution curve
    can be (locally) written $(p, u_p)$ with
    \begin{equation}
      \label{eq:transcritical}
      u_{p} = u^* + \frac{p - p^*}{b} \phi^* + o(p - p^*).
    \end{equation}

  \item Assuming $a\ne 0$, $b=0$ and $\partial^3_uF$ is continuous, if
    \begin{equation*}
      c := -\frac{1}{6a} \Bigl(
      \bigl\langle\psi, \partial_u^3F(p^*,u^*)[\phi^*,\phi^*,\phi^*]
      \bigr\rangle
      + 3 \bigl\langle\psi, \partial_u^2F(p^*,u^*)[\phi^*,w] \bigr\rangle
      \Bigr)\ne 0,
    \end{equation*}
    where $w \in X$ is any solution of the equation
    \begin{math}
      \partial_u F(p^*,u^*)[w] = - \partial^2_uF(p^*,u^*)[\phi^*,\phi^*],
    \end{math} %
    we have
    \begin{equation*}
      u_p = u^* \pm \Bigl( \frac{p-p^*}{c}\Bigr)^{1/2} \phi^*
      + o \bigl({\abs{p - p^*}}^{1/2}\bigr).
    \end{equation*}
    In particular, the bifurcation point is \emph{supercritical} if
    $c > 0$ and \emph{subcritical} if $c < 0$.
  \end{enumerate}
\end{Prop}

\begin{figure}[htb]
  \null\hfil
  \begin{minipage}[b]{0.3\linewidth}
    \begin{tikzpicture}[x=4mm,y=4mm]
      \draw[color=black] (2.5,-2.5) node[anchor=north] {\strut
        Transcritical};
      \draw[->] (0,0) -- (5.5,0) node[below]{$p$};
      \draw[->] (0,-2) -- (0,2) node[left]{$u$};
      \draw[thick] (0.5,-1.5) .. controls (2.5, 0) and (3.5,0)
      .. (4.5,1.5);
      \draw[thick,fill] (2.9,0) circle(1pt) node[below]{$p^*$};
    \end{tikzpicture}
  \end{minipage}
  \hfil
  \begin{minipage}[b]{0.3\linewidth}
    \begin{tikzpicture}[x=4mm,y=4mm]
      \draw[color=black] (2.5,-2.5) node[anchor=north] {\strut
        Supercritical};
      \draw[->] (0,0) -- (5.5,0) node[below]{$p$};
      \draw[->] (0,-2) -- (0,2) node[left]{$u$};
      \draw[thick] plot[samples at={-1.6,-1.5,...,1.6}]
      (2.5 + \x * \x, \x);
      \draw[thick,fill] (2.5,0) circle(1pt)
      node[below left, xshift=3pt]{$p^*$};
    \end{tikzpicture}
  \end{minipage}
  \hfil
  \begin{minipage}[b]{0.3\linewidth}
    \begin{tikzpicture}[x=4mm,y=4mm]
      \draw[color=black] (2.5,-2.5) node[anchor=north] {\strut
        Subcritical};
      \draw[->] (0,0) -- (5.5,0) node[below]{$p$};
      \draw[->] (0,-2) -- (0,2) node[left]{$u$};
      \draw[thick] plot[samples at={-1.5,-1.4,...,1.6}]
      (2.5 - \x * \x, \x);
      \draw[thick,fill] (2.5,0) circle(1pt)
      node[below right, xshift=-3pt]{$p^*$};
   \end{tikzpicture}
  \end{minipage}
  \hfil\null

  \vspace{-1ex}
  \caption{Type of bifurcations.}
  \label{fig:bifurcation-types}
\end{figure}
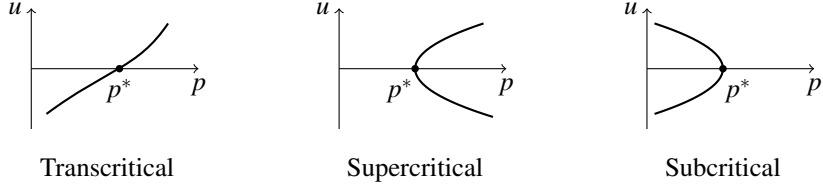

We now apply these statements to our problem. We still consider the
map \eqref{mapF}.
We fix $(p^{*},u^{*})=(2+\lambdarad_i,1)$. Assumption (i) is clear
while (ii) can be checked with standard arguments. We deduce from
\eqref{duF} that
$$\ker\bigl(\partial_uF(2 + \lambdarad_i, 1)\bigr)
= \langle \phi_i\rangle,$$
where we can assume that $\phi_i$ is the unique radial eigenvalue of
$-\Delta$ associated to $\lambdarad_i$, normalized in $L^{2}(B_{R})$.
By the Fredholm alternative, we also have
$$\codim\bigl(\Ran(\partial_uF(2+\lambdarad_i, 1))\bigr) = 1$$
and
$f\in \Ran\bigl(\partial_uF(2+\lambdarad_i, 1) \bigr)$ if and only if
$\int_{B_R} f \phi_i = 0$,
so that one can take
\begin{equation*}
  \psi : {\C}^{0,\alpha}_{\text{rad}} \to \IR :
  f\mapsto \langle\psi, f\rangle
  := \int_{B_R} f \phi_i.
\end{equation*}
Simple computations show that
\begin{equation}\label{a}
  a = \int_{B_R} \partial_{pu}F(2+\lambdarad_i,1)[\phi_i]\phi_i =
  -\int_{B_R}\phi_i^2  = -1
\end{equation}
and
\begin{align}\label{b}
  b = -\frac{1}{2a}
  \int_{B_R} \partial^2_{u}F(2+\lambdarad_i, 1)[\phi_i,\phi_i]\phi_i
  \intd x
  = -\frac{1}{2}
  (1 + \lambdarad_i) \lambdarad_i \int_{B_R} \phi_i^3 .
\end{align}
In order to compute $b$, we will use the following property of
Bessel's functions which is in fact the key in our analysis.

\begin{Lem}
  \label{lemma:integ>0}
  Let $\nu \ge 1/2$, $\beta > 0$ and $\alpha \in \intervaloc{-1 - \nu \beta,
    \beta/2}$. If $\nu = 1/2$, assume further that $\alpha < \beta/2$.
  Then for every $x > 0$, we have
  \begin{equation}
    \label{eq:intJ>0}
    \int_0^x s^\alpha J_\nu^\beta(s) \intd s > 0,
  \end{equation}
  where $J_\nu^\beta(s)$ means $\sign\bigl(J_\nu(s)\bigr)
  \abs{J_\nu(s)}^\beta$.
\end{Lem}
\begin{proof}
  First note that the integral exists.  Indeed, since $J_\nu(x)$
  behaves like $x^{-\nu}$ as $x \to 0$, the integrant is integrable in
  a neighborhood of~$0$ iff $\alpha + \nu \beta > -1$.

  Next, recall that, for $\nu \ge 0$, the following
  representation of Bessel functions holds :
  \begin{equation}
    \label{eq:Jnu=sin}
    \forall x > 0,\qquad
    J_\nu(x) = \sqrt{\frac{2}{\pi x}} \, \sqrt{p_\nu(x)}
    \sin\biggl( \int_0^x \frac{\mathrm{d}\xi}{p_\nu(\xi)} \biggr),
  \end{equation}
  where
  \begin{equation*}
    p_\nu(x)
    := \tfrac{1}{2}\pi x \bigl( J_\nu^2(x) + Y_\nu^2(x) \bigr).
  \end{equation*}
  According to formulas
  (\href{http://dlmf.nist.gov/10.18.E4}{10.18.4}),
  (\href{http://dlmf.nist.gov/10.18.E6}{10.18.6}) and
  (\href{http://dlmf.nist.gov/10.18.E8}{10.18.8}) of~\cite{dlmf}, we
  have
  $$J_\nu(x) = M_\nu(x) \cos\bigl(\theta_\nu(x) \bigr),$$
  where $p_\nu(x) = \tfrac{1}{2} \pi x M_\nu^2(x) > 0$
  and $\partial_x \theta_\nu(x) = 1 /
  p_\nu(x)$.  Given that
  $$\theta_\nu(x) \to -\tfrac{1}{2}\pi,\ \text{ as }x
  \xrightarrow{>} 0,$$ one deduces\footnote{Since, when $\nu > 0$
    (resp.\ $\nu=0$), $p_\nu(x)$ behaves
    like $x^{1-2\nu}$ (resp.\ $x \ln^2 x$)
    as $x \to 0$, the function $1/p_\nu(x)$ is
    integrable in a neighborhood of $0$ for $\nu \ge 0$.} that
  $\theta_\nu(x) = -\tfrac{1}{2}\pi + \int_0^x {1}/{p_\nu(\xi)}
  \intd\xi$, hence the claimed formula (we also refer
  to~\cite{Lorch-Muldoon08}).

  Now, note that the formulas (\href{http://dlmf.nist.gov/10.7.E8}{10.7.8})
  in~\cite{dlmf} imply that
  $$\lim_{x \to +\infty}p_\nu(x)=1,$$
  whence
  $$\lim_{x \to +\infty} \int_0^x {1}/{p_\nu(\xi)} \intd\xi = + \infty.$$
  Using \eqref{eq:Jnu=sin} and performing the change of
  variables $t = \int_0^x {1}/{p_\nu(\xi)} \intd\xi$, the claim
  \eqref{eq:intJ>0} can be written
  \begin{equation*}
    \forall \tau > 0,\qquad
    \int_0^\tau \bigl(X(t)\bigr)^{\alpha - \beta/2}
    \bigl(p_\nu(X(t))\bigr)^{1 + \beta/2}  \sin^\beta(t) \intd t
    > 0
  \end{equation*}
  where $X : \intervalco{0, +\infty} \to \intervalco{0, +\infty}$ is
  such that $X(0) = 0$ and $\partial_tX(t) = p_\nu(X(t)) > 0$.
  Since the sine function is periodic, it is enough to show
  that $t \mapsto (X(t))^{\alpha - \beta/2}
  p_\nu^{1 + \beta/2}(X(t))$ is decreasing because then the
  integral on the interval $\intervalcc{2k\pi, (2k+1)\pi}$ will be
  greater than the negative contribution in the next interval
  $\intervalcc{2(k+1)\pi, (2k+2)\pi}$.
  As the function $X$ is increasing, it is equivalent to show that $x
  \mapsto x^{\alpha - \beta/2} \bigl(p_\nu(x)\bigr)^{1 + \beta/2}$ is
  decreasing, or, setting $\gamma = (\alpha - \beta/2)/(1 + \beta/2)$,
  that $x \mapsto x^\gamma p_\nu(x)$ is decreasing.

  According to formula (\href{http://dlmf.nist.gov/10.9.E30}{10.9.30})
  of~\cite{dlmf}, we have
  \begin{equation}
    \label{eq:Nicholson}
    p_\nu(x)
    = \tfrac{4}{\pi} x
    \int_0^\infty \cosh(2\nu t) K_0(2x \sinh t) \intd t
  \end{equation}
  where $K_0$ is the second modified Bessel function.  We have to
  distinguish between $\nu > 1/2$ and $\nu = 1/2$.

  Assume first $\nu > 1/2$. Performing the change of variable $s := x
  \sinh t$ in~\eqref{eq:Nicholson}, we get
  \begin{equation*}
    p_\nu(x)
    = \tfrac{4}{\pi} \int_0^\infty
    \frac{\cosh(2\nu t)}{\cosh t}\Bigr|_{t = \arcsinh(s/x)}
    K_0(2s) \intd s.
  \end{equation*}
  The function $t \mapsto {\cosh(2\nu t)}/{\cosh t}$ is
  increasing. Therefore the first term of the product is a decreasing
  function of~$x$. As $K_0 > 0$, so is $p_\nu$. It is therefore
  sufficient that $\gamma \le 0$ in this case.

  If $\nu = 1/2$, $p_{\nu}(x) = 1$ for all $x > 0$ (see e.g.\
  (\href{http://dlmf.nist.gov/10.43.E18}{10.43.18})
  in~\cite{dlmf}). It follows that the map $x \mapsto x^\gamma
  p_\nu(x)$ is decreasing iff $\gamma < 0$.
\end{proof}

\begin{Rem}
  When $\nu = 1/2$ (i.e., $N = 3$ for our application in upcoming
  Theorem~\ref{transcritical-radial-bifurc}), $J_\nu(x) =
  \sqrt{2/(\pi x)} \sin x$ and so the statement simplifies to
  \begin{equation*}
    \forall x > 0,\qquad
    \int_0^x x^{\alpha - \beta/2} \sin^\beta(x) \intd x > 0
  \end{equation*}
  which is true as soon as $x \mapsto x^{\alpha - \beta/2}$ is
  decreasing.  This is the only case where the assumptions of
  Lemma~\ref{lemma:integ>0} are sharp: for $\nu > 1/2$, the
  statement~\eqref{eq:intJ>0} remains true for some $\alpha >
  \beta/2$.
\end{Rem}
\begin{Rem}
  \label{transcritical-N=2}%
  When $\nu = 0$ (i.e., $N = 2$ for our application), although $p_0$
  is increasing and an asymptotic analysis around~$0$ shows that
  $x^\gamma p_0(x)$ is not decreasing whatever $\gamma$, numerics
  indicate that~\eqref{eq:intJ>0} is positive at least if $\alpha
  < 0.45 \beta - 0.15$, and in particular in the case of interest for
  Theorem~\ref{transcritical-radial-bifurc}: $\alpha = 1 - \nu = 1$
  and $\beta = 3$.

 \begin{figure}[htb]
  \centering
  \begin{tikzpicture}[x=1.4ex, y=28ex]
    \draw (0,0) -- (34.5, 0);
    \draw[->, dashed] (34,0) -- (38,0) node[below]{$z$};
    \draw[->] (0,0) -- (0, 0.7);
    \foreach \z in {0,2,...,34}{%
      \draw (\z,0.01) -- (\z, -0.01) node[below]{\footnotesize $\z$};
    }
    \draw (36,0.01) -- (36, -0.01) node[below]{\footnotesize $\infty$};
    \foreach \y in {0,1,...,6}{%
      \draw (0.2,0.\y) -- (-0.2,0.\y) node[left]{\footnotesize $0.\y$};
    }
    \begin{scope}[thick]
      \input{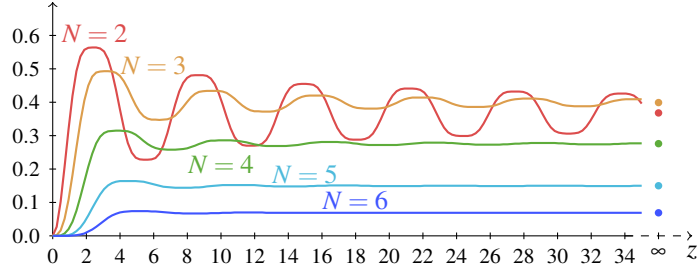}
    \end{scope}
    \node[color=n2, above] at (2.5,0.55) {$N=2$};
    \node[color=n3, right] at (3.5,0.51) {$N=3$};
    \node[color=n4, below right] at (7.5,0.27) {$N=4$};
    \node[color=n5, above] at (15,0.13) {$N=5$};
    \node[color=n6, above] at (18,0.05) {$N=6$};
   \end{tikzpicture}
   \caption{Plots of $z \mapsto  \int_0^z s^{1-\nu} J_{\nu}^3(s) \intd s$
     with $\nu= N/2-1$ and $N=2,\ldots,6$.}
  \end{figure}
  Moreover, Equation~(\href{http://dlmf.nist.gov/10.22.E74}{10.22.74})
  of~\cite{dlmf} asserts
  \begin{equation*}
    \int_0^{\infty} x^{1-\nu} J_\nu^3(x) \intd x
    = \frac{2^{\nu-1} (3/16)^{\nu - 1/2}}{\pi^{1/2} \,\Gamma(\nu+1/2)}
    > 0,
  \end{equation*}
  so that the \eqref{eq:intJ>0} holds true for $x$ large enough in the case
  $\alpha = 1 - \nu$ and $\beta = 3$.
\end{Rem}

With Lemma \ref{lemma:integ>0} at hand, we can prove the following.

\begin{Thm}\label{transcritical-radial-bifurc}
  Assume $\Omega=B_{R}$ and $N \ge 3$.
  For every $i \ge 2$,  $(p_{i},u_{0}) := (2+\lambdarad_i,1)$ is a
  bifurcation point  in
  $\Tilde{\C}^{2,\alpha}(B_R)$ of Problem~\eqref{pblP1}.
  Denote $B_i$ the branch bifurcating from $(2+\lambdarad_i,1)$.
  The following holds:
  \begin{enumerate}[(i)]
  \item close to $(2+\lambdarad_i,1)$, the branch is a $\C^1$-curve;
  \item there exists $\varepsilon>0$ (which does not depend on $i$)
    such that if $(p,u_{p}) \in B_i$ then $u_{p}$ is positive and
    $p>2+\varepsilon$;
  \item\label{transcritical} the bifurcation point
    $(2+\lambdarad_i,1)$ is transcritical.
    Furthermore, if $B_i^-$ denotes the part of the
    branch starting at $(2+\lambdarad_i,1)$
    which bifurcates to the right of $2 + \lambdarad_i$,
    we have $(p,u_{p}) \in B_i^- \limplies u_{p}(0) < 1$,
    while on the part $B_i^+$ of the
    branch bifurcating to the left of $2 + \lambdarad_i$,
    the branch is made of solutions satisfying $u_{p}(0) > 1$.
  \end{enumerate}
\end{Thm}

Observe that for $i = 2$, the functions on the branch emanating to the
right (resp.\ left) of $2 + \lambdarad_2$ are increasing (resp.\
decreasing), at least when $p$ is close enough to $2 +
\lambdarad_2$. We will prove later on that this holds actually along
the whole branch. Some parts of the proof follow from nowadays standard arguments. We give them for completeness. 

\begin{proof}
  (i) Since $a\ne 0$, Crandall-Rabinowitz theorem~\ref{bifurc-simple}
  implies that $(2+\lambdarad_i,1)$ is a continuous bifurcation point
  for $F$.  In addition, the set of non-trivial solutions of $F=0$ is
  composed locally of a unique curve of class~$\C^1$. Hence (i) holds.

  (ii) Let $\mathcal{B}_{i}\subset \IR \times \Tilde{\C}^{2,\alpha}_{\text{rad}}$
  be the continuum that branches out from $(2+\lambdarad_i,1)$,
  $B_i := \mathcal{B}_{i} \setminus\{(2+\lambdarad_i,1)\}$
  and $(p,u)\in B_{i}$. Close to the bifurcation point
  $(2+\lambdarad_i,1)$, $u$ is close to $1$ in the $\C^2$ topology,
  so that $u$ is clearly positive.
  Let $\tilde p > 2$ be given by Theorem~\ref{1=unique sol}.
  We can assume that $\tilde p$ is smaller than
  $2 + \lambdarad_2 > 2$.
  We claim that
  \begin{equation*}
    \forall (p,u)\in B_{i},\quad
    p> \tilde p  \text{ and } u>0.
  \end{equation*}
  By connectedness, if the claim does not hold, there exists, on the continuum
  $B_{i}$, a nonnegative solution $u$
  to~\eqref{pblP1} such that $p = \tilde p$
  or $u$ vanishes at at least one point.
  In the first case $p = \tilde p$,
  Theorem~\ref{1=unique sol} implies that $u\equiv 1$ or
  $u\equiv 0$ but this is impossible since neither $(\tilde p,1)$
  nor $(p,0)$, $p \in \intervaloo{2, +\infty}$,
  are bifurcation points of $F$.  We can
  therefore suppose that $p> \tilde p$, $u\ge0$ and $u \not\equiv 0$.
  If the set $\{r\in[0,R]\mid u(r)=0\}$ contains a point $r_{0} \in
  \intervaloc{0,R}$, then $u'(r_{0}) = 0$ because $r_{0}$ is either an
  interior minimum or the boundary condition holds. The local uniqueness
  of the solution for the Cauchy problem \eqref{eq:radial} with initial
  data $u(r_{0})=0$, $u'(r_{0})=0$ now implies $u\equiv 0$ which is a
  contradiction.
  It remains to deal with the case where $\{r\in[0,R]\mid
  u(r)=0\}=\{0\}$. Choose $s > 0$ small enough so that $u(r) < 1$
  for all $r \in \intervalcc{0,s}$. We claim that $u' > 0$ on
  $\intervaloc{0,s}$.  Because $u > 0$ on $\intervaloc{0,s}$,
  there are points $r$ arbitrarily close to $0$ such that $u'(r) >
  0$. It is thus sufficient to show that $u'(r) \ne 0$ for  $r \in
  \intervaloc{0,s}$.  Notice that, in view
  of equation~\eqref{eq:radial},
  \begin{equation}
    \label{eq:crit pt u(r)}
   \text{for all }r \in \intervaloc{0,s},\
    u(r) \in \intervaloo{0,1} \text{ and } u'(r) = 0
    \text{ implies }  u''(r) > 0.
  \end{equation}
  Suppose that, on the contrary, $u'$ vanishes at $t\in \intervaloc{0,s}$.
  Using~\eqref{eq:crit pt u(r)} with $r = t$, one deduces that the
  maximum of $u$ over $\intervalcc{0,t}$ occurs at some interior point $r_1
  \in \intervaloo{0,t}$ such that $u'(r_1) = 0$ and $u''(r_1) \le
  0$.  This contradicts~\eqref{eq:crit pt u(r)} and thus we indeed have
  $u' > 0$ on~$\intervaloc{0,s}$.
  Now, remember that, because $u$ describes the profile
  of a radial function, $u'(0) = 0$.
  Therefore, for all $r \in \intervalcc{0,R}$, $h(r) \le h(0) = 0$ where
  $h$ is defined by~\eqref{eq:Hamiltonian}.  This implies
  \begin{equation*}
    (u') ^2 \le u^2 - \tfrac{2}{p} u^p
    \le u^2
  \end{equation*}
  and thus $u' \le u$ on $\intervaloo{0,s}$.  Using Gronwall's
  inequality, one concludes $u(r) \ge u(s)  \e^{r-s}$ for all $r
  \in \intervalcc{0,s}$ which is a contradiction when $r=0$.

  \smallbreak

  (iii) The bifurcation is transcritical if $b \ne 0$, where $b$
  is defined by \eqref{b}.
  Taking the explicit form of the eigenfunctions, and integrating in
  spherical coordinates, one finds that $b \ne 0$ if
  and only if
  \begin{equation*}
    \int_0^R \Bigl( r^{-\frac{N-2}{2}}
    J_{\nu} \bigl(r \sqrt{\Bar\lambda_i} / R\bigr)
    \Bigr)^3  r^{N-1}\intd r  \ne 0,
 \end{equation*}   
or equivalently
 \begin{equation}  
    \int_0^{\sqrt{\Bar\lambda_i}} t^{1-\nu} J_\nu^3(t) \intd t
    \ne 0,
    \label{eq:reduced-b}
  \end{equation}
  where $\Bar\lambda_i =
  \lambdarad_i(B_1)$
  is the corresponding spherical eigenvalue of $-\Delta$ on the unit
  ball (so that $\lambdarad_i = \Bar\lambda_i / R^2$) and $\nu = N/2 - 1 \ge 0$.
  According to Lemma~\ref{lemma:integ>0} with $\alpha = 1 - \nu$
  and $\beta = 3$, this integral is positive for all $i$.
  Recalling that
  \begin{equation*}
    b = -\frac{1}{2}
    (1 + \lambdarad_i) \lambdarad_i
    \, \abs{\IS^{N-1}}\int_0^{\sqrt{\Bar\lambda_i}} t^{1-\nu} J_\nu^3(t) \intd t,
  \end{equation*}
  one gets $b < 0$ for all $i$.

\medbreak

  For the last statement in (iii), first notice that,
  up to a positive normalization factor, the
  function $\phi_i$ is $x \mapsto \abs{x}^{-\nu}
  J_\nu(\sqrt{\lambdarad_i} \abs{x})$.  In view of
  equation~(\href{http://dlmf.nist.gov/10.7.E3}{10.7.3})
  of~\cite{dlmf}, $\phi_i(0) > 0$. 
  Using the fact that $b < 0$ and the asymptotic expansion of
  the branches~\eqref{eq:transcritical}, one concludes that,
  in a neighborhood of $(2 + \lambdarad_i, 1)$,  the
  functions on the branch emanating to the right (resp.\ left) of $2 +
  \lambdarad_i$ satisfy $u(0) < 1$ (resp.\ $u(0) > 1$).
  This property remains true along the whole branch.
  Indeed, since the function $h$ defined
  by~\eqref{eq:h} is non-increasing, one sees that $u(0) \ne 1$
  otherwise $u$ is the constant solution~$u\equiv 1$ which does not
  belong to the branch~$B_{i}$.
\end{proof}

\begin{Rem}
  In contrast with Theorem~\ref{transcritical-radial-bifurc}, the
  bifurcation points for the one-dimensional case are always
  supercritical.  On $B_R = \intervaloo{-R, R}$, all eigenvalues are
  simple.  Let $\phi_i$ be a $i$\,th eigenfunction, sorted so that the
  corresponding eigenvalues $\lambda_i$ are increasing, and normalized
  so that $\norm{\phi_i}_{L^2} = 1$.  As before, we take $\langle\psi,
  v\rangle := \int_{B_R} v \phi_i$.  Elementary but tedious
  computations then show that
  $$a = -\int_{-R}^R \phi_i^2 = -1 \ne 0,\quad
  b = -\frac{1}{2}(1+\lambda_i) \lambda_i \int_{-R}^R \phi_i^3 = 0$$
  and
  \begin{equation*}
    c
    = \tfrac{-1}{6a} (1+\lambda_i) \lambda_i
    \Bigl( - (\lambda_i - 1) \int_{-R}^R \phi_i^4
    - 3 (1+\lambda_i) \lambda_i \int_{-R}^R \phi_i^2 w \Bigr)
    = \frac{\pi^2 i^2}{12\, R^3} + \frac{5 \pi^4 i^4}{192\,R^5}
    + \frac{\pi^6i^6}{768\,R^7}
  \end{equation*}
  where $w$ is any solution of $-w'' - \lambda_iw = \phi_i^2$ with
  Neumann boundary conditions.
\end{Rem}

When $N=2$, we conjecture Theorem~\ref{transcritical-radial-bifurc}
remains but we have to leave that as an open question for which a
positive answer is strongly supported by the numerical
computations.  We can state the following weaker result.

\begin{Thm}
  Assume $\Omega=B_{R}$ and $N =2$.  For every $i \ge 2$,
  $(p_{i},u_{0}) := (2+\lambdarad_i,1)$ is a bifurcation point in
  $\Tilde{\C}^{2,\alpha}(B_R)$ of Problem~\eqref{pblP1}.
  Denote $B_i$ the branch bifurcating from $(2+\lambdarad_i,1)$.
  The following holds:
  \begin{enumerate}[(i)]
  \item close to $(2+\lambdarad_i,1)$, the branch is a $\C^1$-curve;
  \item there exists $\varepsilon>0$ (which does not depend on $i$)
    such that if $(p,u_{p}) \in B_i$ then $u_{p}$ is positive and
    $p>2+\varepsilon$;
  \item in addition to the point $(2+\lambdarad_i,1)$,
    the branch is composed of two connected components, one along
    which $u(0) < 1$ and another one along which $u(0) > 1$;
  \item\label{transcritical-large} if $i$ is large enough, then the
    bifurcation from $(2+\lambdarad_i,1)$ is transcritical and the
    characterization of the branch stated in
    assertion~(\ref{transcritical}) of
    Theorem~\ref{transcritical-radial-bifurc} holds.
  \end{enumerate}
\end{Thm}

\begin{proof}
  The first two assertions follow as in the proof of Theorem~\ref{transcritical-radial-bifurc}.
  Assertion~(\ref{transcritical-large}) is a consequence of
  Remark~\ref{transcritical-N=2}.  Finally,
  Theorem~\ref{bifurc-simple} asserts that the curve
  $s \mapsto (\bar p(s), \bar u(s))$ locally giving the bifurcation
  branch around $(2+\lambdarad_i,1)$ is such that
  \begin{equation*}
    \bar u(s) = 1 + s \phi_i + \order(s)
  \end{equation*}
  where $\phi_i$ can be chosen so that $\phi_i(0) > 0$.
  Thus, $u(0) > 1$ when $s > 0$ and $u(0) < 1$ when $s < 0$.
  The same argument as for Theorem~\ref{transcritical-radial-bifurc}
  implies that these properties are preserved along the corresponding
  continuums.
\end{proof}

\subsection{Properties of the solutions along the branches}

In this subsection, we first show that the branches $B_{i}$ are
unbounded and that they do not cross. We introduce the following
definition which is intended to distinguish the solutions bifurcating
at $(2+\lambdarad_{i},1)$.
\begin{Def}
  A positive radial solution $u$ is of type $i$ if and only if the number of
  zeros of $r \mapsto u(r) - 1$ is the
  same as the number of zeros of ${\phi}_{i}$,
  the radial eigenfunction associated to $\lambdarad_i$. If
  $u$ is of type $i$ and $u(0)>1$ then we say that $u$ is of type $i_{+}$
  while if $u(0)<1$, we say that $u$ is of type $i_{-}$.
\end{Def}

The next proposition states the classical separation of the branches via nodal properties. 

\begin{Prop}
  \label{unbd}
  The branches $B_{i} \subset \IR \times
  \bigl(\C_{\text{\upshape rad}}^{2,\alpha} \setminus\{1\} \bigr)$
  starting from $(2 + \lambdarad_i, 1)$ ($i > 1$)
  are unbounded for the $\C_{\text{\upshape rad}}^{2,\alpha}$-topology and do not
  intersect. Moreover, along the
  branch $B_{i}$, the solutions of \eqref{pblP1} are
  of type $i$.
\end{Prop}
\begin{proof}
  We know that the branches are unbounded in
  $\IR \times \C_{\text{rad}}^{2,\alpha}$
  or linked by pair. To disprove the second possibility, it is enough to
  prove that along $B_{i}$, the solutions are of type $i$.
  Let $(p,u) \in B_i$.
  We know that $p>2+\varepsilon$ and as observed in the proof of
  statement~(iii) of Theorem \ref{transcritical-radial-bifurc}, $u(0)\ne 1$.
  Moreover, as $u$ is a solution of the ODE
  \begin{equation}
    \label{eq:radial-ode}
    -\partial_r^2 u - \frac{N-1}{r}\partial_r u + u = \abs{u}^{p-2} u,
  \end{equation}
  which also possesses the constant solution $1$, the roots of $u-1$
  are simple. Therefore, the number of roots of $u-1$ along the
  $\C^{2,\alpha}$-continuum $B_i$ cannot change.  To prove that the
  number of roots of $u-1$ is the same as $\phi_i$, we consider a
  sequence $\bigl((p_n, u_n)\bigr)_{n} \subset B_i$ converging to
  $(2+\lambdarad_i, 1)$ in~$\IR \times \C^{2,\alpha}$ and we set
  $v_n := (u_n - 1)/\norm{u_n-1}_{\C^{2,\alpha}}$.  Due to the fact
  that the embedding $\C^{2,\alpha} \hookrightarrow \C^{0,\alpha}$ is
  compact, one can assume that $v_n \to v^*$ in~$\C^{0,\alpha}$.
  Recalling that $u_n > 0$, the equation for $v_n$ can be written as
  \begin{equation*}
    v_n = ( -\Delta + \id)^{-1}
    \frac{u_n^{p_n-1} - 1}{\norm{u_n-1}_{\C^{2,\alpha}}}
    = ( -\Delta + \id)^{-1} \bigl( (1 + \lambdarad_i + o(1)) v_n
    + O(\norm{u_n-1}_{\C^{0,\alpha}}) \bigr).
  \end{equation*}

  Since the inverse of the $-\Delta + \id$ is continuous, one deduces that the
  convergence  $v_n \to v^*$ actually occurs in~$\C^{2,\alpha}$, thus
  $\norm{v^*}_{\C^{2,\alpha}} =1$ and $v^*$ is a radial
  eigenfunction of $-\Delta$ with eigenvalue $\lambdarad_i$.
  By simplicity, $v^*$ is a multiple of $\phi_i$ and has the same number
  of zeros.  Since these zeros are simple, $v_n$ also has the same
  number of zeros as $v^*$ for $n$ large. This completes the proof.
\end{proof}

Summing up the previous results, we can distinguish the behavior of
the solutions on the two connected components of the branch~$B_i$.
\begin{Thm}\label{boundary2}
  Let $N \ge 3$.
  The set $B_i$ consists of two branches $B_{i}^{\pm}$ such that
  \begin{enumerate}[(i)]
  \item on $B_{i}^{-}$, $p > 2 + \lambdarad_i$ close to the
    bifurcation point, the solutions $u$ are of type $i_-$, $u(0)$ is
    a global minimum, $u$ has exactly $i$ critical points which are
    all non degenerate local extrema, each maxima (resp.\ minima)
    being strictly
    greater (resp.\ smaller) than~$1$ and strictly smaller
    (resp.\ larger) than the previous one;
  \item on $B_{i}^{+}$, $p < 2 + \lambdarad_i$ close to the
    bifurcation point, the solutions $u$ are of type $i_+$, $u(0)$ is
    a global maximum, $u$ has exactly $i$ critical points which are
    all non degenerate local extrema, each maxima (resp.\ minima) being strictly
    greater (resp.\ smaller) than~$1$ and strictly smaller
    (resp.\ larger) than the previous one.
  \end{enumerate}
\end{Thm}
\begin{proof}
 The proof is a simple consequence of the previous statements and of
  standard ODE arguments based on the energy dissipation \eqref{eq:h}. 
\end{proof}

\begin{Rem}
  The same result holds for $N=2$ except that we do not know how the
  branches $B^\pm_i$ behave for $p$ close to~$2 + \lambdarad_i$.
\end{Rem}

Note that on the first bifurcation, the solutions are increasing on
one part of the branch and decreasing on the other part. On the other
branches, the solutions are oscillating around~$1$ with a ``decreasing
envelope''.

\subsection{Degeneracy and multiplicity}
\label{main}

We now collect some of the consequences of the bifurcation
analysis. In particular, the a priori estimates given in
Section~\ref{aprioriesti} lead to further qualitative results. We
recall that a positive solution $u_{0}$ of \eqref{pblP1} is said to be
degenerate if there exists $v\ne 0$ such that
\begin{equation*}
\left\{
\begin{aligned}
-\Delta v+v&= (p-1)\abs{ u_{0}}^{p-2}v,&&\text{ in } \Omega, \\
\partial_{\nu}v&=0,&&\text{ on } \partial \Omega,
\end{aligned}
\right.
\end{equation*}
or, equivalently, if the kernel of $\partial_uF(p,u_{0})$ is
larger than $\{0\}$.

\begin{Prop}
  \label{deg}
  Let $N \ge 3$, $\Omega=B_{R}$, and $i\ge 2$.
  If $2+ \lambdarad_i(B_{R}) < 2^*$, Problem
  \eqref{pblP1} admits a degenerate positive
  radial solution of type $i_+$ for some $p = p_{i} \in \intervaloo{2,
    2 + \lambdarad_i}$.
\end{Prop}

\begin{proof}
  We know that statement~(\ref{transcritical}) of
  Theorem~\ref{transcritical-radial-bifurc} holds. On
  $B_{i}^{+}$, the branch starting from
  the left of~$(2+\lambdarad_i,1)$, the
  solutions are a priori bounded as long as $p<2^{*}$
  (see Proposition~\ref{apriori-bound}).  On the other
  hand, we know that the branch is unbounded (for the topology of
  $\IR \times \C^{2,\alpha}$) and that
  $p \ge 2+\varepsilon$ for some $\varepsilon>0$
  that depends only on $R$. It follows that $p$ achieves a minimum along
  the continuum at some value $p_{i} \ge 2+\varepsilon$. If the
  corresponding solution $u_{p_{i}}$ was non degenerate, the Implicit
  Function Theorem would allow to extend the branch to the left of
  $p_{i}$ which is a contradiction.
\end{proof}

Observe that $\lambdarad_i(B_{R}) < 2^*-2$ as soon as $R$ is large so
that on large balls, there exist many degenerate positive
solutions. These turning points on the bifurcation diagram should
imply a change of Morse index (for instance in the space of radial
functions) along the branches and a change in the minimax property of
the solution.  This is supported by the numerical computations of
section~\ref{num:1st radial bifurc}. 
It also implies a local multiplicity result for
solutions of type $i_{+}$.

\begin{Cor}\label{multiloc}
  Let $N \ge 3$.
  If $2+ \lambdarad_i(B_{R})< 2^*$, there exists $\varepsilon_{i}>0$ such that if $2+
  \lambdarad_i(B_{R})-\varepsilon_{i}<p<2+
  \lambdarad_i(B_{R})$, Problem \eqref{pblP1} with $\Omega=B_{R}$ has
  at least two positive radial solutions of type $i_+$.
\end{Cor}

The numerical computations of Sections~\ref{num:1st radial bifurc}
and~\ref{num:open-questions} indicate that
$\lambdarad_i(B_{R})-\varepsilon_{i}$ is actually the turning point
$p_{i}$ from Proposition~\ref{deg} but a proof of this fact actually requires a deeper
analysis that we do not pursue here.  When $2+ \lambdarad_2(B_{R})< 2^*$, this means there exist
two decreasing solutions for $p_{2} < p < 2+ \lambdarad_2(B_{R})$.

In Section~\ref{num:open-questions},
we also numerically observe that the solution on the unbounded part of the
branch $B^+_i$, i.e.\ after the turning point, explodes as $p\to 2^{*}$.
The solutions seem to concentrate at the origin when $p\to
2^{*}$. It is therefore natural to conjecture that all these branches bifurcate from
infinity at $p=2^{*}$.

\smallbreak

We next derive a global multiplicity result which answers positively a conjecture in \cite{BonheureNorisWeth} at least in the case of a pure power nonlinearity (we believe that the general case can be derived with similar arguments). Indeed,   
assuming 
\begin{enumerate}[(i)]
\item $f\in \C^1\bigl([0,+\infty),\IR\bigr)$, $f(0)=0$ and
  $\displaystyle f'(0) = \lim_{s\to0^+} \frac{f(s)}{s}=0$;
\item $f$ is nondecreasing;
\item $\displaystyle{  \liminf \limits_{s\to+\infty} \frac{f(s)}{s}>1}$ and there exists $u_0>0$ such that $f(u_0)=u_0$ and $f'(u_0)>
    1+\lambdarad_{2}$;
\end{enumerate}
it is proved in \cite{BonheureNorisWeth} that the problem
\begin{equation}\label{eq:main_equation_a=1}
  \begin{cases}
    -\Delta u+u = f(u) &\text{ in } B, \\
    u>0 &\text{ in } B,\\
    \partial_\nu u=0 &\text{ on } \partial B,
  \end{cases}
\end{equation}
has at least one nonconstant increasing radial solution while the authors conjectured that there 
exists a radial solution with $k$ 
intersections with $u_0$ provided that $f'(u_0) >1+ \lambdarad_{k+1}$.
For the pure power nonlinearity $f(u)=\abs{ u}^{p-2}u$, the condition
is $p-1> 1+\lambdarad_{k+1}$.

\begin{Prop}
  \label{multi}
  Assume $N \ge 2$, $\Omega=B_{R}$ and $n\ge 1$.
  If $\bar p \in \bigintervaloo{2 +
  \lambdarad_{n+1}(B_{R}),\, +\infty}$ then, for all $i=2,\dotsc,n+1$,
  Problem \eqref{pblP1} with $p = \bar p$ has at least one
  non-constant positive solution of type $i_{-}$.
\end{Prop}
\begin{proof}
  Consider the $n$ branches $B_{i}^{-}$ bifurcating from
  $(2+\lambdarad_{i}, 1)$ for $i = 2, 3,\dotsc, n+1$.  Along all
  the branches $B_{i}^{-}$, $u(0)<1$. Since we have an a priori bound
  for such solutions (see Corollary~\ref{cor21}),
  the projection of these branches are
  unbounded in the parameter $p$.  Thus
  each of the branches $B_{i}^{-}$,
  $i=2,\dotsc,n+1$, contains a solution of type $i_{-}$
  to Problem \eqref{pblP1} with $p=\bar p$.
  These solutions are
  non-constant and different because they are distinguished by their type.
\end{proof}

Note that the assumption can be interpreted in term of the size of
the ball. Namely, if
\begin{equation*}
  p>2 \quad\text{and}\quad
  R > \sqrt{{\lambdarad_{n+1}(B_1)} \mathbin{\bigm/} ({p-2})},
\end{equation*}
Problem \eqref{pblP1} possesses at least one non-constant positive solutions
of type $i_{-}$ for $i=2,\dotsc,n+1$.  Indeed, the assumption on $R$
equivalently reads $2 + \lambdarad_{n+1}(B_R) < p$.

\medbreak

If $p<2^{*}$ we can derive a stronger
multiplicity result as the branches $B_{i}^{+}$
give solutions of type $i_{+}$.

\begin{Prop}
  \label{multisubcritical}
  Assume $N \ge 2$, $\Omega=B_{R}$, $n\ge1$ and $2 +
  \lambdarad_{n+1}(B_R)<2^{*}$.  If $2 + \lambdarad_{n+1}(B_R)< p
  <2^{*}$, Problem \eqref{pblP1} has at least one
  positive solution of type $i_{+}$ for $i=2,\dotsc,n+1$.
\end{Prop}
\begin{proof}
  For $i=2,\dotsc,n+1$, the branch $B_{i}^{+}$ gives rise to a family of
  solutions of type ${i}_{+}$ which are a priori bounded as long as
  the branch stays away from $\{2^{*}\}\times
  \Tilde{\C}^{2,\alpha}_{\text{rad}}$.
\end{proof}

Again this result can be interpreted in term of the size of the
ball.

\smallbreak

The proof of Theorem \ref{intro1} can now be achieved by combining
Proposition \ref{multi}, Proposition \ref{multisubcritical} and
Corollary \ref{multiloc}.

\begin{Rem}
  If $\Omega$ is an annulus, Bessel functions of first and second kind
  (see~\cite{gt}) can also be used to give a characterization of the
  eigenfunctions of $-\Delta$. Then, one can do the same bifurcation
  analysis. Numerical computations show that the corresponding values
  of $b$ are not zero, so that we conjecture that the radial bifurcation
  points are all transcritical. Since the critical exponent does not
  play any role here for radial solutions,
  we can even prove a priori bounds for $p>2^{*}$
  and therefore derive the existence of more solutions in this region
  than in the case of the ball.
\end{Rem}

\section{Small diffusion}
\label{sec:small diffusion}

In this section, we consider the singular perturbation
problem~\eqref{pblE}. As mentioned in the Introduction, 
the existence of positive solutions for this
problem as $\varepsilon \to 0$ has already been investigated by many
authors, essentially by using perturbative methods, and different
concentration phenomena have been highlighted,
both with and without symmetry assumptions. Our study here is of a
non perturbative nature and gives some insight on the radial boundary
clustered layer solutions obtained via a Lyapunov-Schmidt
reduction in~\cite{MR2056434,Malchiodi-Ni-Wei05}. In our analysis,
our main goal is not the behavior of the
solutions in the singular limit $\varepsilon\to 0$ though we will
link our result to the existing literature. We rather focus on
the exact values of $\varepsilon$ where new type of radial solutions
appear and survive for smaller values of the diffusion coefficient.

A bifurcation analysis of problem~\eqref{pblE} was performed by Ni and
Takagi \cite{nt} in a general domain (with a slight refinement on
simple rectangles). Since we deal with radial solutions on a ball, we
are able to go much deeper in the analysis of the behavior of the
branches. The radial bifurcation analysis for the problem \eqref{pblE}
with $f(u)=|u|^{p-2}u$ in a ball as been performed by Miyamoto
\cite{Miyamoto}. The complete picture is given 
in~\cite[Theorem~B]{Miyamoto} when $p$ is supercritical.
In that case, all radial
regular solutions of \eqref{pblE} lie on branches that bifurcate
from the constant solution $u=1$. Each branch of solutions
$(\varepsilon,u_\varepsilon)$ can be parametrised by
$\bigl(\varepsilon,u_\varepsilon(0)\bigr)$. The other main concern of
\cite{Miyamoto} is a careful analysis of the upper half-branches of
the bifurcation diagram (i.e. the parts of the branches where
$u_\varepsilon(0)>1$) when $2^* < p < 2^*_{JL}$ where
$2^*_{JL} := 2 + 4/(N-4 - 2\sqrt{N-1})$ if $N \ge 11$
and $2^*_{JL} := +\infty$ if $2 \le N \le 10$, is the critical exponent of
Joseph and Lundgren. We will rather focus on the lower parts of the
branches (i.e.\ the parts of the branches where $u_\varepsilon(0)<1$)
as those exist for a wide class of nonlinearity.

In this Section, we only consider radial solutions, so that by a
solution of Problem~\eqref{pblE}, we necessarily mean a \emph{radial}
solution.

\medbreak

Without loss of generality,
we assume throughout this section that \eqref{F1} is satisfied with
$u_{0}=1$, namely $f(1)=1$ and $f'(1)>1$ which in particular implies
that $u \equiv 1$ is a solution for all $\varepsilon$.  We
investigate locally the bifurcations from
$u_{0}=1$ and then follow some of their associated global branches.
We only focus on the lower part of the branches of solutions, namely those that survive as
$\varepsilon\to 0$ without having to impose a growth condition
on $f$ at infinity. We
can also easily study the upper part of the branches at the cost of some additional
assumptions on the growth of~$f$ at infinity. For instance, we will
comment, at the end of the section, the special case $f(u)=u^{p-1}$
where $p$ is subcritical. On the other hand, when $f$ has a critical
or supercritical growth, the analysis of the upper part of the branch
is much more involved and blow up may occur (and actually does, see
\cite[Theorem B]{Miyamoto}) at some $\varepsilon^{*}>0$.

In the sequel, we assume that $f$ is of class $\C^{1}$.
The assumption $f'(1)>1$ implies that $f$ is
locally super linear. It will be seen that it is a necessary and
sufficient condition for the (local) existence of branches of
solutions bifurcating from the trivial one at positive values of the parameter. Since many of the
arguments needed to treat Problem~\eqref{pblE} are similar to those
used in Section \ref{bifu} and in \cite{Miyamoto},
we will only sketch the arguments in this section.

Consider the map
\begin{equation*}
  G: \intervaloo{0, +\infty} \times \Tilde{\C}^{2,\alpha}_{\text{rad}} \to
  \C^{0,\alpha}_{\text{rad}} : (\epsilon, u) \mapsto
  (-\epsilon\Delta + \id) u - f(u).
\end{equation*}
Clearly, the function $u\in \Tilde{\C}^{2,\alpha}_{\text{rad}}$ is a
classical solution to Problem~\eqref{pblE} with $\Omega=B_{R}$ if and
only if the couple $(\varepsilon,u)$ is a zero of the function $G$ and
$u>0$. The positivity of $u$ will again be checked a posteriori. We set
\begin{equation}
  \label{eq:epsiloni}
  \varepsilon_{i} := \frac{f'(1)-1}{\lambdarad_i}
  \quad\text{for } i > 1.
\end{equation}
Classical bifurcation
theory implies that $(\varepsilon_{i}, 1)$ is a bifurcation point in
$\Tilde{\C}^{2,\alpha}_{\textup{rad}}$.  Again we can improve this
first insight by using Crandall-Rabinowitz's Theorem. For that
purpose, we compute
\begin{equation}\label{duG}
  \partial_uG(\varepsilon_{i}, 1)[v]
  = \frac{f'(1)-1}{\lambdarad_{i}}
  \bigl(-\Delta v - \lambdarad_{i} v\bigr).
\end{equation}
Keeping the notation of Section \ref{Section:radialbif}, we have
 $$\ker\bigl(\partial_uG(\varepsilon_{i}, 1)\bigr)
  = \langle \phi_i\rangle,$$
where we still assume that $\phi_i(0) > 0$
and that $\phi_i$ is normalized in $L^{2}(B_{R})$;
  $$\codim\bigl(\Ran(\partial_uG(\varepsilon_{i}, 1))\bigr) = 1$$
 and
$g\in \Ran\bigl(\partial_uG(\varepsilon_{i}, 1)\bigr)$
if and only if $\int_{B_R} g \phi_i = 0$
so that we still take
\begin{equation*}
  \psi : {\C}^{0,\alpha}_{\text{rad}} \to \IR :
  g\mapsto \langle\psi, g\rangle := \int_{B_R} g \phi_i.
\end{equation*}
 Simple computations also show that
\begin{equation}\label{asingpert}
  a = \int_{B_R} \partial_{\epsilon u}G(\varepsilon_{i},1)[\phi_i] \, \phi_i
  = \lambdarad_{i}\int_{B_R}\phi_i^2
  = \lambdarad_{i}
  \ne 0
\end{equation}
(remember that $i > 1$) and
\begin{align}\label{b-epsilon}
  b = -\frac{1}{2a}
  \int_{B_R} \partial^2_{u}G(\varepsilon_{i},1)[\phi_i,\phi_i]\phi_i
  \intd x
  = \frac{f''(1)}{2\lambdarad_{i}}\int_{B_R}\phi_i^3.
\end{align}
Recall that the property $\int_{B_R}\phi_i^3 > 0$ was established in
the proof of Theorem~\ref{transcritical-radial-bifurc} for $N\ge 3$.
%
Arguing as in Section~\ref{bifu}, we can easily prove the following
statement.

\begin{Prop}\label{thm2E}
  Assume $N \ge 2$ and that \eqref{F1} and \eqref{F2} hold with $u_{0}=1$.
  Let $i \ge 2$.
  Then $(\varepsilon_{i},1) =
  \bigl(\frac{f'(1)-1}{\lambdarad_i}, 1\bigr)$ is a radial bifurcation
  point in $\Tilde\C^{2,\alpha}(B_R)$ of
  Problem~\eqref{pblE}.  Letting
  $C_i \subset \IR \times \bigl(\Tilde{\C}^{2,\alpha}_{\textup{rad}}
  \setminus\{1\} \bigr)$ denote the bifurcating branch,
  the following assertions hold:
  \begin{enumerate}[(i)]
  \item close to $(\epsilon_i, 1)$, $C_i$ is a $\C^0$-curve
    (even $\C^1$ if $f$ is of class $\C^2$);
  \item the set $C_i$ consists in two
    connected components $C_{i}^{\pm}$ such that, the solutions $u$ on
    $C_{i}^{-}$ satisfy $u(0) < 1$ while, along $C_{i}^{+}$, one has
    $u(0) > 1$.
  \end{enumerate}
\end{Prop}
%

\begin{Rem}
  If $N \ge 3$, $f$ is of class $\C^2$ around~$u_0 = 1$,
  and $f''(1) \ne 0$, the bifurcation points are
  transcritical and on the part of the branch that bifurcates to the
  left of $\varepsilon_{i}$, we have
  $\sign\bigl(u_{\varepsilon}(0) - 1\bigr) = -\sign(f''(1))$, while on
  the part of the branch bifurcating to the right of
  $\varepsilon_{i}$, the branch is made of solutions satisfying
  $\sign\bigr(u_{\varepsilon}(0) - 1\bigr) = \sign(f''(1))$.
  We conjecture that this remains true for $N=2$.
\end{Rem}

Still arguing as in Section \ref{bifu}, we can derive further
properties of the solutions along the branches.
\begin{Prop}
  \label{thm3E}
  Assume $N\ge 2$, that \eqref{F0},
  \eqref{F1} and \eqref{F2} hold with $u_{0}=1$.
  Then all the branches $C_{i}$, $i \ge 2$,
  are unbounded and do not intersect each other.  Moreover,
  along $C_{i}$, the solutions are of type $i$.  More precisely,
  \begin{enumerate}[(i)]
  \item the solutions on the branch $C_{i}^{+}$ are of type $i_+$,
    $u(0) > 1$ is a global maximum, $u$ has exactly $i$ critical
    points which are all non degenerate local extrema, each maxima
    (resp.\ minima) being strictly greater (resp.\ smaller) than~$1$
    and strictly smaller (resp.\ larger) than the previous one;
  \item the solutions on the branch $C_{i}^{-}$ are of type~$i_-$,
    $u(0) < 1$ is a global minimum, $u$ has exactly $i$ critical
    points which are all non degenerate local extrema, each maxima
    (resp.\ minima) being strictly greater (resp.\ smaller) than~$1$
    and strictly smaller (resp.\ larger) than the previous one.
  \end{enumerate}
  Moreover $\varepsilon\to
  0$ along the branch $C_{i}^{-}$ in the sense that
  the projection of $C_i^-$ on the $\epsilon$-axis contains
  $\intervaloo{0, \epsilon_i}$.
\end{Prop}

\begin{proof}
  We first observe that due to the assumption~\eqref{F0}, arguing
  as in Section~\ref{bifu}, the solutions remain positive along the
  branches. The fact that the solutions are of type $i$
  and the behavior of the extrema
  is also proved as before.
  The only assertion which deserves maybe more attention is the fact that
  $\varepsilon\to 0$ along the branch $C_{i}^{-}$. Since the solutions
  on $C_{i}^{-}$ are of type $i_-$, we infer from Lemma
  \ref{Lem:uniqeps} that there exists $\overline{\varepsilon}$ such that if
  $\varepsilon\ge \overline{\varepsilon}$, any solution of type $i_-$ is
  constant. As there always exists an interior point where the solution is above $1$ along the continuum,  
  and since the continuum cannot return to the constant solution $1$, we conclude that $\varepsilon$ is a priori bounded along the branch $C_{i}^{-}$.
  Since the branch is unbounded and
  Proposition \ref{apriori:bound:small:diffus} with
  $\varepsilon_0 = \overline{\varepsilon}$
  provides an a priori bound along the branch, we conclude that
  the branch must contain points $(\epsilon, u)$ for any 
  $0<\epsilon<\epsilon_i$.
\end{proof}

This bifurcation analysis directly leads to the following qualitative
and quantitative result for Problem~\eqref{pblE}
which is the natural counterpart to Proposition~\ref{multi}
for~\eqref{pblE}.

\medbreak

\begin{Cor}
  \label{multiplicity:espilon->0}
  Assume $N\ge 2$, that \eqref{F0},
  \eqref{F1} and \eqref{F2} hold with $u_{0}=1$.
For any $n \ge 1$ and any $\varepsilon
< ({f'(1)-1})/{\lambdarad_{n+1}}$, there exists at least one
solution of Problem~\eqref{pblE} of type $i_-$ for every $2\le i\le n+1$.
\end{Cor}

Again this result can be seen as depending on the size of the ball,
namely, for any $\varepsilon >0$ and any $n \ge 1$,
if $R > \sqrt{\epsilon \lambdarad_{n+1}(B_1) / (f'(1) - 1)}$,
there exists at least one solution of type $i_-$ on $B_R$
for any $i = 2,\dotsc, n+1$.

\medskip

We now show that the branches $C_{i}^{-}$ contain all possible
solutions $u$ such that $u(0)<1$.

\begin{Thm}\label{Thm:classCi-}
  Assume $N\ge 2$, \eqref{F0}, \eqref{F1} and \eqref{F2} with
  $u_{0}=1$ and $f \in \C^{1,1}$ satisfies $\forall s \in
  \intervaloo{0, 1},\ f(s) \ne s$.
  If $u$ is a solution to Problem~\eqref{pblE} such that
  $u(0)<1$, then $u$ lies on a branch $C_{i}^{-}$ for some $i\ge 2$.
\end{Thm}
\begin{proof}
  Assume $u$ is a solution to~\eqref{pblE} for some $\varepsilon >0$
  and let $\gamma_{0} = u(0) \in \intervaloo{0,1}$.
  From our assumptions, $f(\gamma_0) \ne \gamma_0$ and so $u$ is non-constant.
  Thus \cite[Proposition~3.1]{Miyamoto}
  --- which is valid for any $f$ of class $\C^{1,1}$ ---
  says that $(\epsilon,u)$ can be uniquely continued locally:
  there is a local $\C^{1}$ parametrization
  \begin{equation*}
    \Gamma :
    \intervaloo{\gamma_{0}-\eta,\gamma_{0}+\eta}\to \IR^{+}\times \C^{2} :
    \gamma \mapsto \bigl(\tilde\varepsilon({\gamma}),
    \tilde u({\gamma},r) \bigr),
  \end{equation*}
  where $\Gamma(\gamma_{0}) = (\varepsilon,u)$,
  $r \mapsto \tilde u({\gamma},r)$
  is the unique solution to~\eqref{pblE} with
  $\varepsilon = \tilde\varepsilon({\gamma})$
  and $\tilde u(\gamma, 0) = \gamma$.
  Repeating the same argument, one sees that the map $\Gamma$
  extends to $\Gamma : \intervaloo{0,1} \to \IR^+ \times \C^2$.
  Lemma~\ref{Lem:uniqeps} implies that $\forall \gamma \in
  \intervaloo{0,1},\ \tilde\varepsilon(\gamma) \le \overline{\varepsilon}$.
  Assume (we will prove it below) that $\tilde\varepsilon(\gamma)$ is
  bounded away from $0$ as $\gamma \to 1$.  Thus limit points
  of $\tilde\varepsilon(\gamma)$ as $\gamma \to 1$ exist and they all lie in
  $\intervalcc{\thinspace \underline{\varepsilon},
    \overline{\varepsilon} \thinspace}$
  for some $\underline{\varepsilon} > 0$.
  Thanks to Proposition~\ref{apriori:bound:small:diffus},
  for any such limit point $\varepsilon^* = \lim
  \varepsilon(\gamma_n)$, the solutions $\tilde u(\gamma_n, \cdot)$
  converge, up to a subsequence, to a solution $u^*$
  to Problem~\eqref{pblE} with $\varepsilon = \varepsilon^*$
  and $u^*(0) = 1$.  Moreover, as these functions belong to the
  continuum $\Gamma(\intervaloo{0,1})$ and are solutions of a second
  order ODE, the number of zeros of $\tilde u(\gamma, \cdot) - 1$ does
  not depend on $\gamma$.
  Thus $u^* = 1$ and $\varepsilon^*$ must be the
  bifurcation value $\varepsilon_i$ for the~$i \ge 2$ for which
  the eigenfunction $\phi_i$ has the same number of zeros as
  $\tilde u(\gamma, \cdot) - 1$.
  So all limit points of $\tilde\varepsilon(\gamma)$ are the same
  $\varepsilon_i$ and consequently
  $\tilde\varepsilon(\gamma) \to \varepsilon_i$ as
  $\gamma \to 1$.  By local uniqueness near the bifurcation point,
  the curve parametrized by $\Gamma$ coincides with
  the branch $C_i^-$ emanating from $(\varepsilon_i, 1)$.
  
  To complete the proof, assume by contradiction that there is a sequence of solutions $(\varepsilon_{n},u_{n})$ such that 
  $\varepsilon_{n}\to 0$, $u_{n}(0)\to 1$.
  As above, the number of zeros of $u_{n}-1$ is the same for all~$n$.
  The convergence $u_{n}(0)\to 1$ actually implies the uniform
  convergence of $u_{n}$ to $1$ because of~\eqref{eq:h}.
  Now, as in Ni~\cite{tech-Ni}, let us consider the function
  $v_{n} (r) := r^{(N-1)/2} \bigl(u_{n}(r) - 1
  \bigr)$.  This function solves
\begin{equation*}
  v''
  + \biggl( \frac{g(u_n(r))}{\varepsilon_n}-\frac{(N-1)(N-3)}{4r^2} \biggr) v
  = 0
  \qquad
  \text{where }
  g(u) :=
  \begin{cases}
    \displaystyle
    \frac{f(u) - u}{u - 1}, &\text{if } u \ne 1,\\[2\jot]
    f'(1) - 1& \text{if } u = 1.
  \end{cases}
\end{equation*}
  Recalling that $f'(1) > 1$, one gets that
  for any $r_0>0$ and $M>0$, there exists $n_0\in\IN$
  such that for $n\ge n_0$, the functions $u_n$ satisfy
  \begin{equation*}
    \forall r \ge r_0,\qquad
    \frac{g(u_n(r))}{\varepsilon_n}-\frac{(N-1)(N-3)}{4r^2}>M^2.
  \end{equation*}
  By Sturm comparison theorem, one deduces that the distance between
  two consecutive zeros of $u_{n}-1$ for $r\ge r_0$ is bounded from
  above by $2\pi/M$.
  Since $M$ can be taken arbitrarily large, we infer
  that the number of zeros of $u_n-1$ cannot remain constant for
  large~$n$.
\end{proof}
We stress that the previous theorem does not state the uniqueness of
the solution of type $i_{-}$ as, even if each branch can be
parametrized as a $\C^{1}$ curve
$\bigl(\varepsilon({\gamma}), u({\gamma},r) \bigr)$,
i.e.\ secondary bifurcations
are excluded, turning points may occur. We strongly believe, and this
is supported by numerics, that uniqueness holds but we have to leave
this as a conjecture for now.

\medbreak

When $f(u) = \abs{u}^{p-2}u$ and $p$ is supercritical,
Miyamoto also obtained
the classification of the solutions such that $u(0)>1$, see
\cite[Theorem B]{Miyamoto}, leading to the complete picture of
positive solutions.
When $p$ is subcritical, we can also complete the classification.

\begin{Prop}
  Assume $N\ge 2$, $f(u)=|u|^{p-2}u$ and $2<p<2^{*}$. If $u$ is a
  solution to Problem~\eqref{pblE}, then either $u=1$ or $u$ lies on a
  branch $C_{i}$ for some $i\ge 2$.
\end{Prop}

\begin{proof}
  Let $u$ be a solution to~\eqref{pblE} for some $\varepsilon >0$ and
  let $\gamma_{0} := u(0)$. If $\gamma_{0}<1$, then
  Theorem~\ref{Thm:classCi-} gives the conclusion. Assume therefore that
  $\gamma_{0} >1$. Again, \cite[Proposition 3.1]{Miyamoto} implies
  this solution $(\varepsilon,u)$ can be uniquely continued locally as
  a curve $\gamma \mapsto \bigl(\varepsilon({\gamma}),
  u({\gamma},r) \bigr)$ and then extended
  as long as we have a priori bounds. It is proved in
  \cite{Lin-Ni-Takagi88} that there exists $\varepsilon_{0}>0$ such
  that for $\varepsilon>\varepsilon_{0}$, Problem~\eqref{pblE}
  only admits the constant solution $u=1$. Then, by Proposition
  \ref{apriori-bound}, we have a priori bounds as long as
  $\varepsilon$ is bounded away from zero. As a consequence, arguing as in the proof of Theorem \ref{Thm:classCi-}, one shows the curve
  $\gamma \mapsto \bigl(\varepsilon({\gamma}), u({\gamma},r)\bigr)$
  can be continued at least
  up to one of the bifurcation points
  $(\varepsilon_i, 1) = \bigl(\frac{p-2}{\lambdarad_i}, 1\bigr)$.
  In particular, $u$ lies
  on a curve $C_{i}^{+}$.
\end{proof}

Since, in the case where $f(u)=|u|^{p-2}u$ with $p<2^{*}$, there
exists a unique entire solution $w$ of the equation on the whole
space, it is easily seen that along the branches $C_{i}^{+}$,
solutions $u$ satisfy $u(0) \to w(0)$ as $\varepsilon\to 0$.
This is illustrated by the computer generated
Figures~\ref{fig:bifurcation-eps}, \ref{fig:rad:epsilon->0}
and~\ref{fig:bifurcation-eps-profile}.

In contrast to the subcritical and supercritical case, as pointed
out in the Introduction, the existence of positive solutions for
$p=2^{*}$ depends on the dimension and not only on $\varepsilon$
(or, equivalently, the size of the ball).

\medbreak

We now briefly turn to the description of the behavior of the
solutions along the branches $C_{i}^{-}$ as $\varepsilon\to 0$.  We
claim that for any $n \ge 2$, the family of solutions of type $i_-$
bifurcating from $\frac{f'(1)-1}{\lambdarad_i}$ is such that the local
maxima cluster around the boundary as $\varepsilon\to 0$. Miyamoto
proved the branch $C_{2}^{-}$ is asymptotically made of increasing
boundary concentrating solutions. Numerical evidence of those facts
are shown in
Section~\ref{sec:concentration-singular-perturb}. In the case
$f(u)=|u|^{p-2}u$, we refer to \cite[Corollary~1.3]{MR2056434} and
\cite{Malchiodi-Ni-Wei05} for the construction, via a Lyapunov-Schmidt
procedure, of solutions with one or multiple interior layers and to
\cite[Corollary~7.11]{Miyamoto} for the construction of a family of
increasing solutions concentrating on the boundary. The result in
\cite{Malchiodi-Ni-Wei05} is valid in our setting and not only for a
pure power.
Combining the arguments of \cite{Malchiodi-Ni-Wei05}
and~\cite{Miyamoto}, one should be able to construct, by reduction,
even more solutions, namely solutions with a prescribed number of
interior layers and a boundary layer.

As a consequence of the previous theorem, the solutions
of Malchiodi, Ni and Wei~\cite{Malchiodi-Ni-Wei05},
concentrating on spheres when $\varepsilon\to 0$,
belong to the branches $C_{i}^{-}$ for odd $i$'s.
The following must therefore hold.
If $(\varepsilon, u_\varepsilon)$ is a family of solutions belonging
to $C_i^-$ and $r^\varepsilon_1 > \cdots > r^\varepsilon_{\lfloor i/2 \rfloor}$
are the local maximums of $u_\varepsilon$, then the following estimates of
Malchiodi, Ni and Wei~\cite{Malchiodi-Ni-Wei05} should be valid:
\begin{equation*}
  1 - r^\varepsilon_1 \sim \varepsilon \log\frac{1}{\varepsilon},
  \qquad
  r^\varepsilon_{j-1} - r^\varepsilon_j
  \sim \varepsilon \log\frac{1}{\varepsilon},
  \quad
  j > 1.
\end{equation*}
These asymptotic estimates should also hold for the branches
$C_{i}^{-}$ for even $i$'s. In these cases, the solutions also have a
local maximum on the boundary.

In Section \ref{expIl}, we give numerics for more general nonlinearities $f$. More clustering solutions may exist when $f$ has more fixed points between $0$ and $1$. We consider either degenerate or nondegenerate additional fixed points of $f$.



\section{Symmetry of least energy solutions}
\label{symbre}

When $p<2^*$, a \textit{least energy solution} is a minimizer of the energy $\E_p$ on the
\textit{Nehari  manifold} $\mathcal{N}_p$ defined by
\begin{equation*}
  \mathcal{N}_p:= \bigl\{u\in H^1\setminus\{0\} \bigm| \langle
  \E_p'(u),u\rangle=0 \bigr\}= \Bigl\{u\in H^1\setminus\{0\} \Bigm|
  \int_{\Omega}\abs{\nabla u}^2+u^2 = \int_{\Omega}\abs{u}^p
  \Bigr\}.
\end{equation*}

It is standard to prove that least energy solutions do not change sign. At the critical exponent $p=2^*$, as already mentioned,
X.~J.~Wang~\cite{wang91} also recovered compactness to get the existence of a positive ground state solution. Remember that $u=1$ is the unique positive solution for $p$ close to~$2$ whence it is the least energy solution.

We now investigate when the least energy are not constant and if they are radially symmetric or not when the domain is a ball. The question of the symmetry breaking has been tackled by M.~Esteban in the case where the domain is the exterior of a ball \cite{ZelatiEsteban,Esteban,Esteban2}. In this case, the least energy solution is never a radial function, whatever $p$
is.

Concerning the Neumann problem in a ball, Lopes \cite{Lopes} showed that any non constant radially symmetric critical point of $\E_p$ cannot be a local minimizer on $ \mathcal{N}_p$. This implies that as soon as we can prove that a least energy solution $u$ is not constant, we have a symmetry breaking result.

In this section, we adapt the
results of A.~Aftalion and F.~Pacella~\cite{aftalion} to the Neumann boundary condition. We show that
a radial positive solution with a Morse index less than $N+1$ must be
constant. The method of \cite{aftalion} allows to consider more general assumptions than in \cite{Lopes} whereas in the setting of \cite{Lopes}, this approach provides an
alternative proof of the symmetry breaking.

We first observe that least energy solutions of \eqref{pblP1} are not constant for $p >2+\lambda_2$. This is true on a general domain and was already pointed out in \cite{Lin-Ni}.
Indeed, by definition, the Morse index of a critical point $u$ of the functional $\E_p$
  corresponds to the sum of the dimensions of the eigenspaces
  associated to the negative eigenvalues $\mu$ of the problem
  \begin{equation*}
  \tag{\protect{$\mathcal{P}'_p$}} \label{pblPD2}
  \begin{cases}
    -\Delta h+h -(p-1)\abs{u}^{p-2}h = \mu h,&\text{in }\Omega, \\
    \partial_{\nu} h=0, &\text{on } \partial \Omega.
  \end{cases}
\end{equation*}
With $u=1$, the solutions $h$ of
  Problem~\eqref{pblPD2} are the eigenfunctions of $-\Delta$
  associated to the eigenvalue $p-2+\mu$. Therefore,
for every $i\in\IN\setminus\{0\}$, $\mu_i =
  \lambda_i-(p-2)$ is an eigenvalue and its multiplicity is that of
  $\lambda_i$.
This implies that for any $i\in\IN\setminus\{0\}$,
if $\lambda_i<p-2\leq \lambda_{i+1}$,
  the Morse index of $1$ is equal to $\sum_{k\leq i}\dim E_k$.

As least energy solutions have a Morse index equal to $1$, the constant solution $u=1$ cannot be a least energy solution of \eqref{pblP1} when $p > 2 + \lambda_2$.
We next focus on the question of the symmetry of non constant least energy solutions.
We consider the problem
\begin{equation*}
\tag{\protect{$\mathcal{P}_f$}} \label{pblPf} \left\{
\begin{aligned}
-\Delta u+u&= f(u),&&\text{ in } B_1, \\
u&>0, &&\text{ in } B_1, \\
\partial_{\nu}u&=0,&&\text{ on } \partial B_1
\end{aligned}
\right.
\end{equation*}
where $B_1$ is the unit ball centered at the origin.
We assume that the nonlinearity $f\in \C^1(\IR)$ satisfies $f(0)\ge 0$.
For any $i\in \{1,2,\linebreak[2]\ldots,N\}$, we denote
\begin{align*}
  \Omega_i &= \{x=(x_1,\ldots, x_N) \in B_1 \mid x_i=0\},\\
  \Omega_i^+ &= \{x=(x_1,\ldots, x_N) \in B_1 \mid x_i>0\},\\
  \intertext{and}
  \Omega_i^- &= \{x=(x_1,\ldots, x_N) \in B_1 \mid x_i<0\}.
\end{align*}
Let $Lv:= -\Delta v + V(x) v$ where $V \in \C(\bar \Omega)$
is even with respect to $x_i$.
 Let us denote by $\mu_i$ the first eigenvalue of $L$ in $\Omega_i^+$ with
zero Dirichlet boundary conditions on $\Omega_i$ and zero Neumann boundary
  conditions on $\partial \Omega_i^+ \setminus \Omega_i$ and $\psi_i$
a first eigenfunction associated to $\mu_i$.
Let $\psi_{i}^*$ denote the odd extension with respect to $x_i$ of
$\psi_i$ over $B_1$.

\begin{Lem}
  \label{lemnotradial}
Assume $V \in \C(\bar \Omega)$. Then, $\psi_{i}^*$ is an eigenfunction of $L$ in
$B_1$ with Neumann boundary conditions, but not a first one. Moreover, if $V$ is even with respect to the variables $x_1,\ldots, x_k$, $1\le k\le N$, the corresponding functions $\psi^*_1,\ldots, \psi^*_k$
are $k$ independent eigenfunctions of~$L$ (none of which is a first eigenfunction).
\end{Lem}

\begin{proof}
  As the potential $V \in L^{N}(\Omega)$, the variational
  formulation for the
  first eigenvalue of $L$ implies that the corresponding eigenspace is
  one-dimensional and all eigenfunctions do not change sign.
  We clearly have $L(\psi_i^*)=\mu_i \psi_i^*$ on
  $B_1\setminus
  \Omega_i$ and $\psi_{i}^{*}$ satisfies the Neumann boundary conditions on
  $\partial B_1$. It remains to verify
  that $L(\psi_i^*)=\mu_i\psi_i^*$ on $\Omega_i$. As $\psi_i^*=0$ on
  $\Omega_i$, we have
  \begin{equation*}
    \forall j\ne i,\quad
    \partial_{x_j}\psi_i^*=\partial^2_{x_j}\psi^*_i=0.
  \end{equation*}
  As $\psi_i \in \C^1(\bar \Omega_i^+)$ and $\psi_i^*$ is odd, $\psi_i^*
  \in \C^1(B_1)$.  Moreover, since $\psi_i$ is an eigenfunction of
  $L$ on $\Omega_i^+$, the equation tells
 that $\partial_{x_i}^2 \psi_i=0$ on $\Omega_i$.  So,
  $L(\psi_{i}^*)=\mu_i \psi_{i}^*$ on the whole of $B_1$. As $\psi_i^*$ is
  sign-changing, it is not a first eigenfuntion.  This concludes
  the first part of the proof.

  The independence of the functions $\psi^*_i$, $i\in\{1,\ldots,k\}$,
  follows from the fact that $\psi_{i}^*$ is the sole function among
  them which is not identically equal to zero on the $x_i$-axis.
\end{proof}

\begin{Thm}
  \label{miN1}
  If $u$ is a non constant positive radial
  solution of~\eqref{pblPf}, its Morse index is at least $N+1$.
\end{Thm}
\begin{proof}
  Let 
  $$L:= v\mapsto -\Delta v + v - f'(u)v$$ 
  be the linearized operator around $u$ associated to~\eqref{pblPf}.
  Lemma~\ref{lemnotradial} implies the existence of
  $N$ linearly independent eigenfunctions $\psi_i^*$ (none of which is
  a first eigenfunction) for $L$ with zero Neumann boundary
  conditions.  We aim to show that the corresponding eigenvalues
  $\mu_{i}$ are negative.  If we do so, the proof will be complete
  because none of the eigenvalues $\mu_i$ corresponds to a first
  eigenfunction, so the first eigenvalue, which is smaller than all
  $\mu_i$, is also negative.

  Take $i\in\{1,\ldots, N\}$.  Recall that $\mu_i$ is the first
  eigenvalue of $L$ on $\Omega_i^+$ with zero Dirichlet boundary
  conditions on $\Omega_i$ and zero Neumann boundary conditions on
  $\partial \Omega_i^+ \setminus \Omega_i$ (hereafter referred to as
  ``mixed boundary conditions'').

  We know that $\partial_{x_i} u=0$ on $\Omega_i$ and $\partial B_1$
  because of the boundary conditions and the fact that $u$ is radial.
  Now, pick $\Bar{x}\in \Omega_{i}^+$ such that
  $\partial_{x_i}u (\Bar{x})\neq 0$. Such a point exists because $u$
  is radially symmetric and not constant. Let $D$ be the connected
  component of $\{x \mid \partial_{x_i}u(x)\neq 0\}$ containing
  $\Bar{x}$.  Taking partial derivatives in the
  equation~\eqref{pblPf}, we infer that $L(\partial_{x_i}u)=0$.
  Since $\partial_{x_i}u$ does not change sign on
  $D$, we infer that $0$ is the first eigenvalue of $L$ in $D$ with
  zero Dirichlet boundary conditions.
  As $D\subset \Omega_i^+$, the first eigenvalue of $L$ in
  $\Omega_i^+$ with zero Dirichlet boundary conditions is non-positive.
  This in turn implies that $\mu_i < 0$.
  Indeed, if this was not true, then the variational formulation of the
  first eigenvalue would imply that the extension of
  $\partial_{x_i} u$ by $0$ on $\Omega_i^+ \setminus D$ gives a first
  eigenfunction of $L$ on $\Omega^+_i$ with both Dirichlet and mixed
  boundary conditions.
  This contradicts H\"opf's Lemma on $\partial\Omega_i^+
  \setminus \Omega_i$.
\end{proof}

Going back to Problem~\eqref{pblP1}, we conclude that, because they
have Morse index~$1$, least energy solutions are either constant or
non-radial.  In particular, least energy solutions cannot be radial
in the range $2+\lambda_2 < p \le 2^*$.






\section{Numerics, conjectures and open questions}
\label{expIl}

In this section, we complete our theoretical study with some
numerical computations. These lead to some further observations and
conjectures.

\subsection{Is the first bifurcation responsible of the symmetry breaking?}
\label{sec:first-bifurc}

We have seen that the constant state $u=1$ is not a least energy solution for $p> 2 + \lambda_2$. A natural question is whether $2 + \lambda_2$ is optimal.  Observe
that the first bifurcation starting from $1$ (which is not a radial
bifurcation) occurs at $(p,u)=(2+\lambda_2, 1)$, see
Section~\ref{bifu}. It is therefore natural to think that the
solutions along this bifurcation branches provide least energy solutions.

In this subsection, we first investigate, on a ball $B_R$, whether this bifurcation is super-critical which is a crucial step to understand the optimality of $2+\lambda_2$. If we apply the Proposition~\ref{bifurc-simple} at the non-radial bifurcation point $(p,u)=(2+\lambda_2,1)$, we get that
$b=0$ as second eigenfunctions of $-\Delta$ are odd with respect to a
diameter. We thus need to compute $c$.
We denote the second eigenfunction and eigenvalue of $-\Delta$ with
Neumann boundary conditions on $B_1$ by $\Bar\phi_2$ and
$\Bar{\lambda}_2 = \lambda_2(B_1) >0$, $\Bar\phi_{2}$ being normalized
so that $\norm{\Bar\phi_2}_{L^2}=1$.
Let $w$ and $\Bar w$ be respectively the solutions to
$$
  -\Delta w -\lambda_2  w
  = \phi_2^2,\ x\in B_{R},\quad\partial_{\nu}w=0,\ x\in\partial
  B_{R}
$$
and
$$-\Delta \Bar w- \Bar{\lambda}_2 \Bar{w} =
\bar{\phi}_2^2,\ x\in B_{1},\quad\partial_{\nu}\bar w=0, x\in\partial B_{1}.$$  
Then we easily get that
\begin{equation*}
  \begin{split}
    c&= \frac{1}{6} (1 + \lambda_2) \lambda_2
    \Bigl( (\lambda_2-1)\int_{B_R}\phi_2^4 - 3 (1+\lambda_2) \lambda_2
    \int_{B_R}\phi_2^2 w\Bigr)\\
    &=\frac{1}{6}\Bar{\lambda}_2 R^{-(N+2)}
    \Bigl(1 + \frac{\Bar{\lambda}_2}{R^2}\Bigr)
    \Bigl( (\beta -  \alpha)
    \frac{\Bar{\lambda}_2}{R^2} + \beta + \alpha\Bigr)
  \end{split}
\end{equation*}
where $\alpha := \int_{B_1} \Bar\phi_2^4$ and $\beta := -3
\Bar{\lambda}_2 \int_{B_1} \Bar\phi_2^2 \Bar w$.

A numerical computation of $c$ leads to Figure~\ref{compC}.  One
can remark that, for $N=2$, the bifurcation seems to always be
super-critical whatever $R > 0$ is.
For larger $N$, the computations show that the
bifurcation should be super-critical, except for small radii $R$.
However, for these radii, we have $2 + \lambda_2 > 2^*$.  Indeed
$2 + \lambda_2(R) < 2^*$ holds if and only if one is to the right of the
dot on the curve.

\begin{figure}[htb]
  \centering
  \tikzsetnextfilename{bifurcation-c}
  \begin{tikzpicture}[x=7ex, y=2.5ex]
    \draw[->] (0,0) -- (4.8,0) node[below right, xshift=-2ex]{%
      $R$, radius of the ball};
    \draw[->] (0,-4.3) -- (0,5.5) node[left]{$R^{N+2} c$};
    \foreach \x in {1,...,4}{%
      \draw (\x,0.2) -- (\x,-0.2) node[below]{\tiny \x};
    }
    \foreach \y in {-4,-3,-2,-1,1,2,3,4}{%
      \draw (0.08,\y) -- (-0.08,\y) node[left]{\tiny \y};
    }

    \node[color=n2, right] at (0.85, 4) {$N=2$};
    \node[color=n3, left] at (1.2, 1.2) {$N=3$};
    \node[color=n7, right] at (1.9, -3) {$N=7$};
    \begin{scope}[thick]
      \clip (0,-4.2) rectangle (4.5,5.3);
      \input{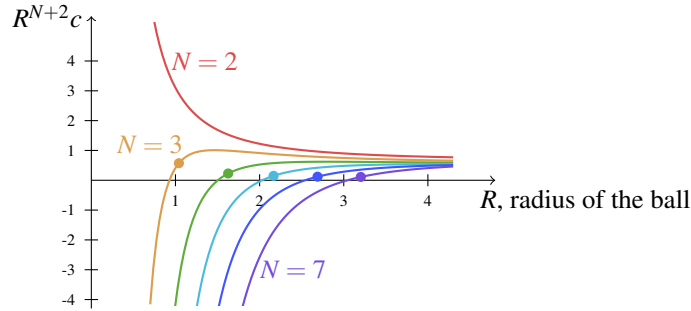}
    \end{scope}
  \end{tikzpicture}
  \vspace*{-2ex}
  \caption{Values of $c$ for $N = 2,\dotsc, 7$ as a function of $R$.
    The dot indicates the $R$ such that $2 + \lambda_2(B_R) = 2^*$.}
  \label{compC}
\end{figure}

Therefore, as we expected, the numerical experiments indicate that the
bifurcation at $p=2+\lambda_2$ yields non-radial solutions with energy
less than the energy of the trivial solution~$1$.

Of course, a bifurcation from the trivial solution $u=1$ is not the only mechanism that can justify the birth of a branch of ground state solutions. Remember that for values of $p$ close to $2$, there is uniqueness of the positive solution and it seems unlikely that a new branch starts from a degenerate ground state at some $p_{0}<2 + \lambda_2(B_R)$.  Whereas we cannot exclude that situation a priori, we give a numerical evidence that excludes this issue by implementing the mountain pass algorithm~\cite{mckenna,zhou1, zhou2}. For any tested values $R>0$ and $p<2+\lambda_2$, the algorithm finds the constant solution $1$. For our choices of  $R>0$ and $p> 2+\lambda_2$, the algorithm always finds a positive non-radial solution with energy less than the energy of $u=1$ (as it it should be from the results of O.~Lopes~\cite{Lopes}). We display the outcome of the mountain pass algorithm on Figure~\ref{GS:R=1} (resp.~\ref{GS:R=3}) and Table~\ref{tableV} for $N=2$, $p = \lceil 2 + \lambda_2\rceil$ and $R =1$ (resp.\ $R=3$). Observe also that the computed solutions look foliated Schwarz symmetric as they should be~\cite{vsc05,Weth}.

\newcommand{\meshline}[5]{
  \ifthenelse{\equal{#1}{0,0,0}}{%
    \draw[thick] (#2,#3) -- (#4,#5); 
  }{%
    \definecolor{mesh}{RGB}{#1}%
    \draw[color=mesh] (#2,#3) -- (#4,#5);
  }}

\begin{figure}[htb]
  \centering
  \begin{minipage}[b]{0.4\linewidth}
    \centering
    \smash{\raisebox{-12ex}{%
        \includegraphics[width=30ex]{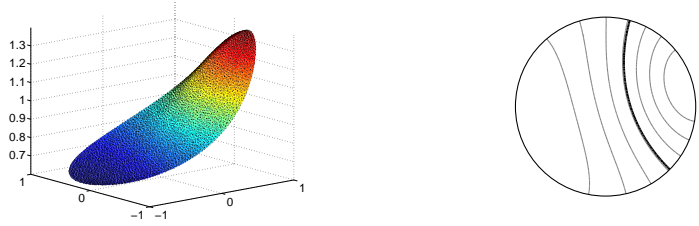}}}
  \end{minipage}
  \hfil
  \begin{minipage}[b]{0.4\linewidth}
    \centering
    \begin{tikzpicture}[x=7.5ex, y=7.5ex]
      \begin{scope}[rotate=35]
        \input{\graphpath neumannR1.tex}
      \end{scope}
      \draw (0,0) circle(1);
    \end{tikzpicture}
  \end{minipage}
  \caption{Non radially symmetric least energy solution on
    $B_1 \subseteq \IR^2$.}

  \label{GS:R=1}
\end{figure}

\begin{figure}[htb]
  \centering
  \begin{minipage}[b]{0.4\linewidth}
    \centering
    \smash{\raisebox{-12ex}{%
        \includegraphics[width=30ex]{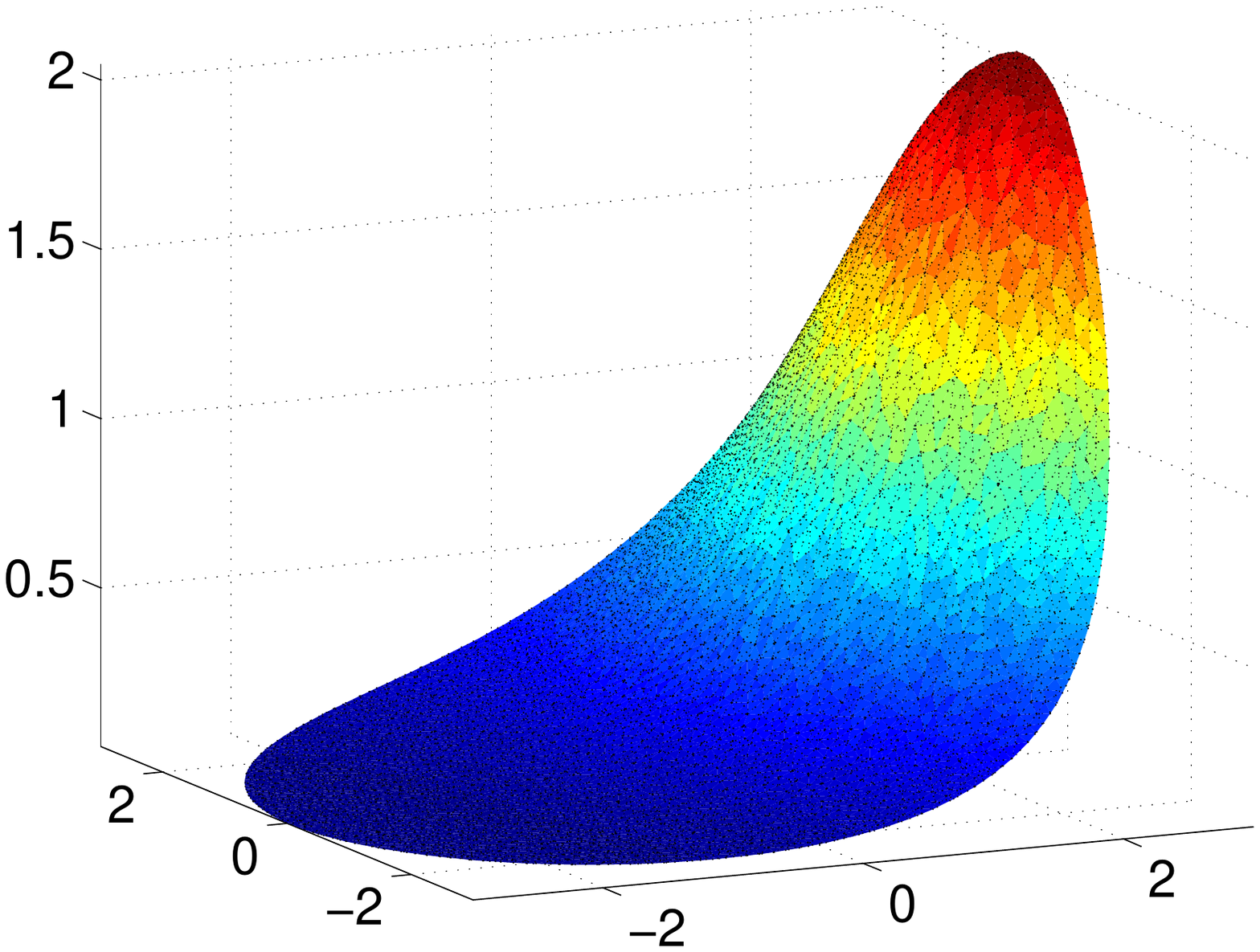}}}
  \end{minipage}
  \hfil
  \begin{minipage}[b]{0.4\linewidth}
    \centering
    \begin{tikzpicture}[x=2.5ex, y=2.5ex]
      \input{\graphpath neumannR3.tex}
      \draw (0,0) circle(3);
    \end{tikzpicture}
  \end{minipage}
  \caption{Non radially symmetric least energy solution on
    $B_3 \subseteq \IR^2$.}
  \label{GS:R=3}
\end{figure}

\begin{table}[!ht]
  \begin{center}
    \begin{tabular}{c r@{.}l c r@{.}l r@{.}l r@{.}l r@{.}l
        r@{.\vrule height 2.3ex width 0pt}l}
      $R$& \multicolumn{2}{c}{$2+\lambda_2$}& $p$ &\multicolumn{2}{c}{$\min u$}
      & \multicolumn{2}{c}{$\max u$} &
      \multicolumn{2}{c}{$\E(u)$} & \multicolumn{2}{c}{$\E(1)$} &
      \multicolumn{2}{c}{$\norm{\nabla \E(u)}$}\\  \hline
      1& 5&39&  6&  0&62&  1&36&  1&024&  1&047& 1&6 $\cdot 10^{-9}$\\
      \hline
      3& 2&38&  3&  0&03&  2&05&  2&800&  4&71&  1&6 $\cdot 10^{-12}$
    \end{tabular}
  \end{center}
  \caption{Characteristics of non-symmetric ground state
    $u$ on $B_R \subseteq \IR^2$.}
\label{tableV}
\end{table}

Owing to this foliated Schwarz symmetry, one can also numerically explore the behavior of ground states in $B_R \subseteq \IR^N$ for $N \ge 2$.  Indeed, we can assume that ground state solutions only depend on
\begin{equation*}
  (x_1,\rho) := \bigl(x_{1},\sqrt{x_2^2 +\cdots + x_N^2} \bigr) \in
  \bigl\{ (x_1, \rho)\in \IR\times\IR \bigm|
  x_1^2 + \rho^2 < R,\ \rho > 0 \bigr\}.
\end{equation*}
Let $(u_p)_{2 < p < 2^*}$ be a family of ground state solutions. Since $u=1$ is a competitor in the Nehari manifold $\N_{1,p}$, we have
\begin{equation*}
  \E_{1,p}(u_{p})
  = (\tfrac{1}{2} - \tfrac{1}{p}) \norm{u_{p}}_{H^1}^2
  \le \E_{1,p}(1)
  = (\tfrac{1}{2} - \tfrac{1}{p}) \norm{1}_{H^1}^2
\end{equation*}
and obviously this implies that ground state solutions are bounded for
all $p \in \intervaloo{2, 2^*}$. The graphs in
Figures~\ref{fig:norm-gs} and~\ref{fig:max-gs} support the fact
that the ground state solution form a continuum bifurcating from
$(2 + \lambda_2, 1)$.

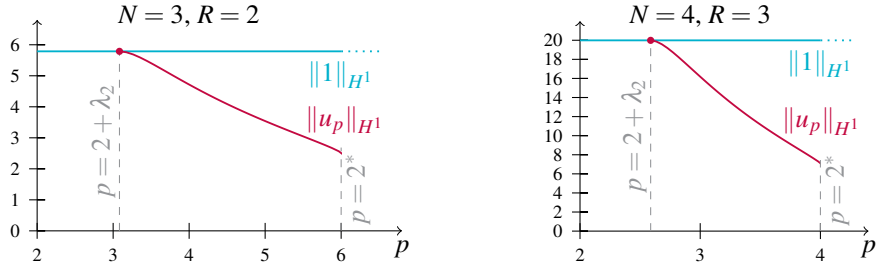
\begin{figure}[htb]
  \begin{minipage}[b]{0.48\linewidth}
    \centering
    \begin{tikzpicture}[y=2.6ex]
      \node at (4,18ex) {$N = 3$, $R = 2$};
      \draw[->] (2,0) -- (6.8,0) node[below]{$p$};
      \draw[->] (2,0) -- +(0, 6.8);
      \foreach \x in {2,...,6}{%
        \draw (\x, 3pt) -- +(0, -6pt) node[below]{\scriptsize $\x$};
      }
      \foreach \y in {0,...,6}{%
        \draw (2,\y) -- +(3pt,0) -- +(-3pt, 0) node[left]{\scriptsize $\y$};
      }
      \begin{scope}[color=umons-gray]
        \draw[dashed] (3.083240, 0) -- (3.083240, 5.788081)
        node[pos=0.5,xshift=-1.4ex,rotate=90]{$p = 2 + \lambda_2$};
        \draw[dashed] (6,0) -- (6., 2.685641)
        node[pos=0.5,xshift=1.4ex,rotate=90]{$p = 2^*$};
      \end{scope}
      \begin{scope}[color=umons-turquoise, line width=0.7pt]
        \draw (2, 5.788081) -- +(4,0) node[below]{$\norm{1}_{H^1}$};
        \draw[dotted] (6, 5.788081) -- +(0.5, 0);
      \end{scope}
      \begin{scope}[color=umons-red, line width=0.7pt]
        \node[above right] at (5.4,2.8) {$\norm{u_p}_{H^1}$};
        \clip (2,0) rectangle (6,6);
        \draw plot file{\graphpath schwarz-gs-N3-R2.dat};
        \draw[fill] (3.083240, 5.788081) circle(1pt);
      \end{scope}
    \end{tikzpicture}
  \end{minipage}
  \hfill
  \begin{minipage}[b]{0.48\linewidth}
    \centering
    \begin{tikzpicture}[x=10ex,y=0.8ex]
      \node at (3, 18ex) {$N = 4$, $R = 3$};
      \draw[->] (2,0) -- (4.4,0) node[below]{$p$};
      \draw[->] (2,0) -- +(0, 22);
      \foreach \x in {2,...,4}{%
        \draw (\x, 3pt) -- +(0, -6pt) node[below]{\scriptsize $\x$};
      }
      \foreach \y in {0,2,...,20}{%
        \draw (2,\y) -- +(3pt,0) -- +(-3pt, 0) node[left]{\scriptsize $\y$};
      }
      \begin{scope}[color=umons-turquoise, line width=0.7pt]
        \draw (2, 19.9896) -- +(2,0) node[below]{$\norm{1}_{H^1}$};
        \draw[dotted] (4, 19.9896) -- +(0.3, 0);
      \end{scope}
      \begin{scope}[color=umons-red, line width=0.7pt]
        \node[above right] at (3.6,9) {$\norm{u_p}_{H^1}$};
        \clip (2,0) rectangle (4,21);
        \draw plot file{\graphpath schwarz-gs-N4-R3.dat};
        \draw[fill] (2.58773, 19.9896) circle(1pt);
      \end{scope}
      \begin{scope}[color=umons-gray]
        \draw[dashed] (2.58773, 0) -- +(0, 19.9896)
        node[pos=0.5,xshift=-1.4ex,rotate=90]{$p = 2 + \lambda_2$};
        \draw[dashed] (4,0) -- (4, 7.449446)
        node[pos=0.58,xshift=1.4ex,rotate=90]{$p = 2^*$};
      \end{scope}
    \end{tikzpicture}
  \end{minipage}

  \vspace{-1ex}%
  \caption{Norm of the ground state $u_p$
    of \eqref{pblP1}
    on $B_R \subseteq \IR^N$.}
  \label{fig:norm-gs}
\end{figure}

\begin{figure}[htb]
  \begin{minipage}[b]{0.48\linewidth}
    \centering
    \begin{tikzpicture}[y=1.7ex]
      \node at (4, 25ex) {$N = 3$, $R = 2$};
      \draw[->] (2,0) -- (6.8,0) node[below]{$p$};
      \draw[->] (2,0) -- +(0, 17);
      \foreach \x in {2,...,6}{%
        \draw (\x, 3pt) -- +(0, -6pt) node[below]{\scriptsize $\x$};
      }
      \foreach \u in {0,2,...,10} {%
        \draw (2,\u) ++(3pt, 0) -- +(-6pt, 0) node[left]{\scriptsize $\u$};
      }
      \foreach \u/\y in {12/11.64506, 15/13.32699, 20/15.1885,
      }{%
        \draw (2,\y) ++(3pt, 0) -- +(-6pt, 0) node[left]{\scriptsize $\u$};
      }
      \draw[color=umons-turquoise, line width=0.7pt]
      (2,1) -- (6, 1);
      \begin{scope}[color=umons-red, line width=0.7pt]
        \node[above right] at (3.8, 3.1) {$\max u_p$};
        \draw[fill] (3.083240, 1) circle(1pt);
        \clip (2,0) rectangle (6, 17);
        \draw plot file{\graphpath max-gs-N3-R2.dat}; 
      \end{scope}
      \begin{scope}[color=umons-gray]
        \draw[dashed] (3.083240, 0) -- +(0, 10)
        node[pos=0.8,xshift=1.4ex,rotate=90]{$2 + \lambda_2$};
        \draw[dashed] (6,0) -- (6., 16)
        node[pos=0.4,xshift=1.4ex, rotate=90]{$p = 2^*$};
      \end{scope}
    \end{tikzpicture}
  \end{minipage}
  \hfill
  \begin{minipage}[b]{0.48\linewidth}
    \centering
    \begin{tikzpicture}[x=10ex,y=1.3ex]
      \node at (3, 25ex) {$N = 4$, $R = 3$};
      \draw[->] (2,0) -- (4.4,0) node[below]{$p$};
      \draw[->] (2,0) -- +(0, 22);
      \foreach \x in {2,...,4}{%
        \draw (\x, 3pt) -- +(0, -6pt) node[below]{\scriptsize $\x$};
      }
      \foreach \u in {0,2,...,10} {%
        \draw (2,\u) ++(3pt, 0) -- +(-6pt, 0) node[left]{\scriptsize $\u$};
      }
      \foreach \u/\y in {12/11.64506, 15/13.32699, 20/15.1885,
        30/17.48589, 50/20.09040,
      }{%
        \draw (2,\y) ++(3pt, 0) -- +(-6pt, 0) node[left]{\scriptsize $\u$};
      }
      \draw[color=umons-turquoise, line width=0.7pt] (2,1) -- (4, 1);
      \begin{scope}[color=umons-red, line width=0.7pt]
        \node[below right] at (3, 4.3) {$\max u_p$};
        \draw[fill] (2.58773, 1) circle(1pt);
        \clip (2,0) rectangle (4,22);
        \draw plot file{\graphpath max-gs-N4-R3.dat}; 
      \end{scope}
      \begin{scope}[color=umons-gray]
        \draw[dashed] (2.58773, 0) -- +(0, 10)
        node[pos=0.8,xshift=1.4ex,rotate=90]{$2 + \lambda_2$};
        \draw[dashed] (4,0) -- +(0, 21)
        node[pos=0.4,xshift=1.4ex, rotate=90]{$p = 2^*$};
      \end{scope}
    \end{tikzpicture}
  \end{minipage}

  \vspace{-1ex}%
  \caption{Maximum value of the ground state $u_p$ of \eqref{pblP1}
    on $B_R \subseteq \IR^N$.}
  \label{fig:max-gs}
\end{figure}
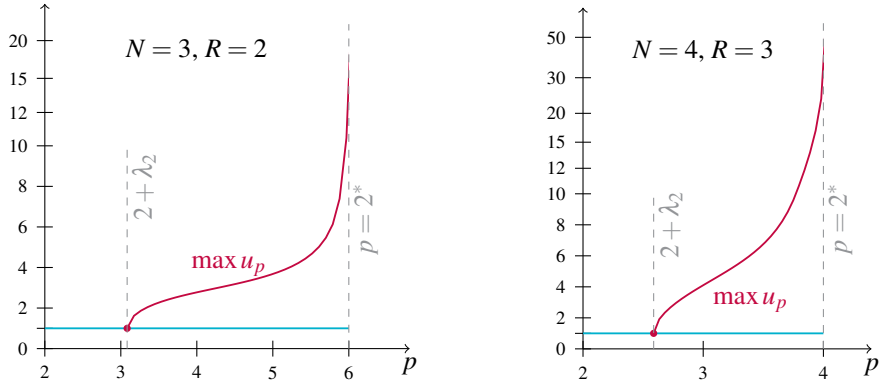

The graphs of some ground state solutions are given by Figure~\ref{fig:ground-states}.  As expected, when $p \approx 2 + \lambda_2$, $u_p$ looks like $1 + \epsilon \phi_2$ where
\begin{equation*}
  \phi{}_2(x) = \abs{x}^{-\frac{N-2}{2}}
  J_{N/2}\bigl(\sqrt{\lambda_2} \abs{x}\bigr) \, x_1
\end{equation*}
is a second eigenfunction of $-\Delta$ that is invariant under rotations in $(x_2, \dotsc, x_N)$ centered at~$0$.
As $p$ approaches $2^*$, the solution becomes mostly flat except for a (bounded) bump  on the $x_1$-axis.  We emphasize that for $p = 2^*$, a least energy solution still exists in $H^1$ as established by Wang~\cite{wang91}.

\begin{figure}[htb]
  \centering
  \centerline{%
    \begin{tikzpicture}
      \node at (0,0) {%
        \includegraphics[width=0.24\linewidth,viewport=69 213 554 580]{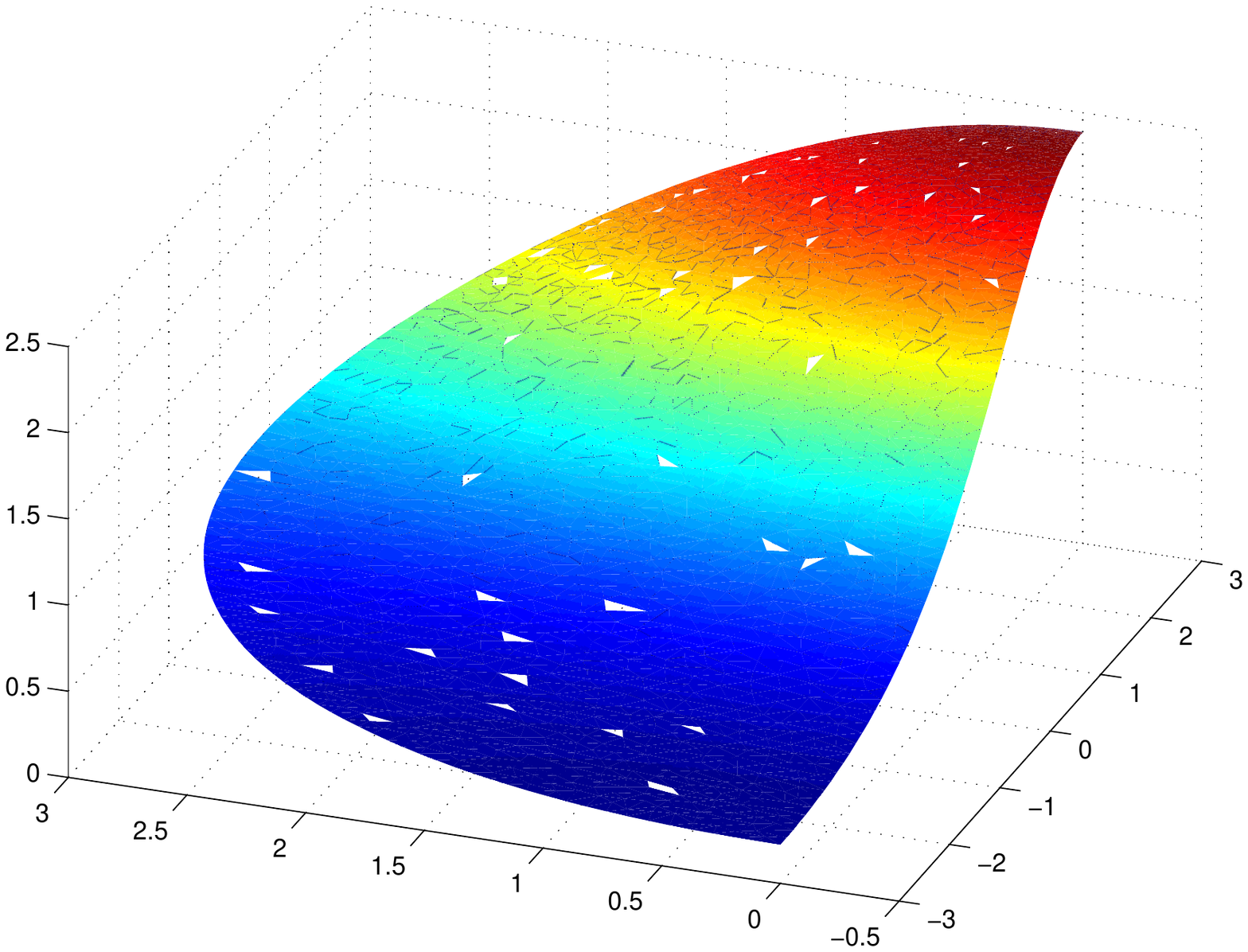}};
      \node at (0, 10ex) {$p = 2.1 + \lambda_2$};
      \node at (-2ex, -8ex) {\scriptsize $\rho$};
      \node at (8ex, -5.5ex) {\scriptsize $x_1$};
    \end{tikzpicture}
    \hfill
    \begin{tikzpicture}
      \node at (0,0) {%
        \includegraphics[width=0.24\linewidth,viewport=68 210 556 580]{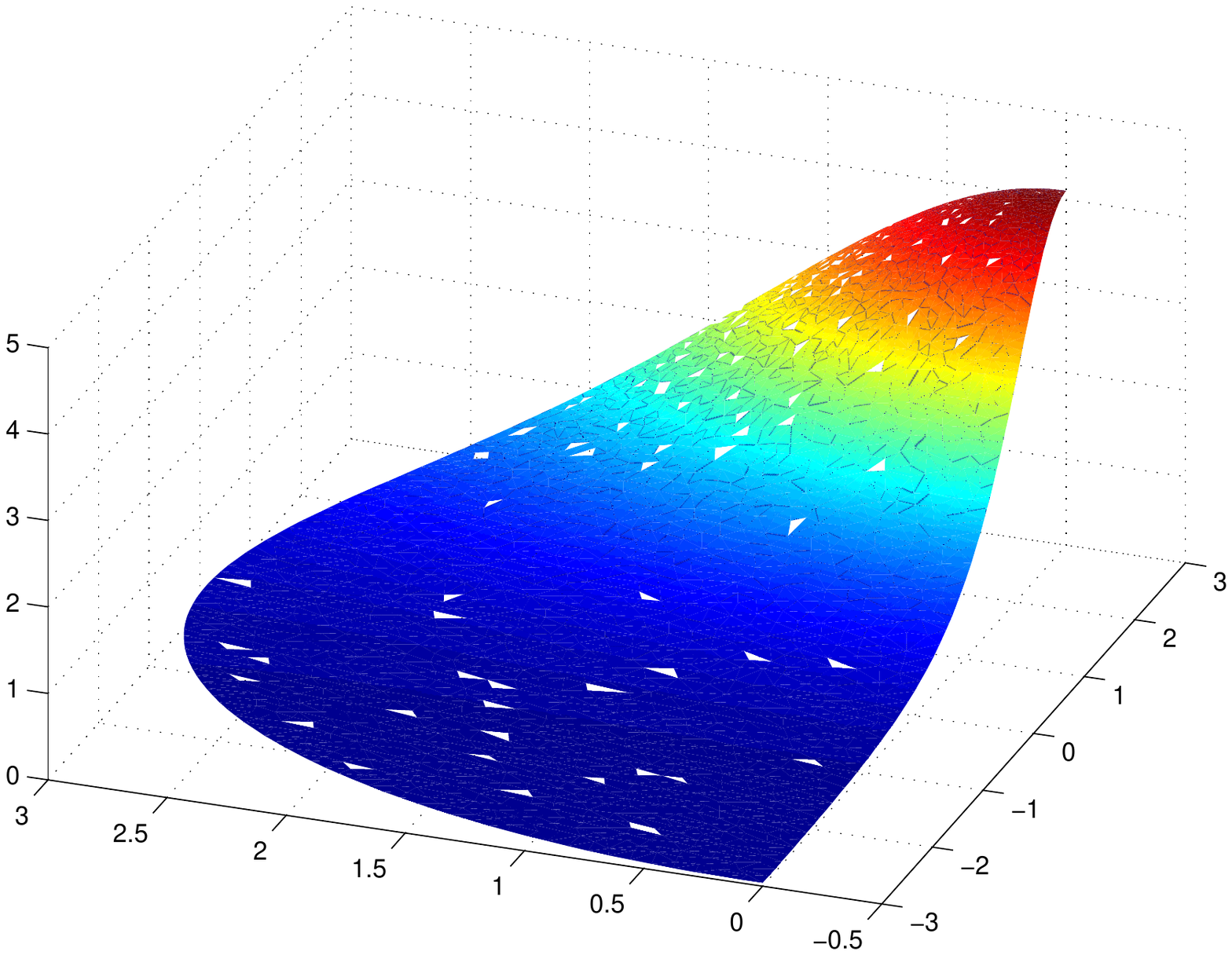}};
      \node at (0, 10ex) {$p = 3$};
      \node at (-2ex, -8ex) {\scriptsize $\rho$};
      \node at (8ex, -5.5ex) {\scriptsize $x_1$};
    \end{tikzpicture}
    \hfill
    \begin{tikzpicture}
      \node at (0,0) {%
        \includegraphics[width=0.24\linewidth,viewport=75 208 556 580]{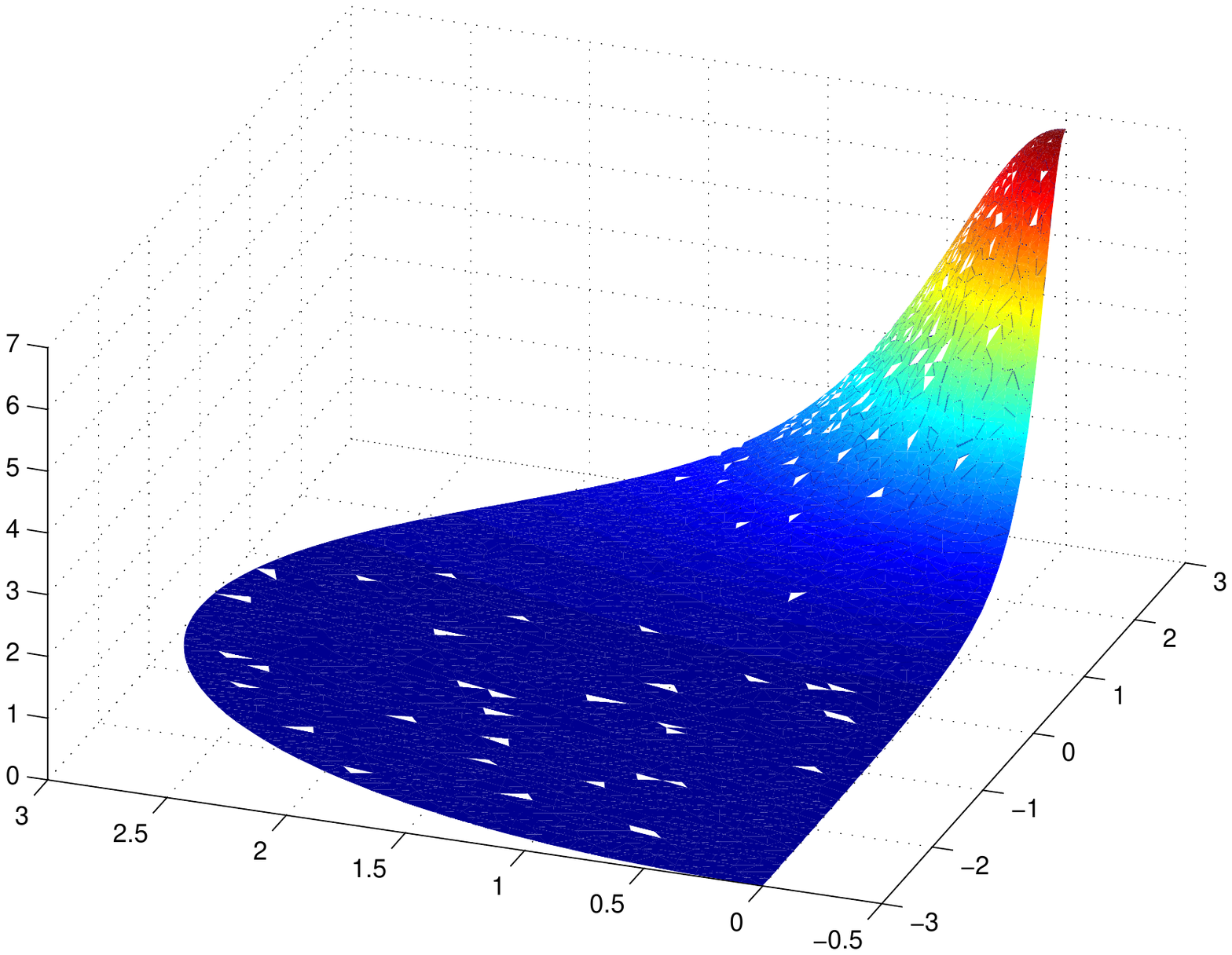}};
      \node at (0, 10ex) {$p = 3.5$};
      \node at (-2ex, -8ex) {\scriptsize $\rho$};
      \node at (8ex, -5.5ex) {\scriptsize $x_1$};
    \end{tikzpicture}
    \hfill
    \begin{tikzpicture}
      \node at (0,0) {%
        \includegraphics[width=0.24\linewidth,viewport=71 212 555 580]{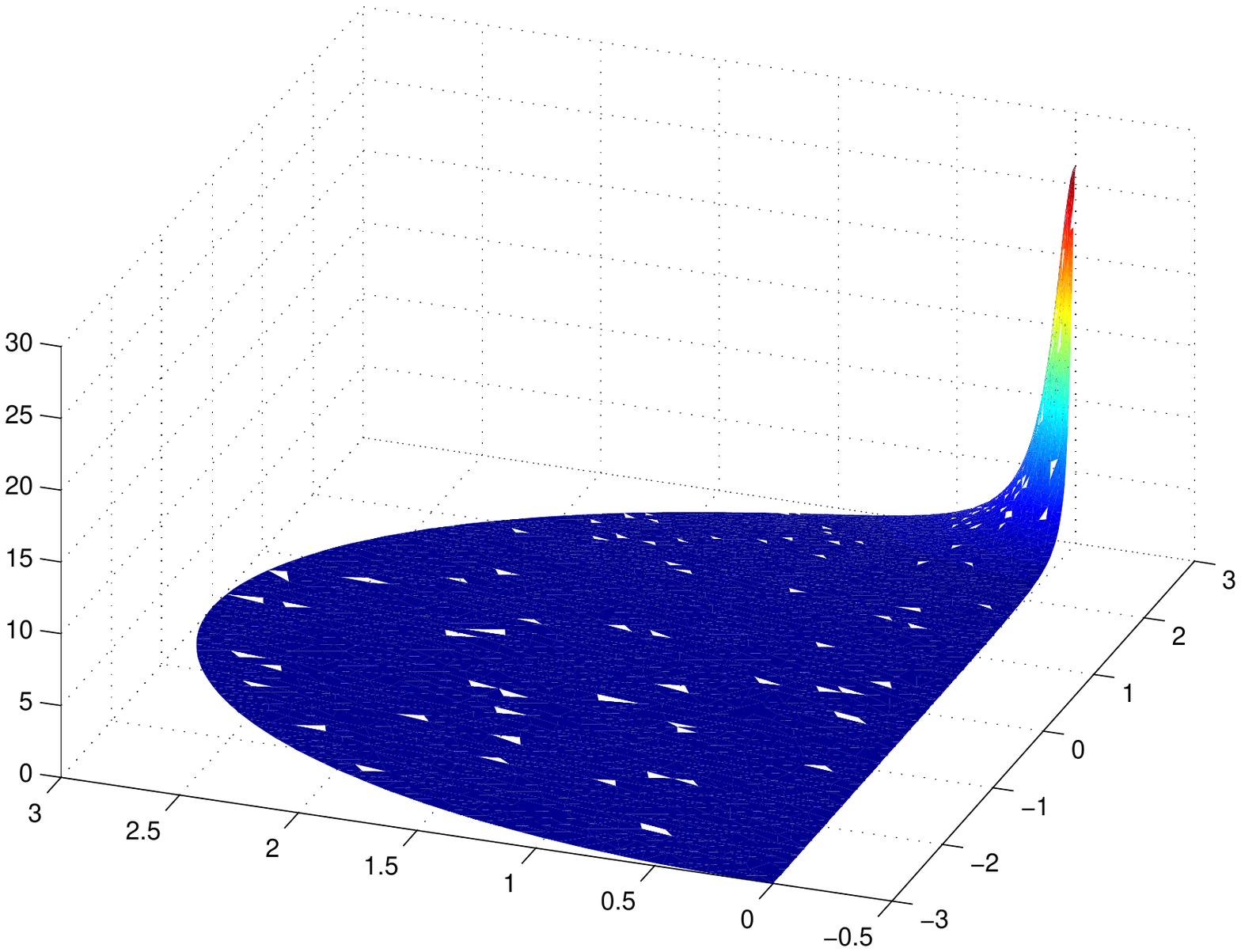}};
      \node at (0, 10ex) {$p = 4 = 2^*$};
      \node at (-2ex, -8ex) {\scriptsize $\rho$};
      \node at (8ex, -5.5ex) {\scriptsize $x_1$};
    \end{tikzpicture}%
  }

  \vspace{-1ex}%
  \caption{Graphs of ground states on $B_3 \subseteq \IR^4$.}
  \label{fig:ground-states}
\end{figure}
These numerical investigations motivate the following conjecture.
\begin{Conj}\label{conj:ground-states}
Let $N \ge 2$, $p \in \intervaloo{2, 2^*}$ and $\Omega = B_R$.
  \begin{enumerate}[(a)]
  \item\label{conj:1} The constant $u=1$ is the least energy solution
    to~\eqref{pblP1}
    if and only if $p \le 2+\lambda_2$. In this case,
    there are no other positive solutions.
  \item\label{conj:gs} For $p> 2+\lambda_2$, least energy solutions
    belong to the branch bifurcating from $p = 2+\lambda_2$.
  \end{enumerate}
\end{Conj}
Assertion (\ref{conj:1}) is supported by the computation of mountain
pass solutions. Assertion (\ref{conj:gs}) is motivated by the
numerical evidence that the bifurcation at $(2+\lambda_2,1)$ is
super-critical and the investigation along the branch of least energy
solutions.

Note that Figures~\ref{fig:norm-gs} and~\ref{fig:max-gs} suggest
that the branch of ground state solutions exists for all
$p \in \intervaloo{2, +\infty}$.  Proving that this is indeed the case
would be interesting, see
Conjecture~\ref{conj:nonradial-supercritical} in the Introduction, and
will be the subject of a future project.

\subsection{The first radial bifurcation}
\label{num:1st radial bifurc}%
We display in this subsection some numerical computations illustrating the first bifurcation in the space of radially symmetric functions. One naturally expects that on this first branch, the solutions are least energy \emph{radial} solutions, namely least energy solutions among radial functions. This bifurcation arises at $(1,2+
\lambdarad_2)$ where $\lambdarad_2$ is the second radial
eigenvalue. The numerics are performed on a ball of radius $R=4$ so
that $2 + \lambdarad_2 < 2^*$ for $N \in\{ 2, 3, 4\}$.
We recall that this bifurcation is
transcritical, as follows from Theorem~\ref{transcritical-radial-bifurc}. Using
the Mountain Pass Algorithm to approximate the least energy radial
solution, one gets (as expected) a
decreasing solution to~\eqref{pblP1} different
from~$1$ for $p \lesssim 2 + \lambdarad_2$, as stated by Theorem~\ref{transcritical-radial-bifurc}.  This solution is drawn on
the left of Figure~\ref{fig:rad-below} and its characteristics are
given in Table~\ref{table:rad-below}.
Using a shooting method, a second positive decreasing solution $u_2$ is found. It is pictured on the right of Figure~\ref{fig:rad-below} and some characteristics are given in Table~\ref{table:rad-below}. It has higher energy that both the previous decreasing solution and~$u=1$.

\begin{figure}[!hbt]
  \centering
  \newcommand{\R}{4}

  \null\hfill
  \tikzsetnextfilename{bifurcation-radial-2-below}%
  \begin{tikzpicture}[x=6ex,y=2.8ex]
    \newcommand{\ymin}{0}%
    \newcommand{\ymax}{5.9}%
    \draw[->] (0,\ymin) -- ++(\R + 0.5, 0)
    node[below right, xshift=-2ex]{$r = \abs{x}$};
    \draw[->] (0,\ymin) -- (0, \ymax) node[left]{$u_1$};
    \foreach \x in {0,...,\R}{%
      \draw (\x,\ymin) ++(0,3pt) -- ++(0,-6pt)
      node[below]{\scriptsize $\x$};
    }
    \foreach \y in {0,...,5}{%
      \draw (0, \y) ++(3pt,0) -- ++(-6pt,0) node[left]{\scriptsize $\y$};
    }
    \draw[dashed] (0,1) -- (\R,1);
    \begin{scope}
      \clip (0,0) rectangle (\R, \ymax);
      \draw[thick, color=n4] plot file{\graphpath radial-N4-R4-2-below.dat};
      \node[color=n4, right] at (0.4, 5) {$N=4$};
    \end{scope}
    \draw[thick, color=n3] plot file{\graphpath radial-N3-R4-2-below.dat};
    \node[color=n3, right] at (0.72, 3) {$N=3$};
    \draw[thick, color=n2] plot file{\graphpath radial-N2-R4-2-below.dat};
    \node[color=n2, above right] at (1.7, 1.05) {$N=2$};
  \end{tikzpicture}
  \hfill
  \begin{tikzpicture}[x=6ex,y=42ex]
    \newcommand{\ymin}{0.9}%
    \newcommand{\ymax}{1.3}%
    \draw[->] (0,\ymin) -- ++(\R + 0.5, 0)
    node[below right, xshift=-2ex]{$r = \abs{x}$};
    \draw[->] (0,\ymin) -- (0, \ymax) node[left]{$u_2$};
    \foreach \x in {0,...,\R}{%
      \draw (\x,\ymin) ++(0,3pt) -- ++(0,-6pt)
      node[below]{\scriptsize $\x$};
    }
    \foreach \y in {0.9, 1, 1.1, 1.2}{%
      \draw (0, \y) ++(3pt,0) -- ++(-6pt,0) node[left]{\scriptsize $\y$};
    }
    \draw[dashed] (0,1) -- (\R,1);
    \begin{scope}
      \clip (0,\ymin) rectangle (\R, \ymax);
      \draw[thick, color=n4] plot file{\graphpath radial-N4-R4-u0-1.00003.dat};
      \node[color=n4, below right] at (0.1, 1) {$N=4$};
    \end{scope}
    \draw[thick, color=n3] plot file{\graphpath radial-N3-R4-u0-1.0953.dat};
    \node[color=n3, above right, yshift=-1pt, rotate=-10] at (0, 1.1) {$N=3$};
    \draw[thick, color=n2] plot file{\graphpath radial-N2-R4-u0-1.20215.dat};
    \node[color=n2, above right, yshift=-1pt] at (0.1, 1.2) {$N=2$};
  \end{tikzpicture}
  \hfill\null

  \vspace{-2ex}
  \caption{Profile of non constant radial solutions
    ($p = 1.95 + \lambdarad_2$).}
  \label{fig:rad-below}
\end{figure}
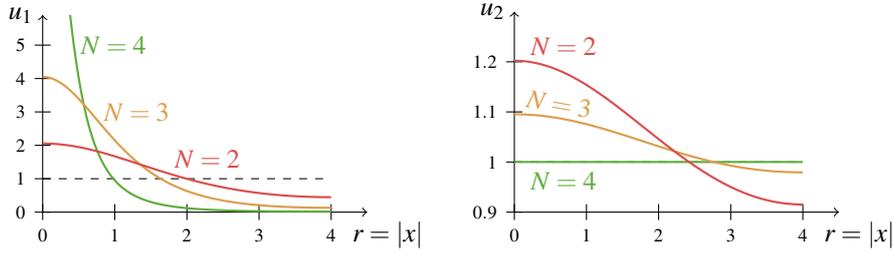

\begin{table}[ht]
  \begin{equation*}
    \begin{array}{c c r@{.}l|r@{.}l r@{.}l r@{.}l|r@{.}l r@{.}l r@{.}l}
      N& {2 + \lambdarad_2}& \multicolumn{2}{c|}{\E(1)}&
      \multicolumn{2}{c}{\min u_1}& \multicolumn{2}{c}{\max u_1}&
      \multicolumn{2}{c|}{\E(u_1)}&
      \multicolumn{2}{c}{\min u_2}& \multicolumn{2}{c}{\max u_2}&
      \multicolumn{2}{c}{\E(u_2)}\\
      \hline
      2&  2.92& 7&604&   0&447& 2&05& 7&45  & 0&915&  1&202& 7&606\\
      3&  3.26& 50&576&  0&130& 4&05& 34&85 & 0&979&  1&095& 50&578\\
      4&  3.65& 280&581& 0&016& 13&3& 66&39 & 0&999& 1&00003& 280&581\\
      \hline
    \end{array}
  \end{equation*}
  \caption{Characteristics of radial solutions
    ($p = 1.95 + \lambdarad_2$, $R=4$).}
  \label{table:rad-below}
\end{table}

For $p \in \intervaloo{2 + \lambdarad_2, 2^*}$,
the Mountain Pass Algorithm finds two radial solutions $u_1$
and $u_2$ to problem \eqref{pblP1} depending on the starting point.
As an illustration, for $p = 2.1 + \lambdarad_2$, these solutions
are depicted in
Figure~\ref{fig:radpart} and their characteristics are given in
Table~\ref{tabradpart}. The accuracy is relatively good since
$\norm{\nabla\E(u_i)} < 10^{-8}$ for $i=1, 2$.
In agreement with the results of Section~\ref{bifu},
they are positive and possess a single intersection
with $1$.  Moreover,  one solution is increasing along the
radius and the other one is decreasing.

\begin{figure}[!hbt]
  \newcommand{\R}{4}
  \null\hfill
  \tikzsetnextfilename{bifurcation-radial-1}%
  \begin{tikzpicture}[x=6ex,y=40ex]
    \newcommand{\ymin}{0.75}%
    \draw[->] (0,\ymin) -- ++(\R + 0.5, 0)
    node[below right, xshift=-2ex]{$r = \abs{x}$};
    \draw[->] (0,\ymin) -- (0, 1.17) node[left]{$u_1$};
    \foreach \x in {0,...,\R}{%
      \draw (\x,\ymin) ++(0,3pt) -- ++(0,-6pt)
      node[below]{\scriptsize $\x$};
    }
    \foreach \y in {0.8, 0.9, 1, 1.1}{%
      \draw (0, \y) ++(3pt,0) -- ++(-6pt,0) node[left]{\scriptsize $\y$};
    }
    \draw[dashed] (0,1) -- (\R,1);
    \draw[thick, color=n4] plot file{\graphpath radial-N4-R4-1.dat};
    \node[color=n4, below right] at (2.8, 1.) {$N=4$};
    \draw[thick, color=n3] plot file{\graphpath radial-N3-R4-1.dat};
    \draw[thick, color=n2] plot file{\graphpath radial-N2-R4-1.dat};
    \node[color=n2, left] at (3.2, 1.08) {$N=2$};
  \end{tikzpicture}
  \hfill
  \tikzsetnextfilename{bifurcation-radial-2}%
  \begin{tikzpicture}[x=6ex,y=2.8ex]
    \newcommand{\ymin}{0}%
    \newcommand{\ymax}{5.9}%
    \draw[->] (0,\ymin) -- ++(\R + 0.5, 0)
    node[below right, xshift=-2ex]{$r = \abs{x}$};
    \draw[->] (0,\ymin) -- (0, \ymax) node[left]{$u_2$};
    \foreach \x in {0,...,\R}{%
      \draw (\x,\ymin) ++(0,3pt) -- ++(0,-6pt)
      node[below]{\scriptsize $\x$};
    }
    \foreach \y in {1,...,5}{%
      \draw (0, \y) ++(3pt,0) -- ++(-6pt,0) node[left]{\scriptsize $\y$};
    }
    \draw[dashed] (0,1) -- (\R,1);
    \begin{scope}
      \clip (0,0) rectangle (\R, \ymax);
      \draw[thick, color=n4] plot file{\graphpath radial-N4-R4-2.dat};
      \node[color=n4, right] at (0.3, 5) {$N=4$};
    \end{scope}
    \draw[thick, color=n3] plot file{\graphpath radial-N3-R4-2.dat};
    \node[color=n3, right] at (0.65, 3) {$N=3$};
    \draw[thick, color=n2] plot file{\graphpath radial-N2-R4-2.dat};
    \node[color=n2, above right] at (1.5, 1.05) {$N=2$};
  \end{tikzpicture}
  \hfill\null

  \vspace{-2ex}
  \caption{Profile of non constant radial solutions
    ($p = 2.1 + \lambdarad_2$).}
  \label{fig:radpart}
\end{figure}
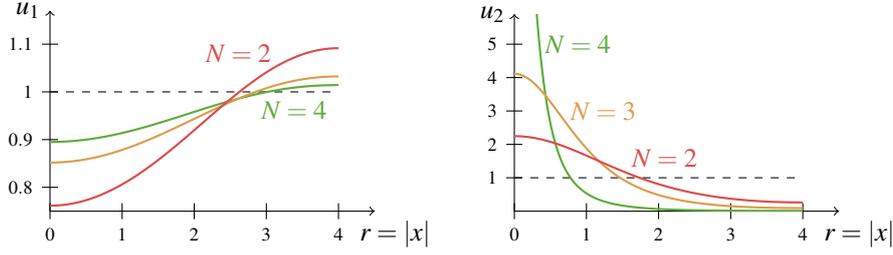

\begin{table}[ht]
  \begin{equation*}
    \begin{array}{c c r@{.}l|r@{.}l r@{.}l r@{.}l|r@{.}l r@{.}l r@{.}l}
      N& {2 + \lambdarad_2}& \multicolumn{2}{c|}{\E(1)}&
      \multicolumn{2}{c}{\min u_1}& \multicolumn{2}{c}{\max u_1}&
      \multicolumn{2}{c|}{\E(u_1)}&
      \multicolumn{2}{c}{\min u_2}& \multicolumn{2}{c}{\max u_2}&
      \multicolumn{2}{c}{\E(u_2)}\\
      \hline
      2&  2.92& 8&48&  0&76& 1&09& 8&47&  0&261& 2&25& 7&39\\
      3&  3.26& 54&30&  0&85& 1&03& 54&29&  0&092& 4&12& 30&74\\
      4&  3.65& 294&63&  0&90& 1&01& 294&62&  0&008& 17&25& 49&61\\
      \hline
    \end{array}
  \end{equation*}
  \caption{Characteristics of radial solutions
    ($p = 2.1 + \lambdarad_2$, $R=4$).}
  \label{tabradpart}
\end{table}

The bifurcation diagram in Figure~\ref{fig:bifurcation} explains how
the above solutions are related: they all belong to the continuum
bifurcating from $2 + \lambdarad_2$. The increasing solutions on the
left of Figure~\ref{fig:radpart} belong
to the branch starting to the right of $2 + \lambdarad_2$. They have lower energy than~$1$ but not the lowest energy.
Their radial Morse index, denoted $\MIrad$, is~$1$ (they are
local minimizers of $\E_{1,p}$ on the Nehari manifold in
$H^1_{\text{rad}}$).  These solutions still exist in the supercritical
range.  The decreasing solutions on Figure~\ref{fig:rad-below}
and on the left of Figure~\ref{fig:radpart}
all belong to the branch emanating to the left of $2 + \lambdarad_2$.
Before the turning point, they have a radial Morse index of~$2$
and have higher energy than~$1$ (see Figure~\ref{fig:bifurcation-energy}).
After the turning point, their radial Morse index is~$1$ and,
as displayed in Figure~\ref{fig:bifurcation-energy}, they become radial ground states for slightly greater~$p$.

\begin{figure}[htb]
  \centering
  \begin{tikzpicture}[x=15ex,y=1.1ex]
    \draw[->] (2,0) -- (5.5,0) node[below]{$p$};
    \draw[->] (2,0) -- +(0, 23) node[left]{$u(0)$};
    \foreach \p in {2,...,5} {%
      \draw (\p, 3pt) -- (\p, -3pt) node[below]{\scriptsize $\p$};
    }
    \foreach \u in {0,2,...,20} {%
      \draw (2,\u) ++(3pt, 0) -- +(-6pt, 0) node[left]{\scriptsize $\u$};
    }
    \begin{scope}[color=umons-red]
      \draw[<-] (3.8, 17.5) -- (4.2, 19)
      node[right, fill=umons-red!10]{\color{black}%
        \footnotesize
        \begin{tabular}{@{}l@{}}
          $u(0) > 1$\\
          decreasing solutions\\
          $\MIrad(u) = 1$
        \end{tabular}};
      \draw[<-] (2.9, 4) .. controls (2.6, 3) and (2.6,6) .. (2.7, 12)
      node[above, fill=umons-red!10]{\color{black}%
        \footnotesize
        \begin{tabular}{@{}l@{}}
          $u(0) > 1$\\
          decreasing solutions\\
          $\E_{1,p}(u) > \E_{1,p}(1)$\\
          $\MIrad(u) = 2$
        \end{tabular}};

      \clip (2,0) rectangle (4, 23);
      \draw[thick] plot file{\graphpath radial-branch-decr.dat};
      \draw[fill] (2.910127, 6.01076) circle(1.5pt);
    \end{scope}
    \begin{scope}[color=umons-turquoise]
      \draw[thick, smooth] plot file{\graphpath radial-branch-incr.dat};
      \draw[<-] (4.5, 0.7) -- (5, 4)
      node[above, fill=umons-turquoise!10]{\color{black}%
        \footnotesize
        \begin{tabular}{@{}l@{}}
          $u(0) < 1$\\
          increasing solutions\\
          $\E_{1,p}(u) < \E_{1,p}(1)$\\
          $\MIrad(u) = 1$
        \end{tabular}};
    \end{scope}
    \draw[dashed] (2,1) -- (5.3, 1);
    \draw[color=black,fill=white] (3.64841, 1) circle(1.5pt);
    \draw (3.64841, 3pt) -- +(0, -6pt);
    \draw [<-] (3.64841, -6pt) -- ++(0, -8pt)
    node[below]{\scriptsize $2 + \lambdarad_2 = 2 + \lambda_5$};
    \draw[dashed, color=umons-gray] (4,0) -- +(0,22.5)
    node[pos=0.5, xshift=1.4ex, rotate=90] {$p = 2^*$};
  \end{tikzpicture}

  \vspace*{-1ex}
  \caption{Radial bifurcation branch from $2 + \lambdarad_2$
    for $N=4$ and $R=4$.}
  \label{fig:bifurcation}
\end{figure}
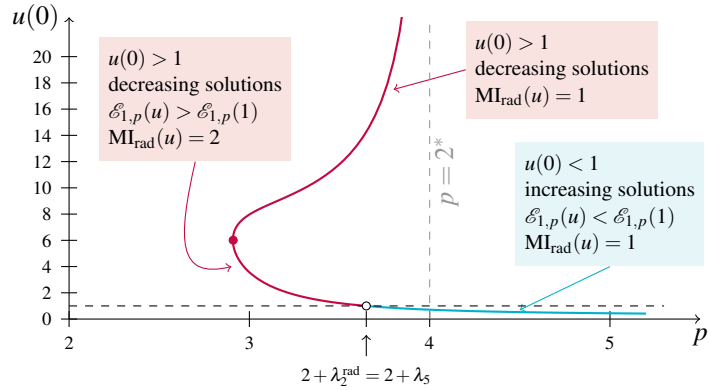

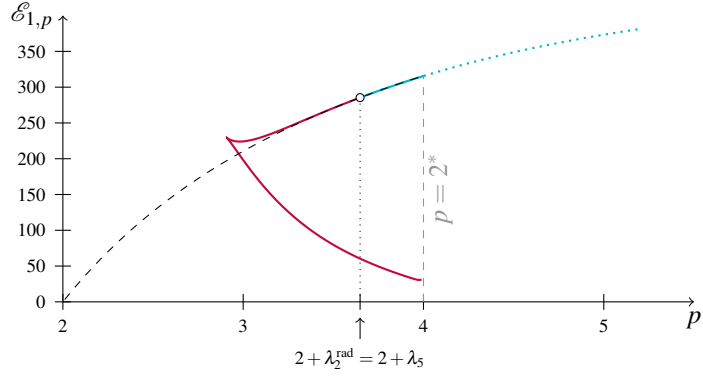
\begin{figure}[htb]
  \centering
  \begin{tikzpicture}[x=15ex,y=0.06ex]
    \draw[->] (2,0) -- (5.5,0) node[below]{$p$};
    \draw[->] (2,0) -- +(0, 400) node[left]{$\E_{1,p}$};
    \foreach \p in {2,...,5} {%
      \draw (\p, 3pt) -- (\p, -3pt) node[below]{\scriptsize $\p$};
    }
    \foreach \u in {0,50,...,350} {%
      \draw (2,\u) ++(3pt, 0) -- +(-6pt, 0) node[left]{\scriptsize $\u$};
    }
    \begin{scope}
      \clip (2,0) rectangle (4, 380);
      \draw[color=umons-red, thick]
      plot file{\graphpath radial-branch-decr-energy.dat};
    \end{scope}
    \begin{scope}[color=umons-turquoise]
      \draw[dotted, thick, smooth]
      plot file{\graphpath radial-branch-incr-energy.dat};
      \clip (2,0) rectangle (4,350);
      \draw[thick, smooth]
      plot file{\graphpath radial-branch-incr-energy.dat};
    \end{scope}
    \begin{scope}
      \clip (2,0) rectangle (4, 350);
      \draw[dashed] plot file{\graphpath radial-branch-1-energy.dat};
    \end{scope}
    \draw[dotted] (3.64841, 0) -- +(0, 285.392024) node(B){};
    \draw[color=black, fill=white] (B) circle(1.5pt);
    \draw (3.64841, 3pt) -- +(0, -6pt);
    \draw [<-] (3.64841, -6pt) -- ++(0, -8pt)
    node[below]{\scriptsize $2 + \lambdarad_2 = 2 + \lambda_5$};
    \draw[dashed, color=umons-gray] (4,0) -- +(0,315.827341)
    node[pos=0.5, xshift=1.4ex, rotate=90] {$p = 2^*$};

  \end{tikzpicture}
  \caption{Energy along the radial branches emanating from $2 +
    \lambdarad_2$ for $N=4$ and $R=4$.}
  \label{fig:bifurcation-energy}
\end{figure}

The above figures suggest that positive increasing
solutions are unique.  Increasing solutions must clearly start with $u(0) \in
\intervaloo{0,1}$.
As an additional evidence supporting uniqueness, we have drawn
on Figure~\ref{fig:time-map} the time maps
$$u_0 \mapsto T_{N,p}(u_0),$$
where $T_{N,p}(u_0)$ is the smaller positive number such that
$u'\bigl( T_{N,p}(u_0) \bigr) = 0$ with $u$ being the solution
to~\eqref{eq:radial-ode} such that $u(0) = u_0$.
These graphs clearly show that $T_{N,p} : \intervaloo{0,1} \to \IR$
is decreasing and so the equation $T_{n,p}(u_0) = R$ has at most a
solution.  Therefore uniqueness holds.

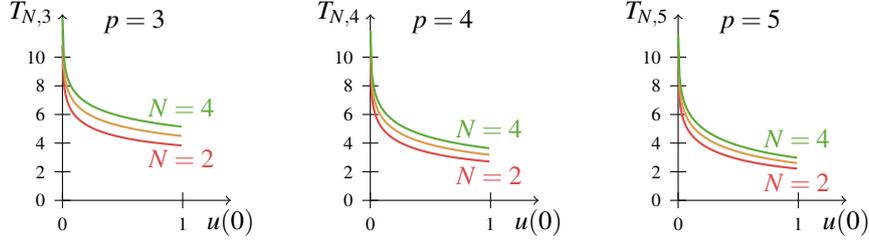
\begin{figure}[htb]
  \centering

  \newcommand{\graph}[1]{%
    \begin{tikzpicture}[x=10ex, y=1.2ex]
      \draw[->] (0,0) -- (1.4, 0) node[below]{$u(0)$};
      \draw[->] (0,0) -- (0, 13) node[left]{$T_{N,#1}$};
      \foreach \x in {0, 1} {%
        \draw (\x, 3pt) -- (\x, -3pt) node[below]{\scriptsize $\x$};
      }
      \foreach \y in {0,2,...,10} {%
        \draw (3pt, \y) -- (-3pt, \y) node[left]{\scriptsize $\y$};
      }
      \node[above] at (0.6, 11) {$p = #1$};
      \draw[thick, color=n2] plot file{\graphpath time-map-N2-p#1.dat}
      node[below, yshift=2pt]{$N = 2$};
      \draw[thick, color=n3] plot file{\graphpath time-map-N3-p#1.dat};
      \draw[thick, color=n4] plot file{\graphpath time-map-N4-p#1.dat}
      node[above]{$N = 4$};
    \end{tikzpicture}
    \ignorespaces}

  \null\hfill
  \graph{3}%
  \hfill
  \graph{4}%
  \hfill
  \graph{5}%
  \hfill\null

  \caption{Time maps for increasing solutions
    for $N \in \{2, 3, 4\}$.}
  \label{fig:time-map}
\end{figure}

\medskip

The preceding considerations naturally lead to some additional
conjectures.

\begin{Conj}\label{Conj:abcd}
  Let $N\ge 2$ and $\Omega = B_R$.
  \begin{enumerate}[(a)]
  \item For any $p > 2 + \lambdarad_2 / \lambda$,
  there exists a \emph{unique} positive
  radial increasing (w.r.t.\ $\abs{x}$) solution to~\eqref{pblP}
  which belongs to the right bifurcation
  branch coming from $p= 2+ \lambdarad_2 / \lambda$.

  \item If $2 + \lambdarad_2 / \lambda \le 2^*$, the least energy radial solutions belong to the
  radial bifurcation branches coming from $p= 2+\lambdarad_2/\lambda$ when $2+ \lambdarad_2 / \lambda < p< 2^*$
  and, moreover, they are decreasing functions of~$\abs{x}$;
  \item \label{Conj:abcd-pointc}
    There exists a turning point $\bar p< 2 + \lambdarad_2 /\lambda$,
    such that the solution is degenerate at $\bar p$ and there exists two
    decreasing radial solutions for $\bar p< p < 2 + \lambdarad_2
    /\lambda$.   Moreover the least energy radial solution becomes
    non constant at some $p_c\in \intervalco{\bar p,
      2 + \lambdarad_2 /\lambda}$.
  \item If $2+\lambdarad_2 / \lambda > 2^*$, any radial positive ground
    state solution to \eqref{pblP}, $p \in \intervaloc{2, 2^*}$,
    is the constant function~$\lambda^{1/(p-2)}$.
  \end{enumerate}
\end{Conj}
The non degeneracy and uniqueness of the radial increasing solution is
proved for large $p$ in \cite{BGN} so that (a) holds true at least
asymptotically as $p\to\infty$.  The proof relies on a careful blow up
argument which crucially uses the identification of a limit problem
for $p\to\infty$.

\subsection{Further conjectures and open questions}
\label{num:open-questions}%

As proved previously, all bran\-ches $B_i^-$ in the set of , starting to
the right of $p = 2 + \lambdarad_i(B_R)$, exist for all $p \in
\intervaloo{2 + \lambdarad_i, +\infty}$.  A natural question is what
happens to the radial branch $B_i^+$ starting to the left of $p = 2 +
\lambdarad_i(B_R)$.  Figures~\ref{fig:bifuc-simple}, \ref{fig:bifurcation}
and~\ref{fig:bifurcation-3} picture the numerical computation of
such branches.  They make clear
that when $2 + \lambdarad_i < 2^*$, the branch $B_i^+$ --- which was
shown to exist for all $p \in \intervalco{2+\lambdarad_i - \delta_i,
  2^*}$ for some $\delta_i > 0$ --- blows up when $p \to 2^*$.  The
following conjecture is therefore natural.



\begin{figure}[htb]
  \centering
  \begin{tikzpicture}[x=15ex,y=1.1ex]
    \draw[->] (2,0) -- (6.5,0) node[below]{$p$};
    \draw[->] (2,0) -- +(0, 33) node[left]{$u(0)$};
    \foreach \p in {2,...,5} {%
      \draw (\p, 3pt) -- (\p, -3pt) node[below]{\scriptsize $\p$};
    }
    \draw (6, 3pt) -- (6, -3pt) node[below, xshift=0.5ex]{\scriptsize $6$};
    \foreach \u in {0,2,...,10} {%
      \draw (2,\u) ++(3pt, 0) -- +(-6pt, 0) node[left]{\scriptsize $\u$};
    }
    \foreach \u/\y in {12/11.64506, 15/13.32699, 20/15.1885,
      30/17.48589, 50/20.09040, 100/23.3692, 200/26.50784, 400/29.5809} {%
      \draw (2,\y) ++(3pt, 0) -- +(-6pt, 0) node[left]{\scriptsize $\u$};
    }
    \draw[dashed, color=umons-gray] (4,0) -- +(0,22.5)
    node[pos=0.3, xshift=-1.4ex, rotate=90] {$p = 2^*$};
    \draw[dashed] (2,1) -- (6.2, 1);
    \draw[dashed, color=umons-gray] (5.38243, 0) -- +(0, 30);
    \begin{scope}
      \clip (2,0) rectangle (6.2, 32);
      \foreach \t in {2+, 2-, 3+, 3-, 4+, 4-, 5+, 5-, 6+, 6-}{%
        \draw[thick, smooth, color=b\t]
        plot file{\graphpath bifurcFEM-N4-R9-\t.dat};
      }
    \end{scope}
    \foreach \i/\l in {2/2.32561, 3/2.87469, 4/3.66692, 5/4.70272,
      6/5.98216}{%
      \draw[fill, color=b\i+, fill=white] (\l,1) circle(1.5pt);
      \draw[<-] (\l, -3pt) -- ++(0, -12pt)
      node[below, yshift=2pt] {\scriptsize $2 + \lambdarad_{\i}$};
    }
  \end{tikzpicture}

  \vspace*{-1ex}
  \caption{Some radial bifurcation branches from~$1$ ($N=4$, $R=9$).
    The vertical axis is linear up to $10$ and then smoothly
    switches to a logarithmic scale.}
  \label{fig:bifurcation-3}
\end{figure}
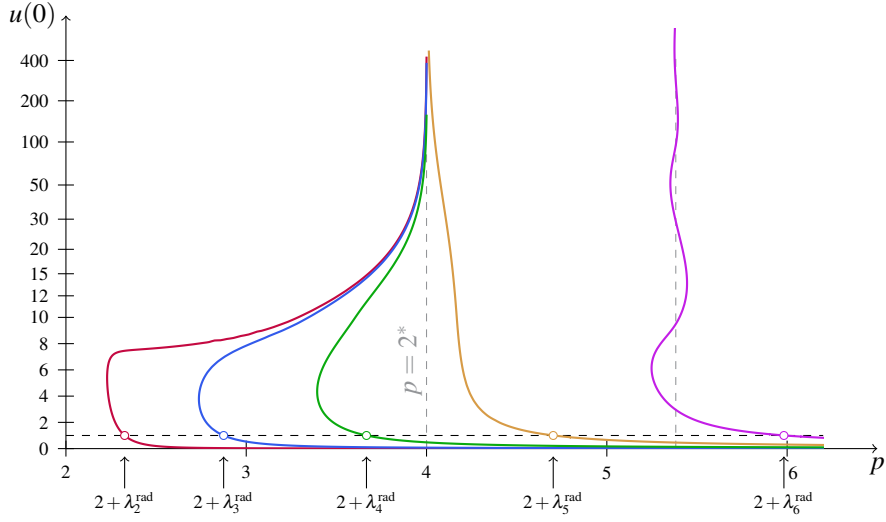

\begin{Conj}
  Assume that $2 + \lambdarad_i < 2^*$, $i \ge 2$, and
  let $(u_p)_{2+\lambdarad_i < p < 2^*}$ be a family of positive
  radial solutions of type~$i_+$.
  When $p \to 2^*$, $\abs{u_p}_\infty \to +\infty$
  and the solution bifurcates from infinity. In particular, we have 
  \begin{equation*}
  \frac{1}{\abs{u_p}_\infty}
  u_p\bigl( \abs{u_p}_\infty^{1- p/2} x \bigr)
   \to
    \Bigl( \frac{N(N-2)}{N(N-2) + \abs{x}^2} \Bigr)^{(N-2)/2}
  \end{equation*}
  uniformly on compact sets.
\end{Conj}

The fact that all branches $B^+_i$ starting from
$2 + \lambdarad_i < 2^*$ blow up (as indicated by
Figure~\ref{fig:bifurcation-3}) also implies that, if
$2 + \lambdarad_{n+1} < 2^*$, then Problem~\eqref{pblP} possesses $n$
positive solutions (distinguished by the nodal properties of $u-1$)
whenever $p$ is slightly subcritical.  These solution concentrate
at~$0$ when $p \xrightarrow{<} 2^*$.
The existence of \emph{one} concentrating solution for slightly
subcritical problems was established by O.~Rey and J.~Wei.  In dimension
three~\cite{reywei1}, it concentrates at an interior point
while, in dimensions four or greater and for
non-convex domains~$\Omega$~\cite{reywei2}, it concentrates on the
least curved part of the boundary~$\partial\Omega$.
What Figure~\ref{fig:bifurcation-3} suggests is that there should
actually exist a bubble tower at an interior point (as was shown for
the Dirichlet case~\cite{MR1998635})


The behavior of the branch $B_i^+$ is more tricky when $2 +
\lambdarad_i > 2^*$.  We will first focus on the case $i=2$.
The shape of the branch depends how small the radius $R$ is but also
on the dimension~$N$.
On Figure~\ref{fig:bifurcation-1st}, the thick line is the branch
bifurcating from $(2 + \lambdarad_2, 1)$ and the thin line is another
continuum of positive radially decreasing solutions of type~$2_+$.
These graphs suggest that, when $N \in \{ 4,5,6 \}$, no matter how
large $2 +\lambdarad_2 > 2^*$ is, the left branch starting from $(2 +
\lambdarad_2, 1)$ always goes below~$2^*$, then makes a U-turn and
blows up in~$L^\infty$ as $p \to 2^*$. Thus this branch always
crosses $p = 2^*$, which implies the existence of a radially
decreasing solution for the critical exponent in accordance with the result of Adimurthi
\& S.\ L.\ Yadava~\cite{Adimurthi-Yadava91}.
In this case, for $p \le 2^*$, the behavior of the energy along the
branch behaves as depicted in
Figures~\ref{fig:bifurcation}--\ref{fig:bifurcation-energy}:
after the turning point, the radial Morse index changes from
$2$ to $1$ and, for $p$ close enough to $2^*$,
the branch has lower energy than~$1$.
These numerical computations thus suggest that, in this case, the
solution~$1$ stops being the radial ground state before $2^*$
and this is not due to a sub-critical bifurcation from~$1$ but most likely
to a bifurcation from infinity (this is part of the
assertion~\eqref{Conj:abcd-pointc} in Conjecture \ref{Conj:abcd}).
For $N =3$ or $N \ge 7$, Figure~\ref{fig:bifurcation-1st} shows
that when $R$ becomes small, the branch emanating from $(2 +
\lambdarad_2, 1)$ does not cross $p=2^*$.  On the graphs, there is
another branch of positive solutions of type~$2_+$ coming from infinity
but this branch must disappear for smaller $R$ because
positive solutions are necessarily constant for a sufficiently
small radius~\cite{Adimurthi-Yadava91, Adimurthi-Yadava97}.

Figure~\ref{fig:bifurcation-3} also indicates is that, for $R$ large
enough, the branch $B_i^+$ emanating from the first $2 + \lambdarad_i$
greater than $2^*$ is asymptotic to $2^*$.  Along that branch, the
solutions concentrate at the origin as $p \to 2^*$.  Proving that this
behavior indeed takes place would be a nice complement to results
showing the existence of solutions concentrating on the boundary of
the domain as
$p \xrightarrow{>} 2^*$~\cite{delpinomussopistoia,MR3262455}.
In addition (as again illustrated by
Figure~\ref{fig:bifurcation-1st}), notice that the branches
bifurcating from higher $2 + \lambdarad_i$ oscillate around
some blow up value $p$.  This behavior was proved by
Miyamoto~\cite{Miyamoto} for~\eqref{pblP} but when looking to the
bifurcation diagram w.r.t.\ the parameter $\lambda$.  It would be
interesting to perform a similar analysis w.r.t.\ the parameter $p$
and analyze the curves in the $(\lambda, p)$-plane for which
entire singular solutions exist (which give the values of the
asymptotes of the branches of solutions).

\begin{figure}[htb]
  \centering
  
  \newcommand{\graph}[4]{
    \begin{tikzpicture}[x=4.ex, y=1pt]
      \draw[->] (2,0) -- (7.5,0);
      \node[below] at (7.3, 0) {\scriptsize \rlap{\hspace{1pt}$p$}};
      \draw[->] (2,0) -- (2, 81) node[left]{\scriptsize \llap{$u(0)$}};
      \foreach \x in {2,...,7}{%
        \draw (\x, 3pt) -- (\x, -3pt) node[below]{\scriptsize $\x$};
      }
      \foreach \y in {0,10,...,70}{%
        \draw (2,\y) +(3pt, 0) -- +(-3pt, 0) node[left]{\scriptsize $\y$};
      }
      \draw[dashed] (#3, 0) -- +(0, 75);
      \begin{scope}[color=umons-red]
        \clip (2,-3pt) rectangle (7.3, 78);
        \draw[thick] plot file{\graphpath bifurcation-N#1-R#2-low.dat};
        \draw plot file{\graphpath bifurcation-N#1-R#2-up.dat};
        \draw[fill=white] (#4, 1) circle(1.5pt);
      \end{scope}
      \node[above] at (4.5, 75) {$N = #1,\ R = #2$};
      \node[below] at (4.5, -2.5ex) {$2+\lambda_2 \approx #4$};
    \end{tikzpicture}%
  }

  \graph{4}{3}{4}{4.93}
  \hfill
  \graph{4}{2}{4}{8.59}
  \hfill
  \graph{4}{1.67}{4}{12.30}

  \vspace{0.5ex}

  \graph{7}{5}{2.8}{3.95}
  \hfill
  \graph{7}{4.9}{2.8}{4.03}
  \hfill
  \graph{7}{4.8}{2.8}{4.12}

  \caption{Branch emanating from $2 + \lambdarad_2 > 2^*$.}
  \label{fig:bifurcation-1st}
\end{figure}
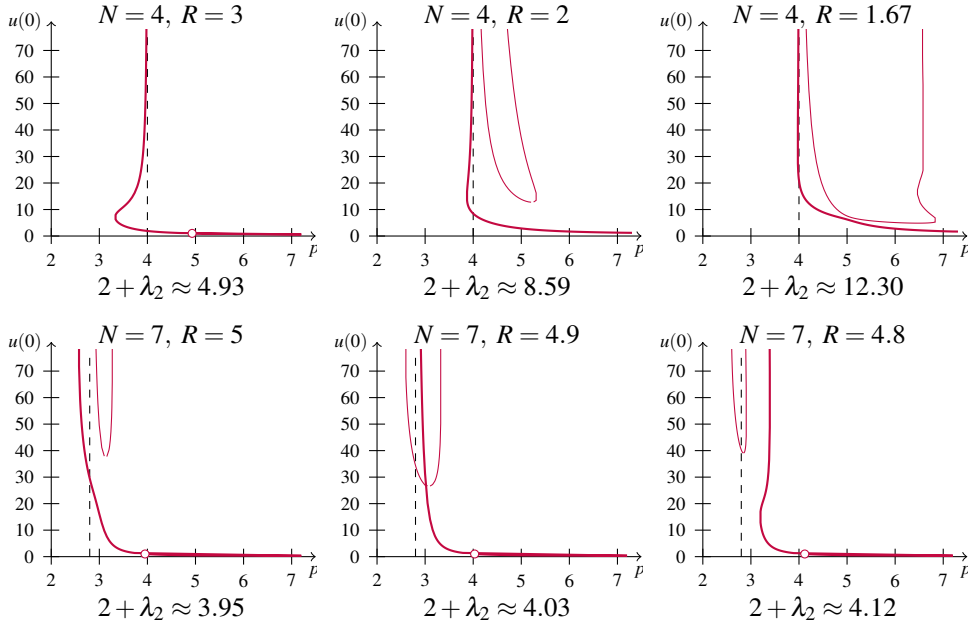

\subsection{Evidence of concentration for the
  singular perturbation
  problem~\texorpdfstring{\eqref{pblE}}{(Pε)}}
\label{sec:concentration-singular-perturb}

In this section, we compute solutions to problem~\eqref{pblE} when
$\epsilon$ is small.
The bifurcation diagram for $N=3$, $R=4$, $f(u) = \abs{u} \, u$
is drawn in Figure~\ref{fig:bifurcation-eps}.
Note that, for this $f$, the values of $\epsilon$ where bifurcation
occurs (see~\eqref{eq:epsiloni}) are $\epsilon_2 \approx 0.7924$,
$\epsilon_3 \approx 0.2681$, $\epsilon_4 \approx 0.1346$,
$\epsilon_5 \approx 0.0809$, $\epsilon_6 \approx 0.0540$,
$\epsilon_7 \approx 0.0386$,...
Figure~\ref{fig:rad:epsilon->0} displays
solutions for $\epsilon = 0.05$ on the branches $C^\pm_i$,
$i=2,\dotsc, 6$, and shows that the ``bumps'' cluster
around the boundary.
Further evidence that the oscillations of the radial solutions of type
$i_+$ ($i \ge 3$) and $i_-$ ($i \ge 2$) accumulate near the boundary
as $\epsilon \to 0$ is given by
Fig.~\ref{fig:bifurcation-eps-profile} where solutions on the branches
$C^\pm_4$ are drawn.
As a consequence, the bifurcating branches from $(\epsilon_i,1)$
provide an easy
way to construct clustered layer solutions~\cite{Malchiodi-Ni-Wei05}.
Since the cubic nonlinearity is subcritical in dimension $3$, the
solutions of type $i_+$ develop a (bounded) peak at the origin as
$\varepsilon\to 0$, the profile being asymptotically that of the rescaled ground state
solution in $\IR^3$.

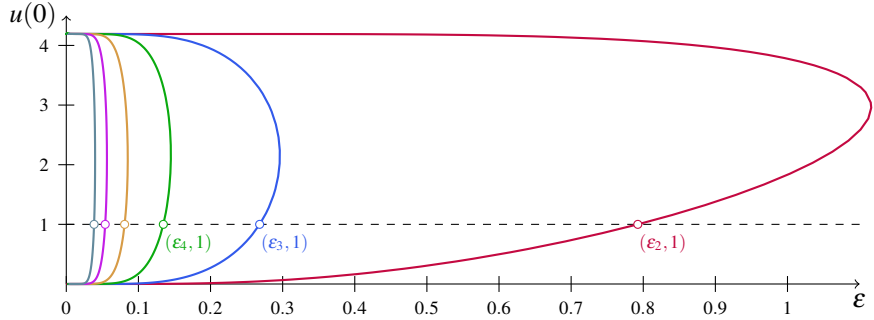
\begin{figure}[htb]
  \centering

  \begin{tikzpicture}[x=60ex,y=5ex]
    \draw[->] (0,0) -- (1.1,0) node[below]{$\epsilon$};
    \draw[->] (0,0) -- +(0, 4.5) node[left]{$u(0)$};
    \foreach \x in {0, 0.1, 0.2, 0.3, 0.4, 0.5, 0.6, 0.7, 0.8, 0.9, 1}{
      \draw (\x, 3pt) -- (\x, -3pt) node[below]{\scriptsize $\x$};
    }
    \foreach \y in {1, 2, 3, 4} {
      \draw (3pt, \y) -- (-3pt, \y) node[left]{\scriptsize $\y$};
    }
    \draw[dashed] (0,1) -- (1.1, 1);
    \foreach \t in {2+, 2-, 3+, 3-, 4+, 4-, 5+, 5-, 6+, 6-, 7+, 7-}{
      \draw[thick, color=b\t]
      plot file {\graphpath bifurcFEM-eps-N3-R4-\t.dat};
    }
    \foreach \i/\e in {2/0.792443, 3/0.268099, 4/0.134567,
      5/0.0808662, 6/0.053953, 7/0.0385551}{
      \draw[fill, color=b\i+, fill=white] (\e, 1) circle(1.5pt);
      \ifthenelse{\i < 5}{
        \node[below right, color=b\i+, xshift=-3pt] at (\e, 1) {%
          \scriptsize $(\epsilon_\i,1)$};
      }{}
    }
  \end{tikzpicture}

  \vspace*{-1ex}
  \caption{Some radial bifurcation branches emanating 
    from~$(\epsilon_i, 1)$ ($N=3$, $p=3$, $R=4$).}
  \label{fig:bifurcation-eps}
\end{figure}

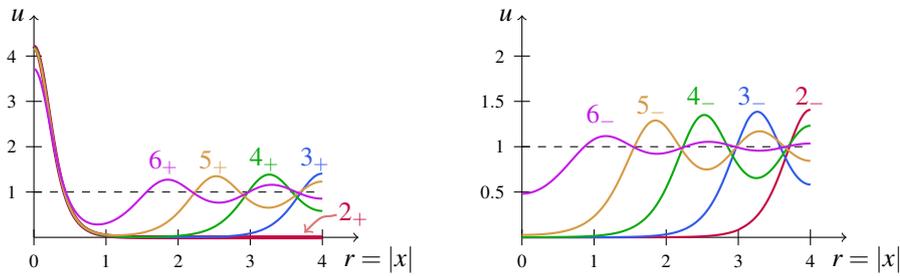
\begin{figure}[!hbt]
  \newcommand{\R}{4}
  \null\hfill

  \begin{tikzpicture}[x=6ex,y=3.8ex]
    \newcommand{\ymin}{0}%
    \newcommand{\ymax}{4.9}%
    \draw[->] (0,\ymin) -- ++(\R + 0.5, 0)
    node[below right, xshift=-2ex]{$r = \abs{x}$};
    \draw[->] (0,\ymin) -- (0, \ymax) node[left]{$u$};
    \foreach \x in {0,...,\R}{%
      \draw (\x,\ymin) ++(0,3pt) -- ++(0,-6pt)
      node[below]{\scriptsize $\x$};
    }
    \foreach \y in {1,...,4}{%
      \draw (0, \y) ++(3pt,0) -- ++(-6pt,0) node[left]{\scriptsize $\y$};
    }
    \draw[dashed] (0,1) -- (\R,1);

    \draw[line width = 1.5pt, color=b2+] plot
    file{\graphpath radial-N3-R4-p3-eps0.05-2+.dat};
    \draw[<-, color=b2+!60, line width=0.8pt]
    (3.75, 0.1) .. controls +(0.3, 0.5) and (4, 0.4)
    .. (4.2, 0.5) node[right, xshift=-3pt, color=b2+] {$2_+$};
    \draw[thick, color=b3+] plot
    file{\graphpath radial-N3-R4-p3-eps0.05-u0-4.19161.dat};
    \node[color=b3+, above] at (3.9, 1.25) {$3_+$};
    \draw[thick, color=b4+] plot
    file{\graphpath radial-N3-R4-p3-eps0.05-u0-4.19041.dat};
    \node[color=b4+, above] at (3.2, 1.25) {$4_+$};
    \draw[thick, color=b5+] plot
    file{\graphpath radial-N3-R4-p3-eps0.05-u0-4.16521.dat};
    \node[color=b5+, above] at (2.5, 1.2) {$5_+$};
    \draw[thick, color=b6+] plot
    file{\graphpath radial-N3-R4-p3-eps0.05-u0-3.71811.dat};
    \node[color=b6+, above] at (1.8, 1.2) {$6_+$};
  \end{tikzpicture}
  %
  %
  \hfill
  \begin{tikzpicture}[x=6ex,y=7.6ex]
    \newcommand{\ymax}{2.45}%
    \newcommand{\ymin}{0.}%
    \draw[->] (0,\ymin) -- ++(\R + 0.5, 0)
    node[below right, xshift=-2ex]{$r = \abs{x}$};
    \draw[->] (0,\ymin) -- (0, \ymax) node[left]{$u$};
    \foreach \x in {0,...,\R}{%
      \draw (\x,\ymin) ++(0,3pt) -- ++(0,-6pt)
      node[below]{\scriptsize $\x$};
    }
    \foreach \y in {0.5, 1, 1.5, 2}{%
      \draw (0, \y) ++(3pt,0) -- ++(-6pt,0) node[left]{\scriptsize $\y$};
    }
    \draw[dashed] (0,1) -- (\R,1);

    \draw[thick, color=b2-] plot
    file{\graphpath radial-N3-R4-p3-eps0.05-u0-2.89993e-06.dat};
    \node[color=b2+, above] at (4, 1.3) {$2_-$};
    \draw[thick, color=b3-] plot
    file{\graphpath radial-N3-R4-p3-eps0.05-u0-7.81442e-05.dat};
    \node[color=b3-, above] at (3.2, 1.3) {$3_-$};
    \draw[thick, color=b4-] plot
    file{\graphpath radial-N3-R4-p3-eps0.05-u0-0.00167092.dat};
    \node[color=b4-, above] at (2.5, 1.3) {$4_-$};
    \draw[thick, color=b5-] plot
    file{\graphpath radial-N3-R4-p3-eps0.05-u0-0.0268115.dat};
    \node[color=b5-, above] at (1.8, 1.2) {$5_-$};
    \draw[thick, color=b6-] plot
    file{\graphpath radial-N3-R4-p3-eps0.05-u0-0.481319.dat};
    \node[color=b6-, above] at (1.1, 1.1) {$6_-$};
  \end{tikzpicture}
  \hfill\null

  \vspace{-2ex}
  \caption{Non constant radial solutions
    ($N = 3$, $p = 3$, $R = 4$, 
    $\epsilon = 0.05$).
  }
  \label{fig:rad:epsilon->0}
\end{figure}

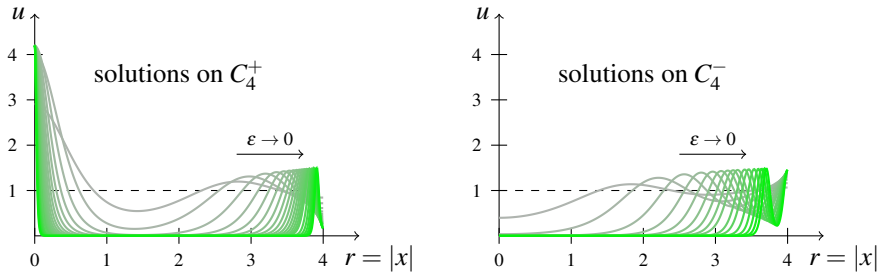
\begin{figure}[!hbt]
  \newcommand{\R}{4}
  \newcommand{\branch}{4}
  \newcommand{\ymin}{0}%
  \newcommand{\ymax}{4.9}%
  \null\hfill
  \begin{tikzpicture}[x=6ex,y=3.8ex]
    \draw[->] (0,\ymin) -- ++(\R + 0.5, 0)
    node[below right, xshift=-2ex]{$r = \abs{x}$};
    \draw[->] (0,\ymin) -- (0, \ymax) node[left]{$u$};
    \foreach \x in {0,...,\R}{%
      \draw (\x,\ymin) ++(0,3pt) -- ++(0,-6pt)
      node[below]{\scriptsize $\x$};
    }
    \foreach \y in {1,...,4}{%
      \draw (0, \y) ++(3pt,0) -- ++(-6pt,0) node[left]{\scriptsize $\y$};
    }
    \draw[dashed] (0,1) -- (\R,1);
    \node[above] at (2, 3) {solutions on $C^+_{\branch}$};

    \foreach \j in {10,25,...,240, 240,280,...,500}{
      \pgfmathtruncatemacro{\k}{\j / 2}%
      \draw[thick, color=b\branch+!\k!gray!60] plot
      file{\graphpath bifurcFEM-eps-N3-R4-p3-\branch+-\j.dat};
    }
    \draw[->] (2.8, 1.8) -- +(2.5em, 0) node[pos=0.5, above]{%
      \scriptsize $\epsilon \to 0$};
  \end{tikzpicture}
  \hfill
  \begin{tikzpicture}[x=6ex,y=3.8ex]
    \draw[->] (0,\ymin) -- ++(\R + 0.5, 0)
    node[below right, xshift=-2ex]{$r = \abs{x}$};
    \draw[->] (0,\ymin) -- (0, \ymax) node[left]{$u$};
    \foreach \x in {0,...,\R}{%
      \draw (\x,\ymin) ++(0,3pt) -- ++(0,-6pt)
      node[below]{\scriptsize $\x$};
    }
    \foreach \y in {1,...,4}{%
      \draw (0, \y) ++(3pt,0) -- ++(-6pt,0) node[left]{\scriptsize $\y$};
    }
    \draw[dashed] (0,1) -- (\R,1);
    \node[above] at (2, 3) {solutions on $C^-_{\branch}$};

    \foreach \j in {10,30,...,210, 240,290,...,500}{
      \pgfmathtruncatemacro{\k}{\j / 2}%
      \draw[thick, color=b\branch-!\k!gray!60] plot
      file{\graphpath bifurcFEM-eps-N3-R4-p3-\branch--\j.dat};
    }
    \draw[->] (2.5, 1.8) -- +(2.5em, 0) node[pos=0.5, above]{%
      \scriptsize $\epsilon \to 0$};
  \end{tikzpicture}
  \hfill\null

  \caption{Profile of the solutions along the branch emanating
    from~$\epsilon_{\branch}$  ($N=3$, $p=3$, $R=4$).}
  \label{fig:bifurcation-eps-profile}
\end{figure}

We now examine more complex nonlinearities~$f$ which possess
fixpoints in the interval~$\intervaloo{0,1}$ i.e.,
Problem~(\ref{pblE}) possesses more constant solutions than $0$
and~$1$.  Such fix point will generate an additional homoclinic
(asymptotic to this point) for the conservative limit equation
$-v'' + v = f(v)$ which will trap the continuum emanating from~$1$
preventing it from reaching the homoclinic asymptotic to~$0$
as for the pure power nonlinearity.

More specifically, we consider a function $f_1$ such that
$u \mapsto F_1(u) - u^2/2$ (where $F_1(u) = \int_0^u f_1$) possesses a
single (necessarily degenerate) critical point in~$\intervaloo{0,1}$
and $f_2$ such that $u \mapsto F_1(u) - u^2/2$ has a local minimum
$u^*_1$ and a local maximum $u^*_2$ in~$\intervaloo{0,1}$.  These
functions are pictured on Fig.~\ref{fig:non-linearities}.

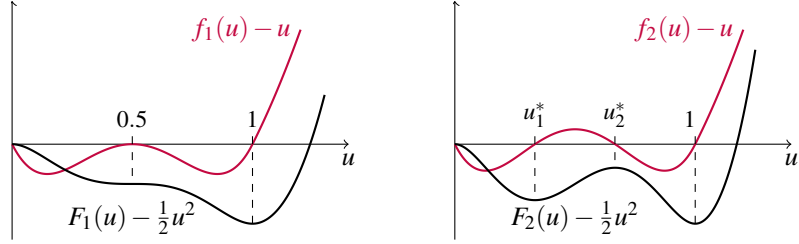
\begin{figure}[htb]
  \centering
  \begin{tikzpicture}[x=20ex, y=40ex]
    \draw[->] (0,0) -- (1.4, 0) node[below]{$u$};
    \draw[->] (0, -0.2) -- (0, 0.3);
    \draw[thick, color=umons-red] plot file{\graphpath f1.dat}
    node[left]{$f_1(u) - u$};
    \draw[thick] plot file{\graphpath F1.dat};
    \node[below] at (0.5, -0.1) {$F_1(u) - \tfrac{1}{2}u^2$};
    \draw[dashed] (1,0) ++(0, 3pt) node[above]{\small $1$}
    -- (1, -0.166);
    \draw[dashed] (0.5,0) ++(0, 3pt) node[above]{\small $0.5$}
    -- (0.5, -0.083);
  \end{tikzpicture}
  \hfil
  \begin{tikzpicture}[x=20ex, y=40ex]
    \draw[->] (0,0) -- (1.4, 0) node[below]{$u$};
    \draw[->] (0, -0.2) -- (0, 0.3);
    \draw[thick, color=umons-red] plot file{\graphpath f2.dat}
    node[left]{$f_2(u) - u$};
    \draw[thick] plot file{\graphpath F2.dat};
    \node[below] at (0.5, -0.1) {$F_2(u) - \tfrac{1}{2}u^2$};
    \draw[dashed] (1,0) +(0, 3pt) node[above, xshift=-2pt]{\small $1$}
    -- (1, -0.166);
    \draw[dashed] (0.333, 0) +(0,3pt) node[above]{\small $u^*_1$}
    -- +(0, -0.117);
    \draw[dashed] (0.666, 0) +(0,3pt) node[above]{\small $u^*_2$}
    -- +(0, -0.049);
  \end{tikzpicture}

  \caption{Nonlinearities $f_i$, $i=1,2$.}
  \label{fig:non-linearities}
\end{figure}

The nonlinearities $f_1$ and $f_2$ are chosen in such a way that the
bifurcations from~$1$ occur for the same $\epsilon_i$, $i \ge 2$, as
for the above pure power ($p=3$).

For both nonlinearities,
Figs.~\ref{fig:f1:rad:epsilon->0}--\ref{fig:f2:rad:epsilon->0} show
that the solutions bifurcating from~$1$ behave similarly to those of
the pure power case except that their bumps resemble to the homoclinic
starting from~$0.5$ or $u^*_2$.  Note that, for $f=f_1$, the speed of
concentration of the bumps is likely to be of order $\epsilon^\alpha$
for some $\alpha > 0$ due to the degeneracy of the critical point
$0.5$ which implies that the associated homoclinic decays like a power.

For $f=f_1$, there are additional solutions $u$ with
$u(0) \in \intervaloo{0, 0.5}$.  These solutions seem to come in
pairs: for $\epsilon > 0$ small enough, there are two solutions of
type $i_-$, one starting close to the homoclinic asymtotic to~$0$ and
another one increasing to $0.5$ and then resembling the homoclinic
asymtotic to $0.5$ before oscillating around~$1$
(see the right graph of Fig.~\ref{fig:f1:rad:epsilon=0.02}).

\begin{figure}[!hbt]
  \newcommand{\R}{4}
  \newcommand{\ymax}{1.45}%
  \newcommand{\ymin}{0.}%
  \newcommand{\solution}[3][]{%
    \draw[#1] plot file{\graphpath radial-N3-R4-f4-eps#2-u0-#3.dat};}%
  \null\hfill
  \begin{tikzpicture}[x=6ex,y=12ex]
    \node[above] at (2, 1.4) {$\epsilon = 0.25$};
    \draw[->] (0,\ymin) -- ++(\R + 0.5, 0)
    node[below right, xshift=-2ex]{$r = \abs{x}$};
    \draw[->] (0,\ymin) -- (0, \ymax) node[left]{$u$};
    \foreach \x in {0,...,\R}{%
      \draw (\x,\ymin) ++(0,3pt) -- ++(0,-6pt)
      node[below]{\scriptsize $\x$};
    }
    \foreach \y in {0.5, 1}{%
      \draw (0, \y) ++(3pt,0) -- ++(-6pt,0) node[left]{\scriptsize $\y$};
    }
    \draw[dashed] (0,1) -- +(\R,0);
    \draw[dashed] (0, 0.5) -- +(\R,0);

    \solution[thick, color=b2-]{0.25}{0.679194}
    \node[color=b2-, below] at (2.9, 0.9) {$2_-$};
    \solution[thick, color=b3-]{0.25}{0.949853}
    \node[color=b3-, above] at (2, 1.) {$3_-$};
  \end{tikzpicture}
  \hfill
  \begin{tikzpicture}[x=6ex,y=12ex]
    \node[above] at (2, 1.4) {$\epsilon = 0.05$};
    \draw[->] (0,\ymin) -- ++(\R + 0.5, 0)
    node[below right, xshift=-2ex]{$r = \abs{x}$};
    \draw[->] (0,\ymin) -- (0, \ymax) node[left]{$u$};
    \foreach \x in {0,...,\R}{%
      \draw (\x,\ymin) ++(0,3pt) -- ++(0,-6pt)
      node[below]{\scriptsize $\x$};
    }
    \foreach \y in {0.5, 1}{%
      \draw (0, \y) ++(3pt,0) -- ++(-6pt,0) node[left]{\scriptsize $\y$};
    }
    \draw[dashed] (0, 1) -- +(\R,0);
    \draw[dashed] (0, 0.5) -- +(\R,0);

    \solution[color=b2-]{0.05}{0.000243914}
    \solution[color=b2-]{0.05}{0.115623}
    \solution[thick, color=b2-]{0.05}{0.541756}
    \node[color=b2-, above] at (3.4, 0.45) {$2_-$};
    \solution[thick, color=b3-]{0.05}{0.563124}
    \node[color=b3-, above] at (3.2, 1.1) {$3_-$};
    \solution[thick, color=b4-]{0.05}{0.603638}
    \node[color=b4-, above] at (2.5, 1.1) {$4_-$};
    \solution[thick, color=b5-]{0.05}{0.683576}
    \node[color=b5-, above] at (1.8, 1.1) {$5_-$};
    \solution[thick, color=b6-]{0.05}{0.902619}
    \node[color=b6-, above] at (1.1, 1.) {$6_-$};
  \end{tikzpicture}
  \hfill\null

  \vspace{-2ex}
  \caption{Non constant radial solutions
    ($N = 3$, $f = f_1$, $R = 4$, $\epsilon \to 0$).}
  \label{fig:f1:rad:epsilon->0}
\end{figure}

\begin{figure}[!hbt]
  \newcommand{\R}{4}
  \newcommand{\ymax}{1.45}%
  \newcommand{\ymin}{0.}%
  \newcommand{\solution}[2][]{%
    \draw[#1] plot file{\graphpath radial-N3-R4-f4-eps0.02-u0-#2.dat};}%
  \null\hfill
  \begin{tikzpicture}[x=6ex,y=12ex]
    \node[above] at (2, 1.4) {$\epsilon = 0.02$};
    \draw[->] (0,\ymin) -- ++(\R + 0.5, 0)
    node[below right, xshift=-2ex]{$r = \abs{x}$};
    \draw[->] (0,\ymin) -- (0, \ymax) node[left]{$u$};
    \foreach \x in {0,...,\R}{%
      \draw (\x,\ymin) ++(0,3pt) -- ++(0,-6pt)
      node[below]{\scriptsize $\x$};
    }
    \foreach \y in {0.5, 1}{%
      \draw (0, \y) ++(3pt,0) -- ++(-6pt,0) node[left]{\scriptsize $\y$};
    }
    \draw[dashed] (0,1) -- +(\R,0);
    \draw[dashed] (0, 0.5) -- +(\R,0);

    \solution[thick, color=b2-]{0.517672}
    \solution[thick, color=b3-]{0.523444}
    \solution[thick, color=b4-]{0.532241}
    \solution[thick, color=b5-]{0.544522}
    \solution[thick, color=b6-]{0.56452}
    \solution[thick, color=b7-]{0.597884}
    \solution[thick, color=b8-]{0.662927}
    \solution[thick, color=b9-]{0.81565}
    \node[color=b2-, right] at (3.45, 0.7) {\scriptsize $2_-$};
    \foreach \n/\r in {3/3.5, 4/2.9, 5/2.5, 6/2.05, 7/1.6}{%
      \node[color=b\n-, above] at (\r, 1.1) {\scriptsize $\n_-$};
    }
    \node[color=b8-, above] at (1.2, 1.05) {\scriptsize $8_-$};
    \node[color=b9-, above] at (0.8, 1) {\scriptsize $9_-$};
  \end{tikzpicture}
  \hfill
  \begin{tikzpicture}[x=6ex,y=12ex]
    \node[above] at (2, 1.4) {$\epsilon = 0.02$};
    \draw[->] (0,\ymin) -- ++(\R + 0.5, 0)
    node[below right, xshift=-2ex]{$r = \abs{x}$};
    \draw[->] (0,\ymin) -- (0, \ymax) node[left]{$u$};
    \foreach \x in {0,...,\R}{%
      \draw (\x,\ymin) ++(0,3pt) -- ++(0,-6pt)
      node[below]{\scriptsize $\x$};
    }
    \foreach \y in {0.5, 1}{%
      \draw (0, \y) ++(3pt,0) -- ++(-6pt,0) node[left]{\scriptsize $\y$};
    }
    \draw[dashed] (0, 1) -- +(\R,0);
    \draw[dashed] (0, 0.5) -- +(\R,0);

    \solution[thick, color=b2-]{1.1e-08}
    \solution[thick, color=b3-]{4.33549e-07}
    \solution[thick, color=b4-]{2.3981e-05}
    \solution[thick, color=b5-]{0.000813239}
    \solution[thick, color=b5-]{0.093006}
    \solution[thick, color=b4-]{0.195051}
    \solution[thick, color=b3-]{0.274162}
    \solution[thick, color=b2-]{0.328249}
    \node[color=b2-, right] at (3.35, 0.6) {$2_-$};
    \node[color=b3-, above] at (3.4, 1.1) {$3_-$};
    \node[color=b4-, above] at (2.9, 1.1) {$4_-$};
    \node[color=b5-, above] at (2.4, 1.1) {$5_-$};
  \end{tikzpicture}
  \hfill\null

  \vspace{-2ex}
  \caption{Non constant radial solutions
    ($N = 3$, $f = f_1$, $R = 4$, $\epsilon = 0.02$).}
  \label{fig:f1:rad:epsilon=0.02}
\end{figure}
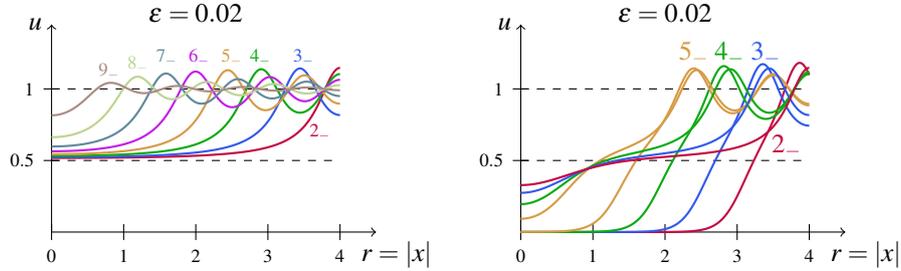

For $f=f_2$, the additional numerically computed solutions oscillate
around the second local minimum $u^*_1$ of
$u \mapsto F_2(u) - \frac{1}{2}u^2$ (see
Fig.~\ref{fig:f2:rad:epsilon->0}).  For these solutions, the
classification in types $i_\pm$ has to be adapted to count the number
of zeros of $u - u^*_1$ with the subscript $\pm$ being the sign of
$u(0) - u^*_1$.  Assuming that $f'(u^*_1) > 1$ (i.e., that the minimum
$u^*_1$ is non-degenerate) and following similar arguments to those
developed above, one can prove a multiplicity result such as
Corollary~\ref{multiplicity:espilon->0}.  In this case however, both
solutions of type $i_-$ and $i_+$ will keep existing no matter how
small $\epsilon > 0$ is,
so, when $\epsilon < (f'(u^*_1) - 1) / \lambdarad_{n+1}$, one will
actually have at least $2n$ solutions to Problem~(\ref{pblE}), one for
each type $i_\pm$, $2 \le i \le n+1$.

\begin{figure}[!hbt]
  \newcommand{\R}{4}
  \newcommand{\ymax}{1.45}%
  \newcommand{\ymin}{0.}%
  \newcommand{\solution}[3][]{%
    \draw[#1] plot file{\graphpath radial-N3-R4-f4.5-eps#2-u0-#3.dat};}%
  \null\hfill
  \begin{tikzpicture}[x=6ex,y=12ex]
    \node[above] at (2, 1.4) {$\epsilon = 0.25$};
    \draw[->] (0,\ymin) -- ++(\R + 0.5, 0)
    node[below right, xshift=-2ex]{$r = \abs{x}$};
    \draw[->] (0,\ymin) -- (0, \ymax) node[left]{$u$};
    \foreach \x in {0,...,\R}{%
      \draw (\x,\ymin) ++(0,3pt) -- ++(0,-6pt)
      node[below]{\scriptsize $\x$};
    }
    \foreach \y in {0.5, 1}{%
      \draw (0, \y) ++(3pt,0) -- ++(-6pt,0) node[left]{\scriptsize $\y$};
    }
    \draw[dashed] (0,1) -- +(\R,0);
    \draw[dashed] (0, 0.333333) -- +(\R,0) node[right]{\small $u^*_1$};
    \draw[dashed] (0, 0.666667) -- +(\R,0) node[right]{\small $u^*_2$};

    \solution[thick, color=b2+]{0.25}{0.113111}
    \node[color=b2+, above] at (1.6, 0.34) {$2_+$};
    \solution[thick, color=b2-]{0.25}{0.399173}
    \node[color=b2-, below] at (2.4, 0.3) {$2_-$};
    \solution[thick, color=b2+]{0.25}{0.727308}
    \node[color=b2+, below] at (2.9, 0.9) {$2_-$};
    \solution[thick, color=b3+]{0.25}{0.954806}
    \node[color=b3+, above] at (2, 1.) {$3_-$};
  \end{tikzpicture}
  \hfill
  \begin{tikzpicture}[x=6ex,y=12ex]
    \node[above] at (2, 1.4) {$\epsilon = 0.05$};
    \draw[->] (0,\ymin) -- ++(\R + 0.5, 0)
    node[below right, xshift=-2ex]{$r = \abs{x}$};
    \draw[->] (0,\ymin) -- (0, \ymax) node[left]{$u$};
    \foreach \x in {0,...,\R}{%
      \draw (\x,\ymin) ++(0,3pt) -- ++(0,-6pt)
      node[below]{\scriptsize $\x$};
    }
    \foreach \y in {0.5, 1}{%
      \draw (0, \y) ++(3pt,0) -- ++(-6pt,0) node[left]{\scriptsize $\y$};
    }
    \draw[dashed] (0, 1) -- +(\R,0);
    \draw[dashed] (0, 0.333333) -- +(\R,0) node[right]{\small $u^*_1$};
    \draw[dashed] (0, 0.666667) -- +(\R,0) node[right]{\small $u^*_2$};

    \solution[thick, color=b2-]{0.05}{1.59751e-05}
    \node[color=b2-, above, yshift=-1pt] at (2.3, 0.) {$2_-$};
    \solution[thick, color=b3-]{0.05}{0.00276109}
    \node[color=b3-, above] at (1, 0.) {$3_-$};
    \solution[thick, color=b3+]{0.05}{0.643381}
    \node[color=b3+, below] at (1, 0.63) {$3_+$};
    \solution[thick, color=b2+]{0.05}{0.665335}
    \node[color=b2+, below] at (4, 0.25) {$2_+$};
    \solution[thick, color=b2-]{0.05}{0.667147}
    \node[color=b2-, above] at (4., 1.1) {$2_-$};
    \solution[thick, color=b3-]{0.05}{0.669612}
    \node[color=b3-, above] at (3.2, 1.1) {$3_-$};
    \solution[thick, color=b4-]{0.05}{0.682263}
    \node[color=b4-, above] at (2.5, 1.05) {$4_-$};
    \solution[thick, color=b5-]{0.05}{0.730164}
    \node[color=b5-, above] at (1.8, 1.) {$5_-$};
    \solution[thick, color=b6-]{0.05}{0.912657}
    \node[color=b6-, above] at (1.1, 1.) {$6_-$};
  \end{tikzpicture}
  \hfill\null

  \vspace{-2ex}
  \caption{Non constant radial solutions
    ($N = 3$, $f = f_2$, $R = 4$, $\epsilon \to 0$).}
  \label{fig:f2:rad:epsilon->0}
\end{figure}

\bibliographystyle{plain}
\bibliography{neumann}

\end{document}

%% file: graphs/transcritical.tex
\draw[color=n2] plot coordinates{(0.000,0.000)(0.235,0.027)(0.470,0.102)(0.705,0.207)(0.940,0.319)(1.174,0.419)(1.409,0.493)(1.644,0.537)(1.879,0.557)(2.114,0.563)(2.349,0.564)(2.584,0.564)(2.819,0.561)(3.054,0.552)(3.289,0.529)(3.523,0.490)(3.758,0.438)(3.993,0.379)(4.228,0.323)(4.463,0.278)(4.698,0.249)(4.933,0.234)(5.168,0.229)(5.403,0.228)(5.638,0.228)(5.872,0.229)(6.107,0.233)(6.342,0.246)(6.577,0.271)(6.812,0.307)(7.047,0.351)(7.282,0.394)(7.517,0.432)(7.752,0.458)(7.987,0.473)(8.221,0.480)(8.456,0.481)(8.691,0.481)(8.926,0.481)(9.161,0.479)(9.396,0.471)(9.631,0.454)(9.866,0.427)(10.101,0.393)(10.336,0.356)(10.570,0.322)(10.805,0.296)(11.040,0.280)(11.275,0.272)(11.510,0.270)(11.745,0.270)(11.980,0.270)(12.215,0.271)(12.450,0.276)(12.685,0.287)(12.919,0.307)(13.154,0.335)(13.389,0.368)(13.624,0.399)(13.859,0.424)(14.094,0.442)(14.329,0.451)(14.564,0.454)(14.799,0.455)(15.034,0.455)(15.268,0.454)(15.503,0.452)(15.738,0.444)(15.973,0.428)(16.208,0.406)(16.443,0.378)(16.678,0.349)(16.913,0.323)(17.148,0.305)(17.383,0.294)(17.617,0.289)(17.852,0.288)(18.087,0.288)(18.322,0.288)(18.557,0.290)(18.792,0.295)(19.027,0.306)(19.262,0.325)(19.497,0.349)(19.732,0.376)(19.966,0.401)(20.201,0.420)(20.436,0.433)(20.671,0.439)(20.906,0.441)(21.141,0.441)(21.376,0.441)(21.611,0.440)(21.846,0.437)(22.081,0.429)(22.315,0.414)(22.550,0.393)(22.785,0.369)(23.020,0.345)(23.255,0.324)(23.490,0.310)(23.725,0.303)(23.960,0.300)(24.195,0.299)(24.430,0.299)(24.664,0.299)(24.899,0.301)(25.134,0.307)(25.369,0.318)(25.604,0.336)(25.839,0.358)(26.074,0.381)(26.309,0.402)(26.544,0.417)(26.779,0.427)(27.013,0.431)(27.248,0.432)(27.483,0.432)(27.718,0.432)(27.953,0.431)(28.188,0.427)(28.423,0.419)(28.658,0.404)(28.893,0.385)(29.128,0.363)(29.362,0.342)(29.597,0.325)(29.832,0.314)(30.067,0.309)(30.302,0.307)(30.537,0.306)(30.772,0.306)(31.007,0.307)(31.242,0.309)(31.477,0.315)(31.711,0.327)(31.946,0.344)(32.181,0.364)(32.416,0.385)(32.651,0.403)(32.886,0.415)(33.121,0.422)(33.356,0.425)(33.591,0.426)(33.826,0.426)(34.060,0.426)(34.295,0.424)(34.530,0.420)(34.765,0.411)(35.000,0.396)};
\draw[color=n2,fill] (36.000,0.368) circle(1pt);
\draw[color=n3] plot coordinates{(0.000,0.000)(0.235,0.002)(0.470,0.016)(0.705,0.051)(0.940,0.108)(1.174,0.182)(1.409,0.264)(1.644,0.341)(1.879,0.405)(2.114,0.450)(2.349,0.477)(2.584,0.489)(2.819,0.492)(3.054,0.493)(3.289,0.493)(3.523,0.492)(3.758,0.488)(3.993,0.479)(4.228,0.463)(4.463,0.441)(4.698,0.415)(4.933,0.391)(5.168,0.371)(5.403,0.358)(5.638,0.351)(5.872,0.348)(6.107,0.348)(6.342,0.347)(6.577,0.348)(6.812,0.349)(7.047,0.353)(7.282,0.360)(7.517,0.372)(7.752,0.386)(7.987,0.401)(8.221,0.415)(8.456,0.425)(8.691,0.431)(8.926,0.433)(9.161,0.434)(9.396,0.434)(9.631,0.434)(9.866,0.434)(10.101,0.432)(10.336,0.428)(10.570,0.420)(10.805,0.411)(11.040,0.400)(11.275,0.390)(11.510,0.382)(11.745,0.376)(11.980,0.373)(12.215,0.373)(12.450,0.372)(12.685,0.372)(12.919,0.372)(13.154,0.373)(13.389,0.376)(13.624,0.381)(13.859,0.387)(14.094,0.395)(14.329,0.404)(14.564,0.411)(14.799,0.416)(15.034,0.419)(15.268,0.420)(15.503,0.420)(15.738,0.420)(15.973,0.420)(16.208,0.420)(16.443,0.418)(16.678,0.415)(16.913,0.410)(17.148,0.404)(17.383,0.397)(17.617,0.391)(17.852,0.386)(18.087,0.383)(18.322,0.382)(18.557,0.381)(18.792,0.381)(19.027,0.381)(19.262,0.381)(19.497,0.382)(19.732,0.384)(19.966,0.388)(20.201,0.393)(20.436,0.398)(20.671,0.404)(20.906,0.409)(21.141,0.412)(21.376,0.414)(21.611,0.414)(21.846,0.414)(22.081,0.414)(22.315,0.414)(22.550,0.414)(22.785,0.412)(23.020,0.410)(23.255,0.406)(23.490,0.401)(23.725,0.396)(23.960,0.392)(24.195,0.389)(24.430,0.387)(24.664,0.386)(24.899,0.386)(25.134,0.386)(25.369,0.386)(25.604,0.386)(25.839,0.387)(26.074,0.388)(26.309,0.391)(26.544,0.395)(26.779,0.400)(27.013,0.404)(27.248,0.407)(27.483,0.409)(27.718,0.410)(27.953,0.411)(28.188,0.411)(28.423,0.411)(28.658,0.411)(28.893,0.410)(29.128,0.409)(29.362,0.407)(29.597,0.403)(29.832,0.400)(30.067,0.396)(30.302,0.392)(30.537,0.390)(30.772,0.389)(31.007,0.388)(31.242,0.388)(31.477,0.388)(31.711,0.388)(31.946,0.388)(32.181,0.389)(32.416,0.391)(32.651,0.394)(32.886,0.397)(33.121,0.401)(33.356,0.404)(33.591,0.406)(33.826,0.408)(34.060,0.409)(34.295,0.409)(34.530,0.409)(34.765,0.409)(35.000,0.409)};
\draw[color=n3,fill] (36.000,0.399) circle(1pt);
\draw[color=n4] plot coordinates{(0.000,0.000)(0.235,0.000)(0.470,0.001)(0.705,0.007)(0.940,0.020)(1.174,0.042)(1.409,0.075)(1.644,0.117)(1.879,0.162)(2.114,0.207)(2.349,0.246)(2.584,0.277)(2.819,0.297)(3.054,0.308)(3.289,0.313)(3.523,0.314)(3.758,0.315)(3.993,0.315)(4.228,0.314)(4.463,0.313)(4.698,0.309)(4.933,0.303)(5.168,0.294)(5.403,0.284)(5.638,0.275)(5.872,0.268)(6.107,0.263)(6.342,0.260)(6.577,0.259)(6.812,0.258)(7.047,0.258)(7.282,0.258)(7.517,0.259)(7.752,0.260)(7.987,0.262)(8.221,0.266)(8.456,0.270)(8.691,0.275)(8.926,0.279)(9.161,0.283)(9.396,0.285)(9.631,0.285)(9.866,0.286)(10.101,0.286)(10.336,0.286)(10.570,0.286)(10.805,0.285)(11.040,0.284)(11.275,0.282)(11.510,0.280)(11.745,0.277)(11.980,0.274)(12.215,0.272)(12.450,0.270)(12.685,0.269)(12.919,0.269)(13.154,0.269)(13.389,0.269)(13.624,0.269)(13.859,0.269)(14.094,0.270)(14.329,0.271)(14.564,0.272)(14.799,0.274)(15.034,0.276)(15.268,0.278)(15.503,0.279)(15.738,0.280)(15.973,0.281)(16.208,0.281)(16.443,0.281)(16.678,0.281)(16.913,0.281)(17.148,0.280)(17.383,0.280)(17.617,0.279)(17.852,0.277)(18.087,0.276)(18.322,0.274)(18.557,0.273)(18.792,0.272)(19.027,0.272)(19.262,0.272)(19.497,0.272)(19.732,0.272)(19.966,0.272)(20.201,0.272)(20.436,0.272)(20.671,0.273)(20.906,0.274)(21.141,0.275)(21.376,0.276)(21.611,0.277)(21.846,0.278)(22.081,0.279)(22.315,0.279)(22.550,0.279)(22.785,0.279)(23.020,0.279)(23.255,0.279)(23.490,0.278)(23.725,0.278)(23.960,0.277)(24.195,0.276)(24.430,0.275)(24.664,0.275)(24.899,0.274)(25.134,0.273)(25.369,0.273)(25.604,0.273)(25.839,0.273)(26.074,0.273)(26.309,0.273)(26.544,0.273)(26.779,0.274)(27.013,0.274)(27.248,0.275)(27.483,0.276)(27.718,0.276)(27.953,0.277)(28.188,0.277)(28.423,0.278)(28.658,0.278)(28.893,0.278)(29.128,0.278)(29.362,0.278)(29.597,0.278)(29.832,0.278)(30.067,0.277)(30.302,0.277)(30.537,0.276)(30.772,0.275)(31.007,0.275)(31.242,0.274)(31.477,0.274)(31.711,0.274)(31.946,0.274)(32.181,0.274)(32.416,0.274)(32.651,0.274)(32.886,0.274)(33.121,0.274)(33.356,0.275)(33.591,0.275)(33.826,0.276)(34.060,0.276)(34.295,0.277)(34.530,0.277)(34.765,0.277)(35.000,0.277)};
\draw[color=n4,fill] (36.000,0.276) circle(1pt);
\draw[color=n5] plot coordinates{(0.000,0.000)(0.235,0.000)(0.470,0.000)(0.705,0.001)(0.940,0.002)(1.174,0.006)(1.409,0.014)(1.644,0.025)(1.879,0.041)(2.114,0.061)(2.349,0.083)(2.584,0.104)(2.819,0.124)(3.054,0.140)(3.289,0.151)(3.523,0.158)(3.758,0.162)(3.993,0.164)(4.228,0.164)(4.463,0.164)(4.698,0.164)(4.933,0.164)(5.168,0.163)(5.403,0.162)(5.638,0.159)(5.872,0.156)(6.107,0.153)(6.342,0.150)(6.577,0.147)(6.812,0.146)(7.047,0.145)(7.282,0.144)(7.517,0.144)(7.752,0.144)(7.987,0.144)(8.221,0.144)(8.456,0.145)(8.691,0.145)(8.926,0.146)(9.161,0.148)(9.396,0.149)(9.631,0.150)(9.866,0.151)(10.101,0.152)(10.336,0.152)(10.570,0.152)(10.805,0.152)(11.040,0.152)(11.275,0.152)(11.510,0.152)(11.745,0.152)(11.980,0.152)(12.215,0.151)(12.450,0.150)(12.685,0.149)(12.919,0.149)(13.154,0.148)(13.389,0.148)(13.624,0.148)(13.859,0.148)(14.094,0.148)(14.329,0.148)(14.564,0.148)(14.799,0.148)(15.034,0.148)(15.268,0.149)(15.503,0.149)(15.738,0.150)(15.973,0.150)(16.208,0.150)(16.443,0.151)(16.678,0.151)(16.913,0.151)(17.148,0.151)(17.383,0.151)(17.617,0.151)(17.852,0.151)(18.087,0.151)(18.322,0.150)(18.557,0.150)(18.792,0.150)(19.027,0.149)(19.262,0.149)(19.497,0.149)(19.732,0.149)(19.966,0.149)(20.201,0.149)(20.436,0.149)(20.671,0.149)(20.906,0.149)(21.141,0.149)(21.376,0.149)(21.611,0.149)(21.846,0.149)(22.081,0.150)(22.315,0.150)(22.550,0.150)(22.785,0.150)(23.020,0.150)(23.255,0.150)(23.490,0.150)(23.725,0.150)(23.960,0.150)(24.195,0.150)(24.430,0.150)(24.664,0.150)(24.899,0.150)(25.134,0.150)(25.369,0.149)(25.604,0.149)(25.839,0.149)(26.074,0.149)(26.309,0.149)(26.544,0.149)(26.779,0.149)(27.013,0.149)(27.248,0.149)(27.483,0.149)(27.718,0.149)(27.953,0.149)(28.188,0.150)(28.423,0.150)(28.658,0.150)(28.893,0.150)(29.128,0.150)(29.362,0.150)(29.597,0.150)(29.832,0.150)(30.067,0.150)(30.302,0.150)(30.537,0.150)(30.772,0.150)(31.007,0.150)(31.242,0.150)(31.477,0.150)(31.711,0.149)(31.946,0.149)(32.181,0.149)(32.416,0.149)(32.651,0.149)(32.886,0.149)(33.121,0.149)(33.356,0.149)(33.591,0.149)(33.826,0.149)(34.060,0.149)(34.295,0.149)(34.530,0.150)(34.765,0.150)(35.000,0.150)};
\draw[color=n5,fill] (36.000,0.150) circle(1pt);
\draw[color=n6] plot coordinates{(0.000,0.000)(0.235,0.000)(0.470,0.000)(0.705,0.000)(0.940,0.000)(1.174,0.001)(1.409,0.002)(1.644,0.004)(1.879,0.007)(2.114,0.013)(2.349,0.019)(2.584,0.027)(2.819,0.036)(3.054,0.045)(3.289,0.054)(3.523,0.061)(3.758,0.066)(3.993,0.070)(4.228,0.072)(4.463,0.073)(4.698,0.074)(4.933,0.074)(5.168,0.074)(5.403,0.074)(5.638,0.074)(5.872,0.073)(6.107,0.073)(6.342,0.072)(6.577,0.071)(6.812,0.070)(7.047,0.069)(7.282,0.068)(7.517,0.068)(7.752,0.067)(7.987,0.067)(8.221,0.067)(8.456,0.067)(8.691,0.067)(8.926,0.067)(9.161,0.067)(9.396,0.068)(9.631,0.068)(9.866,0.068)(10.101,0.069)(10.336,0.069)(10.570,0.069)(10.805,0.070)(11.040,0.070)(11.275,0.070)(11.510,0.070)(11.745,0.070)(11.980,0.070)(12.215,0.070)(12.450,0.070)(12.685,0.069)(12.919,0.069)(13.154,0.069)(13.389,0.069)(13.624,0.069)(13.859,0.069)(14.094,0.069)(14.329,0.069)(14.564,0.069)(14.799,0.069)(15.034,0.069)(15.268,0.069)(15.503,0.069)(15.738,0.069)(15.973,0.069)(16.208,0.069)(16.443,0.069)(16.678,0.069)(16.913,0.069)(17.148,0.069)(17.383,0.069)(17.617,0.069)(17.852,0.069)(18.087,0.069)(18.322,0.069)(18.557,0.069)(18.792,0.069)(19.027,0.069)(19.262,0.069)(19.497,0.069)(19.732,0.069)(19.966,0.069)(20.201,0.069)(20.436,0.069)(20.671,0.069)(20.906,0.069)(21.141,0.069)(21.376,0.069)(21.611,0.069)(21.846,0.069)(22.081,0.069)(22.315,0.069)(22.550,0.069)(22.785,0.069)(23.020,0.069)(23.255,0.069)(23.490,0.069)(23.725,0.069)(23.960,0.069)(24.195,0.069)(24.430,0.069)(24.664,0.069)(24.899,0.069)(25.134,0.069)(25.369,0.069)(25.604,0.069)(25.839,0.069)(26.074,0.069)(26.309,0.069)(26.544,0.069)(26.779,0.069)(27.013,0.069)(27.248,0.069)(27.483,0.069)(27.718,0.069)(27.953,0.069)(28.188,0.069)(28.423,0.069)(28.658,0.069)(28.893,0.069)(29.128,0.069)(29.362,0.069)(29.597,0.069)(29.832,0.069)(30.067,0.069)(30.302,0.069)(30.537,0.069)(30.772,0.069)(31.007,0.069)(31.242,0.069)(31.477,0.069)(31.711,0.069)(31.946,0.069)(32.181,0.069)(32.416,0.069)(32.651,0.069)(32.886,0.069)(33.121,0.069)(33.356,0.069)(33.591,0.069)(33.826,0.069)(34.060,0.069)(34.295,0.069)(34.530,0.069)(34.765,0.069)(35.000,0.069)};
\draw[color=n6,fill] (36.000,0.069) circle(1pt);

%% file: graphs/c.tex
\draw[color=n2] plot coordinates{(0.7,6.075896) (0.736364,5.480866) (0.772727,4.980193) (0.809091,4.554768) (0.845455,4.190101) (0.881818,3.875042) (0.918182,3.600897) (0.954545,3.360808) (0.990909,3.149300) (1.02727,2.961967) (1.06364,2.795220) (1.1,2.646122) (1.13636,2.512240) (1.17273,2.391553) (1.20909,2.282365) (1.24545,2.183246) (1.28182,2.092983) (1.31818,2.010543) (1.35455,1.935039) (1.39091,1.865706) (1.42727,1.801885) (1.46364,1.743002) (1.5,1.688556) (1.53636,1.638108) (1.57273,1.591274) (1.60909,1.547713) (1.64545,1.507126) (1.68182,1.469245) (1.71818,1.433834) (1.75455,1.400682) (1.79091,1.369598) (1.82727,1.340413) (1.86364,1.312975) (1.9,1.287147) (1.93636,1.262803) (1.97273,1.239831) (2.00909,1.218131) (2.04545,1.197609) (2.08182,1.178180) (2.11818,1.159769) (2.15455,1.142305) (2.19091,1.125724) (2.22727,1.109967) (2.26364,1.094980) (2.3,1.080713) (2.33636,1.067122) (2.37273,1.054163) (2.40909,1.041798) (2.44545,1.029991) (2.48182,1.018708) (2.51818,1.007920) (2.55455,0.997598) (2.59091,0.987714) (2.62727,0.978245) (2.66364,0.969167) (2.7,0.960459) (2.73636,0.952101) (2.77273,0.944075) (2.80909,0.936364) (2.84545,0.928950) (2.88182,0.921819) (2.91818,0.914956) (2.95455,0.908349) (2.99091,0.901984) (3.02727,0.895850) (3.06364,0.889936) (3.1,0.884231) (3.13636,0.878726) (3.17273,0.873412) (3.20909,0.868279) (3.24545,0.863319) (3.28182,0.858525) (3.31818,0.853890) (3.35455,0.849406) (3.39091,0.845067) (3.42727,0.840866) (3.46364,0.836799) (3.5,0.832859) (3.53636,0.829041) (3.57273,0.825340) (3.60909,0.821751) (3.64545,0.818270) (3.68182,0.814893) (3.71818,0.811615) (3.75455,0.808433) (3.79091,0.805342) (3.82727,0.802340) (3.86364,0.799423) (3.9,0.796588) (3.93636,0.793831) (3.97273,0.791151) (4.00909,0.788543) (4.04545,0.786006) (4.08182,0.783537) (4.11818,0.781134) (4.15455,0.778794) (4.19091,0.776515) (4.22727,0.774295) (4.26364,0.772131) (4.3,0.770023)};
\draw[color=n3] plot coordinates{(0.7,-4.158979) (0.736364,-2.949789) (0.772727,-2.036436) (0.809091,-1.340453) (0.845455,-0.806142) (0.881818,-0.393395) (0.918182,-0.072955) (0.954545,0.176775) (0.990909,0.371907) (1.02727,0.524583) (1.06364,0.644034) (1.1,0.737337) (1.13636,0.809960) (1.17273,0.866155) (1.20909,0.909252) (1.24545,0.941874) (1.28182,0.966097) (1.31818,0.983574) (1.35455,0.995626) (1.39091,1.003313) (1.42727,1.007490) (1.46364,1.008847) (1.5,1.007944) (1.53636,1.005237) (1.57273,1.001097) (1.60909,0.995826) (1.64545,0.989673) (1.68182,0.982839) (1.71818,0.975492) (1.75455,0.967766) (1.79091,0.959773) (1.82727,0.951604) (1.86364,0.943334) (1.9,0.935023) (1.93636,0.926720) (1.97273,0.918466) (2.00909,0.910292) (2.04545,0.902224) (2.08182,0.894282) (2.11818,0.886481) (2.15455,0.878835) (2.19091,0.871351) (2.22727,0.864036) (2.26364,0.856895) (2.3,0.849930) (2.33636,0.843142) (2.37273,0.836531) (2.40909,0.830097) (2.44545,0.823838) (2.48182,0.817751) (2.51818,0.811833) (2.55455,0.806082) (2.59091,0.800493) (2.62727,0.795064) (2.66364,0.789790) (2.7,0.784667) (2.73636,0.779691) (2.77273,0.774859) (2.80909,0.770165) (2.84545,0.765606) (2.88182,0.761177) (2.91818,0.756876) (2.95455,0.752697) (2.99091,0.748638) (3.02727,0.744693) (3.06364,0.740860) (3.1,0.737135) (3.13636,0.733514) (3.17273,0.729995) (3.20909,0.726573) (3.24545,0.723246) (3.28182,0.720010) (3.31818,0.716862) (3.35455,0.713800) (3.39091,0.710821) (3.42727,0.707922) (3.46364,0.705101) (3.5,0.702354) (3.53636,0.699680) (3.57273,0.697076) (3.60909,0.694540) (3.64545,0.692069) (3.68182,0.689662) (3.71818,0.687317) (3.75455,0.685031) (3.79091,0.682803) (3.82727,0.680631) (3.86364,0.678513) (3.9,0.676448) (3.93636,0.674433) (3.97273,0.672467) (4.00909,0.670550) (4.04545,0.668678) (4.08182,0.666851) (4.11818,0.665068) (4.15455,0.663327) (4.19091,0.661627) (4.22727,0.659966) (4.26364,0.658344) (4.3,0.656760)};
\draw[color=n3,fill] (1.040788,0.572441) circle(1.5pt);
\draw[color=n4] plot coordinates{(0.7,-22.220078) (0.736364,-17.769087) (0.772727,-14.318398) (0.809091,-11.612942) (0.845455,-9.470053) (0.881818,-7.756960) (0.918182,-6.375841) (0.954545,-5.253723) (0.990909,-4.335546) (1.02727,-3.579329) (1.06364,-2.952748) (1.1,-2.430691) (1.13636,-1.993483) (1.17273,-1.625591) (1.20909,-1.314659) (1.24545,-1.050793) (1.28182,-0.826019) (1.31818,-0.633872) (1.35455,-0.469079) (1.39091,-0.327318) (1.42727,-0.205029) (1.46364,-0.099265) (1.5,-0.007573) (1.53636,0.072092) (1.57273,0.141448) (1.60909,0.201939) (1.64545,0.254786) (1.68182,0.301023) (1.71818,0.341530) (1.75455,0.377058) (1.79091,0.408250) (1.82727,0.435657) (1.86364,0.459753) (1.9,0.480949) (1.93636,0.499599) (1.97273,0.516009) (2.00909,0.530446) (2.04545,0.543144) (2.08182,0.554304) (2.11818,0.564104) (2.15455,0.572701) (2.19091,0.580229) (2.22727,0.586810) (2.26364,0.592548) (2.3,0.597538) (2.33636,0.601861) (2.37273,0.605590) (2.40909,0.608790) (2.44545,0.611519) (2.48182,0.613826) (2.51818,0.615758) (2.55455,0.617355) (2.59091,0.618653) (2.62727,0.619684) (2.66364,0.620476) (2.7,0.621055) (2.73636,0.621445) (2.77273,0.621665) (2.80909,0.621734) (2.84545,0.621669) (2.88182,0.621485) (2.91818,0.621194) (2.95455,0.620809) (2.99091,0.620341) (3.02727,0.619799) (3.06364,0.619192) (3.1,0.618527) (3.13636,0.617813) (3.17273,0.617055) (3.20909,0.616259) (3.24545,0.615430) (3.28182,0.614573) (3.31818,0.613693) (3.35455,0.612792) (3.39091,0.611874) (3.42727,0.610943) (3.46364,0.610000) (3.5,0.609050) (3.53636,0.608093) (3.57273,0.607132) (3.60909,0.606169) (3.64545,0.605205) (3.68182,0.604242) (3.71818,0.603280) (3.75455,0.602322) (3.79091,0.601369) (3.82727,0.600420) (3.86364,0.599477) (3.9,0.598541) (3.93636,0.597613) (3.97273,0.596692) (4.00909,0.595779) (4.04545,0.594876) (4.08182,0.593982) (4.11818,0.593097) (4.15455,0.592222) (4.19091,0.591357) (4.22727,0.590502) (4.26364,0.589657) (4.3,0.588823)};
\draw[color=n4,fill] (1.626282,0.227809) circle(1.5pt);
\draw[color=n5] plot coordinates{(0.7,-53.483525) (0.736364,-43.381014) (0.772727,-35.509501) (0.809091,-29.305057) (0.845455,-24.363021) (0.881818,-20.388645) (0.918182,-17.164288) (0.954545,-14.527220) (0.990909,-12.354364) (1.02727,-10.551634) (1.06364,-9.046403) (1.1,-7.782090) (1.13636,-6.714240) (1.17273,-5.807649) (1.20909,-5.034229) (1.24545,-4.371418) (1.28182,-3.800969) (1.31818,-3.308036) (1.35455,-2.880473) (1.39091,-2.508283) (1.42727,-2.183201) (1.46364,-1.898355) (1.5,-1.648009) (1.53636,-1.427352) (1.57273,-1.232333) (1.60909,-1.059525) (1.64545,-0.906021) (1.68182,-0.769345) (1.71818,-0.647379) (1.75455,-0.538308) (1.79091,-0.440570) (1.82727,-0.352818) (1.86364,-0.273883) (1.9,-0.202756) (1.93636,-0.138554) (1.97273,-0.080509) (2.00909,-0.027949) (2.04545,0.019714) (2.08182,0.063000) (2.11818,0.102363) (2.15455,0.138205) (2.19091,0.170883) (2.22727,0.200712) (2.26364,0.227971) (2.3,0.252909) (2.33636,0.275748) (2.37273,0.296685) (2.40909,0.315898) (2.44545,0.333545) (2.48182,0.349768) (2.51818,0.364694) (2.55455,0.378438) (2.59091,0.391104) (2.62727,0.402784) (2.66364,0.413563) (2.7,0.423517) (2.73636,0.432715) (2.77273,0.441220) (2.80909,0.449088) (2.84545,0.456370) (2.88182,0.463114) (2.91818,0.469363) (2.95455,0.475155) (2.99091,0.480526) (3.02727,0.485508) (3.06364,0.490131) (3.1,0.494423) (3.13636,0.498408) (3.17273,0.502109) (3.20909,0.505547) (3.24545,0.508742) (3.28182,0.511710) (3.31818,0.514469) (3.35455,0.517033) (3.39091,0.519417) (3.42727,0.521632) (3.46364,0.523691) (3.5,0.525605) (3.53636,0.527383) (3.57273,0.529036) (3.60909,0.530571) (3.64545,0.531996) (3.68182,0.533320) (3.71818,0.534548) (3.75455,0.535687) (3.79091,0.536744) (3.82727,0.537722) (3.86364,0.538629) (3.9,0.539468) (3.93636,0.540244) (3.97273,0.540960) (4.00909,0.541622) (4.04545,0.542231) (4.08182,0.542793) (4.11818,0.543308) (4.15455,0.543782) (4.19091,0.544216) (4.22727,0.544612) (4.26364,0.544974) (4.3,0.545303)};
\draw[color=n5,fill] (2.166044,0.148865) circle(1.5pt);
\draw[color=n6] plot coordinates{(0.7,-107.898283) (0.736364,-87.931937) (0.772727,-72.346502) (0.809091,-60.038248) (0.845455,-50.214573) (0.881818,-42.297700) (0.918182,-35.860646) (0.954545,-30.583887) (0.990909,-26.225532) (1.02727,-22.600505) (1.06364,-19.565794) (1.1,-17.009875) (1.13636,-14.845033) (1.17273,-13.001731) (1.20909,-11.424426) (1.24545,-10.068446) (1.28182,-8.897624) (1.31818,-7.882499) (1.35455,-6.998932) (1.39091,-6.227033) (1.42727,-5.550331) (1.46364,-4.955119) (1.5,-4.429933) (1.53636,-3.965148) (1.57273,-3.552643) (1.60909,-3.185541) (1.64545,-2.857999) (1.68182,-2.565028) (1.71818,-2.302358) (1.75455,-2.066321) (1.79091,-1.853755) (1.82727,-1.661928) (1.86364,-1.488470) (1.9,-1.331321) (1.93636,-1.188685) (1.97273,-1.058991) (2.00909,-0.940863) (2.04545,-0.833094) (2.08182,-0.734618) (2.11818,-0.644496) (2.15455,-0.561898) (2.19091,-0.486089) (2.22727,-0.416413) (2.26364,-0.352290) (2.3,-0.293201) (2.33636,-0.238683) (2.37273,-0.188321) (2.40909,-0.141745) (2.44545,-0.098621) (2.48182,-0.058650) (2.51818,-0.021561) (2.55455,0.012889) (2.59091,0.044920) (2.62727,0.074730) (2.66364,0.102500) (2.7,0.128394) (2.73636,0.152559) (2.77273,0.175131) (2.80909,0.196232) (2.84545,0.215975) (2.88182,0.234462) (2.91818,0.251785) (2.95455,0.268032) (2.99091,0.283279) (3.02727,0.297600) (3.06364,0.311058) (3.1,0.323716) (3.13636,0.335629) (3.17273,0.346847) (3.20909,0.357419) (3.24545,0.367387) (3.28182,0.376791) (3.31818,0.385670) (3.35455,0.394056) (3.39091,0.401983) (3.42727,0.409478) (3.46364,0.416570) (3.5,0.423284) (3.53636,0.429643) (3.57273,0.435669) (3.60909,0.441382) (3.64545,0.446801) (3.68182,0.451943) (3.71818,0.456826) (3.75455,0.461464) (3.79091,0.465871) (3.82727,0.470061) (3.86364,0.474046) (3.9,0.477838) (3.93636,0.481448) (3.97273,0.484885) (4.00909,0.488159) (4.04545,0.491279) (4.08182,0.494254) (4.11818,0.497091) (4.15455,0.499798) (4.19091,0.502381) (4.22727,0.504848) (4.26364,0.507203) (4.3,0.509453)};
\draw[color=n6,fill] (2.688589,0.120461) circle(1.5pt);
\draw[color=n7] plot coordinates{(0.7,-203.747569) (0.736364,-166.388878) (0.772727,-137.203351) (0.809091,-114.134912) (0.845455,-95.706560) (0.881818,-80.841230) (0.918182,-68.742665) (0.954545,-58.814760) (0.990909,-50.606067) (1.02727,-43.771022) (1.06364,-38.042482) (1.1,-33.212031) (1.13636,-29.115668) (1.17273,-25.623315) (1.20909,-22.631022) (1.24545,-20.055138) (1.28182,-17.827903) (1.31818,-15.894090) (1.35455,-14.208419) (1.39091,-12.733567) (1.42727,-11.438602) (1.46364,-10.297765) (1.5,-9.289504) (1.53636,-8.395708) (1.57273,-7.601088) (1.60909,-6.892689) (1.64545,-6.259492) (1.68182,-5.692088) (1.71818,-5.182413) (1.75455,-4.723536) (1.79091,-4.309481) (1.82727,-3.935074) (1.86364,-3.595828) (1.9,-3.287839) (1.93636,-3.007697) (1.97273,-2.752424) (2.00909,-2.519403) (2.04545,-2.306336) (2.08182,-2.111198) (2.11818,-1.932198) (2.15455,-1.767753) (2.19091,-1.616458) (2.22727,-1.477063) (2.26364,-1.348457) (2.3,-1.229646) (2.33636,-1.119743) (2.37273,-1.017953) (2.40909,-0.923563) (2.44545,-0.835933) (2.48182,-0.754485) (2.51818,-0.678701) (2.55455,-0.608109) (2.59091,-0.542286) (2.62727,-0.480847) (2.66364,-0.423443) (2.7,-0.369759) (2.73636,-0.319505) (2.77273,-0.272420) (2.80909,-0.228265) (2.84545,-0.186821) (2.88182,-0.147890) (2.91818,-0.111289) (2.95455,-0.076850) (2.99091,-0.044422) (3.02727,-0.013862) (3.06364,0.014958) (3.1,0.042156) (3.13636,0.067843) (3.17273,0.092120) (3.20909,0.115078) (3.24545,0.136805) (3.28182,0.157380) (3.31818,0.176876) (3.35455,0.195361) (3.39091,0.212899) (3.42727,0.229548) (3.46364,0.245361) (3.5,0.260391) (3.53636,0.274683) (3.57273,0.288281) (3.60909,0.301225) (3.64545,0.313555) (3.68182,0.325305) (3.71818,0.336507) (3.75455,0.347193) (3.79091,0.357392) (3.82727,0.367130) (3.86364,0.376433) (3.9,0.385324) (3.93636,0.393825) (3.97273,0.401957) (4.00909,0.409739) (4.04545,0.417190) (4.08182,0.424327) (4.11818,0.431165) (4.15455,0.437720) (4.19091,0.444007) (4.22727,0.450037) (4.26364,0.455825) (4.3,0.461382)};
\draw[color=n7,fill] (3.202802,0.111198) circle(1.5pt);